\RequirePackage{fix-cm}

\documentclass[smallextended]{svjour3}       % onecolumn (second format)

\usepackage{amsthm}
\smartqed  % flush right qed marks, e.g. at end of proof
\usepackage[linesnumbered,ruled,vlined]{algorithm2e}
\usepackage{graphicx}
\usepackage{relsize}
\usepackage{amsmath}
\usepackage{amsfonts}
\usepackage{bm}
\usepackage{nccmath}
\usepackage[margin=1in]{geometry}
\numberwithin{equation}{section}
\usepackage{appendix}
\usepackage{xcolor}
\usepackage{enumitem}
\usepackage{booktabs}
\usepackage{caption}
\usepackage{subcaption}
\usepackage{multirow}
\usepackage{multicol}
\DeclareMathOperator*{\argmin}{argmin}
\SetKwInput{KwInput}{Input}                % Set the Input

\SetKwInput{KwInitialization}{Initialization} 
\SetKwInput{KwOutput}{Output}              % set the Output
\SetKwInput{KwCondition}{Condition}

\begin{document}

\title{Sparsity-Guided Multi-Parameter Selection in $\ell_1$-Regularized Models via a Fixed-Point Proximity Approach\thanks{%Grants or other notes
%about the article that should go on the front page should be
%placed here. General acknowledgments should be placed at the end of the article.
Qianru Liu is supported in part by the Natural Science Foundation of China under grant 12571562, and the Doctor Foundation of Henan University
of Technology, China (No.2023BS061), and the Innovative
Funds Plan of Henan University of Technology (No.2021ZKCJ11).
Rui Wang is supported in part by the Natural Science Foundation of China under grants 12171202 and 12571562. Yuesheng Xu is supported in part by the US National Science Foundation under grant DMS-2208386, and the US National Institutes of Health under grant R21CA263876. The computer codes that generate the numerical results presented in this paper can be found in the following website: https://github.com/qrliuMath/Multi-ParameterChoicesL1.}
}
%\subtitle{Do you have a subtitle?\\ If so, write it here}

%\titlerunning{Short form of title}        % if too long for running head

\author{Qianru Liu \and Rui Wang \and Yuesheng Xu %etc.
}

%\authorrunning{Short form of author list} % if too long for running head

\institute{Qianru Liu \at
             School of Mathematics and Statistics, Henan University of Technology, Zhengzhou, 450001, People’s Republic of China 
             %\\  Tel.: +123-45-678910
             %\\ Fax: +123-45-678910
             \\ \email{liuqr23@haut.edu.cn}   %  \\
             %   \emph{Present address:} of F. Author  %  if needed
           \and
           Rui Wang \at School of Mathematics, Jilin University, Changchun 130012, People’s Republic of China\\
           \email{rwang11@jlu.edu.cn} \\ All correspondence should be sent to this author
           \and Yuesheng Xu \at Department of Mathematics and Statistics, Old Dominion University, Norfolk, VA 23529, United States of America\\ 
               \email{y1xu@odu.edu} 
}

\date{Received: date / Accepted: date}
% The correct dates will be entered by the editor

\maketitle

\begin{abstract}
We study a regularization framework that combines a convex fidelity term with multiple $\ell_1$-based regularizers, each linked to a distinct linear transform. This multi-penalty model enhances flexibility in promoting structured sparsity. We analyze how the choice of regularization parameters governs the sparsity of solutions under the given transforms and derive a precise relationship between the parameters and resulting sparsity patterns. This  insight enables the development of an iterative strategy for selecting parameters to achieve prescribed sparsity levels. 
A key computational challenge arises in practice: effective parameter tuning requires simultaneous access to the regularized solution and two auxiliary vectors derived from the sparsity analysis. 
To address this, we propose a fixed-point proximity algorithm that jointly computes all three vectors. Together with our theoretical characterization, this algorithm forms the basis of a practical multi-parameter selection scheme. Numerical experiments demonstrate that the proposed method reliably produces solutions with desired sparsity patterns and strong approximation accuracy.

% \end{abstract}
%\begin{abstract}
% Insert your abstract here. Include keywords, PACS and mathematical
% subject classification numbers as needed.
%We consider a regularization problem with an objective function consisting of a convex fidelity term and multiple $\ell_1$-based regularization terms, each coupled with a linear transform. Such a regularization model has been empirically demonstrated to enhance the flexibility of promoting sparsity in regularized solutions. We investigate how the selection of multiple regularization parameters influences the sparsity of regularized solutions. Specifically, we characterize the relationship between these parameters and the sparsity of solutions under transform matrices, enabling the development of an iterative scheme for selecting parameters that achieve prescribed sparsity levels. Special attention is given to scenarios where the fidelity term is non-differentiable, or the transform matrix lacks full row rank. In such cases, the regularized solution, along with two auxiliary vectors arising in the sparsity characterization, are essential components of the multi-parameter selection strategy. To address this, we propose a fixed-point proximity algorithm that simultaneously determines these three vectors. This algorithm, combined with our sparsity characterization, forms the basis of a practical multi-parameter selection strategy. Numerical experiments demonstrate the effectiveness of the proposed approach, yielding regularized solutions with both predetermined sparsity levels and satisfactory approximation accuracy.

% \PACS{PACS code1 \and PACS code2 \and more}
\end{abstract}

\section{Introduction}\label{intro}

Sparse regularization using the $\ell_1$-norm has become a foundational technique in modern data analysis, signal processing, and machine learning. It is particularly effective in high-dimensional and ill-posed settings, where promoting sparsity is essential for achieving interpretable, robust, and computationally efficient solutions. Traditionally, $\ell_1$-regularization has been applied using a single penalty term, as in the LASSO model \cite{Chen2001,tibshirani1996regression,TibshiraniThe2011}. However, in many practical problems, especially those involving structured signals or multiple modalities, a single regularization parameter is often insufficient to capture the underlying sparsity patterns.

To overcome this limitation, extensive research has focused multi-parameter regularization frameworks that incorporate multiple $\ell_1$-based penalties, each linked to a distinct linear transform \cite{Afonso2010,Friedman2007,Rapaport2008,Selesnick2014Simultaneous,Shen2024Sparse,Tibshirani2005,Tolosi2011,Yu2015}. This multi-transform setting allows different structural features of the signal to be promoted in parallel. For example, one transform may promote sparsity in the time domain while another targets frequency components. By assigning separate regularization parameters to each transform, the model gains enhanced flexibility to adapt to complex data structures, while still mitigating the ill-posedness of the underlying problem.

%{\color{blue}Multi-parameter regularization has gained significant attention in fields such as statistical learning \cite{Friedman2007,Tibshirani2005}, image and signal processing \cite{Afonso2010,Lu2013,Selesnick2014Simultaneous}, and inverse problems \cite{Belkin2006,Chen2008,wang2013}. However, most existing work centers on models with an $\ell_2$-regularizer. }{\color{red} $\rightarrow$ Delete} 
In $\ell_1$-regularized models, selecting  regularization parameters to guide sparsity remains a significant challenge. While empirical evidence indicates that parameter choices influence solution sparsity, practical methods for tuning multiple parameters to achieve desired sparsity levels are still lacking. Most theoretical results focus on the single-parameter case, relating regularization strength to sparsity \cite{bach2011optimization,koh2007interior,shi2011concentration,tibshirani2012degrees,zou2007degrees}, but often assume knowledge of the true solution and offer limited practical utility. Our recent work \cite{Liu2023parameter} introduced a principled method for single-parameter selection. Extending this to the multi-parameter setting is challenging due to the complex and interdependent influence of multiple parameters on both sparsity and solution structure.

Building on the single-parameter selection strategy from \cite{Liu2023parameter}, this paper develops a practical approach for selecting multiple $\ell_1$-regularization parameters to achieve a prescribed level of structured sparsity. We begin by establishing a theoretical foundation for how each regularization parameter influences sparsity under associated transform matrices. Unlike the single-parameter setting, the multi-parameter framework enables independent control of sparsity across different transforms.
Using convex analysis, we characterize the relationship between the regularization parameters and the solution's sparsity under each transform. In the special case where the transform matrices are scaled identities and the fidelity term is block separable, this leads to an explicit parameter selection rule that guarantees desired block-wise sparsity in the solution.
For the general case, where such direct rules are not available, we propose an iterative algorithm to select the parameters. This scheme accounts for the interactions among parameters and ensures the solution achieves the targeted sparsity levels. By explicitly handling parameter interactions, the proposed approach enhances both flexibility and performance in achieving the desired sparsity.

The iterative scheme faces a key computational challenge at each step: in addition to computing the regularized solution, it must also evaluate two auxiliary vectors. This becomes particularly demanding when the fidelity term is non-differentiable and the block matrix formed from the transform matrices lacks full row rank—conditions under which explicit formulas for the auxiliary vectors are unavailable. To address this, inspired by the fixed-point framework in \cite{li2015multi}, we formulate all three components—the solution and the two auxiliary vectors—as a system of fixed-point equations involving proximity operators. We then develop a fixed-point proximity algorithm that computes them simultaneously.
Importantly, this algorithm is not limited to the aforementioned difficult case; it also applies to broader settings, including problems with differentiable fidelity terms or full-rank transform matrices. We provide a rigorous convergence analysis of the algorithm.
Building on this foundation, we integrate the fixed-point algorithm with our sparsity characterization to propose a robust multi-parameter selection strategy that ensures the regularized solution achieves a prescribed sparsity level.

In our recent works \cite{Liu2023parameter,Shen2024Sparse}, we studied parameter selection strategies for $\ell_1$-based sparse regularization. This paper makes two key contributions that extend those efforts:

\begin{itemize}
\item[$\bullet$] We propose a strategy for selecting multiple regularization parameters in models with a convex fidelity term and several $\ell_1$-type regularizers, each composed with a distinct linear transform. Unlike \cite{Liu2023parameter}, which focused on a single-parameter setting, our algorithm handles multiple parameters and explicitly accounts for their interactions. While \cite{Shen2024Sparse} also considered multi-parameter selection, its method was limited by the complexity of nonconvex optimization involving deep neural networks and large datasets. In contrast, our convex framework enables a more principled and refined selection algorithm.

\item[$\bullet$] Our method removes two key limitations present in \cite{Liu2023parameter} and \cite{Shen2024Sparse}, which assume differentiable fidelity terms and full row rank transform matrices. We develop a generalized iterative scheme that remains effective when the fidelity term is nondifferentiable or the block matrix formed by the transforms is rank-deficient, thereby substantially broadening the applicability of multi-parameter selection.
\end{itemize}

This paper is organized into seven sections and an appendix. Section \ref{Multi} introduces the general multi-parameter regularization model and presents several motivating examples. In Section \ref{Sparsity}, we characterize how each regularization parameter influences the sparsity of the solution under its corresponding transform, with a specialization to the case of degenerate identity transforms. Building on these insights, Section \ref{Itera} proposes an iterative scheme to jointly compute the regularized solution and select parameters that achieve prescribed sparsity levels.
To address the computational challenges of this scheme, Section \ref{fixed} develops a fixed-point proximity algorithm and establishes its convergence. Section \ref{Numer} presents six numerical experiments. The first four demonstrate the effectiveness of the algorithm in achieving the prescribed sparsity across different models, while the remaining two provide numerical evidence supporting the theoretical convergence assumptions and examine the sensitivity of the method to key hyperparameters. Section \ref{Concl} concludes the paper, and Appendix A provides closed-form expressions for the proximity operators used in the algorithm.

\section{Multi-parameter regularization with the $\ell_1$-norm}\label{Multi}
In this section, we present the multi-parameter regularization problem studied in this paper. We begin by formulating the general problem and then review several optimization models of practical relevance, demonstrating how they can be expressed within this unified framework.

We start with describing the multi-parameter regularization problem.  For each $m\in\mathbb{N}$, let  $\mathbb{N}_m:=\{1, 2, \ldots,m\}$ and set $\mathbb{N}_0:=\emptyset$. For each vector $\mathbf{x}:=[x_j:j\in\mathbb{N}_m]\in\mathbb{R}^m$, we define its $\ell_1$-norm by $\|\mathbf{x}\|_1:=\sum_{j\in \mathbb{N}_m}|x_j|$. Let $n,d\in\mathbb{N}$. Suppose that $\bm{\psi}:\mathbb{R}^n\to\mathbb{R}_+:=[0,+\infty)$ is a convex function and for each $j\in\mathbb{N}_d$, $m_j\in\mathbb{N}$ and  $\mathbf{B}_{j}$ is an $m_j\times n$ real matrix. We consider the multi-parameter regularization problem 
\begin{equation}\label{optimization_problem_under_Bj}
\min
\left\{\bm{\psi}(\mathbf{u})+\sum_{j\in\mathbb{N}_d}\lambda_j\|\mathbf{B}_{j}\mathbf{u}\|_{1}:\mathbf{u}\in\mathbb{R}^n\right\},
\end{equation}
where $\lambda_j>0,$ $j\in\mathbb{N}_d,$ is a sequence of positive regularization parameters. 

The multi-parameter regularization problem \eqref{optimization_problem_under_Bj} appears in many application areas. Below, we present several examples of it. In imaging reconstruction problems, a combination of multiple regularizers
was used to encourage their solution to simultaneously exhibit the characteristics enforced by each of them. For example, a combination of frame-based synthesis and analysis $\ell_1$-norm regularizers was proposed in \cite{Afonso2010} for an imaging deblurring problem. Specifically, for $m,n,s\in\mathbb{N}$, we assume that $\mathbf{A}\in\mathbb{R}^{m\times s}$ represents a periodic convolution, $\mathbf{W}\in \mathbb{R}^{s\times n}$ is a synthesis operator whose columns contain the elements of a frame, and $\mathbf{P}\in \mathbb{R}^{m\times s}$ is an analysis operator of a tight Parseval frame satisfying $\mathbf{P}^{\top}\mathbf{P}=\mathbf{I}_s$, where $\mathbf{I}_s$ denotes the identity matrix of order $s$. Let $\mathbf{y}\in\mathbb{R}^m$ be  observed data. The regularization problem combining the synthesis and analysis $\ell_1$-norm regularizers has the form
\begin{equation}\label{CSA}
\min\left\{\frac{1}{2}\|\mathbf{AW}\mathbf{u}-\mathbf{y}\|_2^2+\lambda_1\|\mathbf{u}\|_1+\lambda_2\|\mathbf{PW}\mathbf{u}\|_1:\mathbf{u}\in\mathbb{R}^n\right\}.
\end{equation}
Clearly, problem \eqref{CSA} may be identified as a special case of \eqref{optimization_problem_under_Bj} with $
\bm{\psi}(\mathbf{u}):=\frac{1}{2}\|\mathbf{AW}\mathbf{u}-\mathbf{y}\|_2^2$ for $\  \mathbf{u}\in\mathbb{R}^n$, $\mathbf{B}_{1}:=\mathbf{I}_n$ and $\mathbf{B}_{2}:=\mathbf{PW}.$

As a generalization of the lasso regularized model \cite{tibshirani1996regression}, the fused lasso regularized model was proposed in \cite{Tibshirani2005} for problems with features that can be ordered in some meaningful way. Let $p,n\in\mathbb{N}$. Suppose that a prediction problem with $p$ cases has outcomes $y_i$, $i\in\mathbb{N}_p$ and features $x_{ij}$, $i\in\mathbb{N}_p$, $j\in\mathbb{N}_n$. Let $\mathbf{X}:=[x_{ij}:i\in\mathbb{N}_p,j\in\mathbb{N}_n]$ be the $p\times n$ matrix of features and $\mathbf{y}:=[y_i:i\in\mathbb{N}_p]\in\mathbb{R}^p$,  $\mathbf{u}\in\mathbb{R}^n$ be the vectors of outcomes and coefficients, respectively. The fused lasso regularized model is formulated as 
\begin{equation}\label{fused_lasso}
\min\left\{\frac{1}{2}\|\mathbf{X}\mathbf{u}-\mathbf{y}\|_2^2+\lambda_1\|\mathbf{u}\|_1+\lambda_2\|\mathbf{Du}\|_1:\mathbf{u}\in\mathbb{R}^n\right\},
\end{equation}
where $\mathbf{D}:=[d_{ij}:i\in \mathbb{N}_{n-1},j\in\mathbb{N}_n]$ is the $(n-1)\times n$ first order difference matrix with $d_{ii}=-1,$ $d_{i,i+1}=1$ for $i\in\mathbb{N}_{n-1}$ and $0$ otherwise. By introducing 
$\bm{\psi}(\mathbf{u}):=\frac{1}{2}\|\mathbf{X}\mathbf{u}-\mathbf{y}\|_2^2$ for $\mathbf{u}\in\mathbb{R}^n$, $\mathbf{B}_{1}:=\mathbf{I}_{n}$ and $\mathbf{B}_{2}:=\mathbf{D},
$
the fused lasso model \eqref{fused_lasso} can be rewritten in the form of \eqref{optimization_problem_under_Bj}. By penalizing the
$\ell_1$-norm of both the coefficients and their successive differences, the fused lasso regularized model encourages the sparsity
of the coefficients and also the sparsity of their differences. As a special case, the fused lasso signal approximation \cite{Friedman2007} has the form \eqref{fused_lasso} with the feature matrix $\mathbf{X}$ being the identity matrix.

Filtering noisy data was considered in \cite{Selesnick2014Simultaneous} for the case where the underlying signal comprises a low-frequency component and  a sparse or sparse-derivative component. Specifically, assume that the noisy data $y(t)$ can be modeled as 
\begin{equation}\label{noisy-data-form}
y(t)=f(t)+u(t)+\eta(t),
\ t \in\mathbb{N}_n,
\end{equation}
where $f$ is a low-pass signal, $u$ is a sparse and sparse-derivative signal and $\eta$ is stationary white Gaussian noise. Given noisy data of the form \eqref{noisy-data-form}, one seeks the estimate of $f$ and $u$ individually. For this purpose, we first solve the compound sparse denoising problem
\begin{equation}\label{CSD}
\min\left\{\frac{1}{2}\|\mathbf{\mathbf{H}}(\mathbf{y}-\mathbf{u})\|_2^2+\lambda_1\|\mathbf{u}\|_1+\lambda_2\|\mathbf{D}\mathbf{u}\|_1:\mathbf{u}\in\mathbb{R}^{n}\right\}  
\end{equation}
to obtain the estimate $\mathbf{u}^*$ of $u$. Here, $\mathbf{y}:=[y(t):t\in\mathbb{N}_n]$ and  $\mathbf{H}$ is a high-pass filter matrix. 
%with the form $\mathbf{H}:=\mathbf{A}^{-1}\mathbf{C}$, where $\mathbf{A}$ and $\mathbf{C}$ are banded matrices. 
We then get the estimate $\mathbf{f}^*$ of $f$ as $\mathbf{f}^*:=(\tilde{\mathbf{I}}-\mathbf{H})(\mathbf{y}-\mathbf{u}^*)$. It is clear that the compound sparse denoising model \eqref{CSD} has the form \eqref{optimization_problem_under_Bj} with 
\begin{equation}\label{fidelity-CSD}
\bm{\psi}(\mathbf{u}):=\frac{1}{2}\|\mathbf{\mathbf{H}}(\mathbf{y}-\mathbf{u})\|_2^2,\ \  \mbox{for all}\  \ \mathbf{u}\in\mathbb{R}^n,
\end{equation} 
and $\mathbf{B}_{1}:=\mathbf{I}_n$,  $\mathbf{B}_{2}:=\mathbf{D}$.

Using a technique similar to that used in the fused lasso regularized model, the fused SVM was proposed for classification of array-based comparative genomic hybridization
(arrayCGH) data \cite{Rapaport2008,Tolosi2011}.  Given training data $\{(\mathbf{x}_j,y_j):j\in\mathbb{N}_p\}$ composed of sample points $\{\mathbf{x}_j:j\in\mathbb{N}_p\}\subset \mathbb{R}^n$ and labels $\{y_j:j\in\mathbb{N}_p\}\subset\{-1,1\}$. The aim of binary classification is to find a decision function $f(\mathbf{x}):=\mathbf{u}^{\top}\mathbf{x}$, $\mathbf{x}\in\mathbb{R}^n$ predicting the class $y_j=-1$ or $y_j=1$. The class prediction for a profile
$\mathbf{x}$ is then $1$ if $f(\mathbf{x})\geq 0$ and $-1$ otherwise. The fused SVM based on the hinge loss function has the form
\begin{equation}\label{fused-SVM}
\min\left\{\sum_{j\in\mathbb{N}_p}\mathrm{max}(0,1-y_j\mathbf{u}^{\top}\mathbf{x}_j)+\lambda_1\|\mathbf{u}\|_1+\lambda_2\|\mathbf{D}\mathbf{u}\|_1:\mathbf{u}\in\mathbb{R}^n\right\}.  
\end{equation}
We rewrite model \eqref{fused-SVM} in the form \eqref{optimization_problem_under_Bj} as follows. We define the matrix $\mathbf{X}:=[\mathbf{x}_j:j\in\mathbb{N}_p]^{\top}$, the matrix $\mathbf{Y}:=\mathrm{diag}(y_j:j\in\mathbb{N}_p)$ and the function $\bm{\phi}(\mathbf{z}):=\sum_{j\in\mathbb{N}_p}\mathrm{max}\{0,1-z_j\}$ for $\mathbf{z}:=[z_j:j\in\mathbb{N}_p]\in\mathbb{R}^p$. Then by introducing 
$\bm{\psi}(\mathbf{u}):=\bm{\phi}(\mathbf{YXu})$ for $\mathbf{u}\in\mathbb{R}^{n}$,  $\mathbf{B}_{1}:=\mathbf{I}_n$ and $\mathbf{B}_{2}:=\mathbf{D}$,
the fused SVM model \eqref{fused-SVM} can be represented in the form \eqref{optimization_problem_under_Bj}.

In many imaging and signal processing applications, it is advantageous to apply non-uniform regularization across different spatial locations or signal components in order to better preserve salient features such as edges. This consideration naturally leads to models with spatially varying, or component-wise, regularization parameters. A standard approach is to introduce a diagonal matrix of regularization weights, resulting in formulations of the form
\begin{equation}\label{spatial-model}
\min \left\{ \bm{\psi}(\mathbf{u}) + \|\mathbf{\Lambda} \mathbf{u}\|_1:\mathbf{u}\in\mathbb{R}^n \right\} \quad \text{or} \quad \min \left\{ \bm{\psi}(\mathbf{u}) + \|\mathbf{\Lambda}  \mathbf{D} \mathbf{u}\|_1:\mathbf{u}\in\mathbb{R}^n \right\},
\end{equation} 
where $\mathbf{\Lambda} := \operatorname{diag}(\lambda_1, \ldots, \lambda_q)$, with $q = n$ in the first model and $q = n-1$ for the second. 

The models in \eqref{spatial-model} assign distinct regularization parameters $\lambda_i$ to individual components $u_i$ or each differences $(\mathbf{D}\mathbf{u})_i$, enabling fine-grained control over local smoothing. Such parameterization is particularly important in image restoration, where homogeneous regions and edges typically require different levels of regularization. These spatially adaptive models can be viewed as special cases of \eqref{optimization_problem_under_Bj}. Specifically, the first model corresponds to \eqref{optimization_problem_under_Bj} with $d = n$, where for each $j \in \mathbb{N}_n$, the matrix $\mathbf{B}_j$ is the $1 \times n$ row vector whose $j$-th entry is $1$ and whose remaining entries are zero. Similarly, the second model corresponds to \eqref{optimization_problem_under_Bj} with $d = n-1$, where for each $j \in \mathbb{N}_{n-1}$, $\mathbf{B}_j$ is the $j$-th row of the difference matrix $\mathbf{D}$. 

To facilitate the technical developments in subsequent sections, Table \ref{tab:key_notations} summarizes the key notations used throughout the paper; each notation is also defined at its first appearance in the text.
\begin{table}[h!]
\centering
\caption{Summary of key notations.}
\label{tab:key_notations}
\begin{tabular}{@{}p{3cm}p{11cm}@{}}
\toprule
\textbf{Symbol} & \textbf{Description} \\
\midrule
$\Omega_{s,l}$ & The set of all vectors in $\mathbb{R}^s$ with exactly $l$ nonzero components, that is, vectors of sparsity level $l$, for $l \in \mathbb{Z}_{s+1}$. \\ \addlinespace[0.3em]
$p_j$ & Cumulative dimension, defined recursively by $p_0 := 0$ and $p_j := \sum_{i \in \mathbb{N}_j} m_i$, for $j \in \mathbb{N}_d$. \\ \addlinespace[0.3em]
$\mathbf{z}_j$ & The $j$-th subvector of a vector $\mathbf{z} \in \mathbb{R}^{p_d}$, where $\mathbf{z}$ is partitioned according to $(m_1, \ldots, m_d)$. Specifically, $\mathbf{z}_j := [z_{p_{j-1}+i} : i \in \mathbb{N}_{m_j}] \in \mathbb{R}^{m_j}$ for $j \in \mathbb{N}_d$. \\ \addlinespace[0.3em]
$\mathbf{B}$ & The column-block matrix obtained by concatenating all transformation matrices, defined as $\mathbf{B}:=[\mathbf{B}_{j}:j\in\mathbb{N}_d]\in\mathbb{R}^{p_d\times n}$. \\ \addlinespace[0.3em]
$r$ & The rank of matrix $\mathbf{B}$, satisfying $0 < r \leq \min\{p_d, n\}$. \\ \addlinespace[0.3em]
$\mathbf{B}'$ & The matrix in $\mathbb{R}^{n \times (p_d + n - r)}$ that reconstructs the original variable $\mathbf{u}$ from the concatenated vector $\begin{bmatrix}
\mathbf{z}\\
\mathbf{v}\end{bmatrix}$ via $\mathbf{u}=\mathbf{B}'{\scalebox{0.9}{$\begin{bmatrix}
\mathbf{z}\\
\mathbf{v}\end{bmatrix}$}}$. This matrix is constructed using the singular value decomposition of $\mathbf{B}$.
\\ \addlinespace[0.3em]
$\mathbf{I}^{'}_{j}$ & A degenerated identity (block-selection) matrix in $\mathbb{R}^{m_j \times (p_d + n - r)}$, defined by $\mathbf{I}^{'}_{j}:=[\mathbf{0}_{m_j\times p_{j-1}}\  \mathbf{I}_{m_j}\ \mathbf{0}_{m_j\times (p_d-p_j)} \  \mathbf{0}_{m_j\times (n-r)}]$.\\ \addlinespace[0.3em]
$\mathcal{R}(\mathbf{M})$ &	The range of a matrix $\mathbf{M}$.\\ \addlinespace[0.3em]
$\mathbb{M}$ & The constraint set in the reformulated problem, defined as $\mathbb{M} := \mathcal{R}(\mathbf{B}) \times \mathbb{R}^{n-r} \subseteq \mathbb{R}^{p_d + n - r}$. \\ \addlinespace[0.3em]
$\mathcal{N}(\mathbf{M})$&	The null space of the matrix $\mathbf{M}$.\\ \addlinespace[0.3em]
$\mathbf{M}_{(i)}$&	The $i$-th column of the matrix $\mathbf{M}$.\\ \addlinespace[0.3em]
$\partial f(\mathbf{x})$ & The subdifferential of the convex function $f$ at $\mathbf{x}$.
\\ \addlinespace[0.3em]
$\mathcal{S}_{s,q}$ & A partition of $\mathbb{N}_s$ into $q$ disjoint, nonempty subsets, denoted by $\mathcal{S}_{s,q} = \{S_{s,1}, S_{s,2}, \ldots, S_{s,q}\}$. \\ \addlinespace[0.3em]
$S_{s,k}$	&The $k$-th subset in the partition $\mathcal{S}_{s,q}$, with cardinality $s_k$.\\ \addlinespace[0.3em]
$\mathbf{w}_{S_{s,k}}$ &	The subvector of $\mathbf{w} \in \mathbb{R}^s$ corresponding to the indices in $S_{s,k}$.\\ \addlinespace[0.3em]
$\mathcal{S}_{s,q}$-block separable	& A function  $\bm{\phi}:\mathbb{R}^s\rightarrow\mathbb{R}$ is called $\mathcal{S}_{s,q}$-block separable if there exist functions $\bm{\phi}_k:\mathbb{R}^{s_k}\rightarrow\mathbb{R}$, $k\in\mathbb{N}_q$ such that
$\bm{\phi}(\mathbf{w})=\sum\limits_{k\in\mathbb{N}_q}\bm{\phi}_{k}(\mathbf{w}_{S_{s,k}}),\ \mathbf{w}\in\mathbb{R}^s.$\\ \addlinespace[0.3em]
$\bm{\psi}_j$	&The $j$-th component function in the block-separable fidelity term $\bm{\psi}$.\\ \addlinespace[0.3em]
$\bm{\psi}_{j,k}$	& The $(j,k)$-th component function when $\bm{\psi}_j$ is further block-separable. \\ \addlinespace[0.3em]
$\mathbf{A}_{[j]}$ & The submatrix of $\mathbf{A}$ consisting of the columns indexed by the $j$-th subset of the partition $\mathcal{S}_{n,d}$. \\ \addlinespace[0.3em]
$\mathbf{A}_{[j,k]}$&The submatrix of $\mathbf{A}_{[j]}$ corresponding to its $k$-th block.\\ \addlinespace[0.3em]
$\gamma_{j,i}(\mathbf{u})$ & For $j \in \mathbb{N}_d$ and $i \in \mathbb{N}_{m_j}$, a quantity determining the sparsity pattern under the transform $\mathbf{B}_j$, defined by $\gamma_{j,i}(\mathbf{u}):=\left|(\mathbf{B}_{(p_{j-1}+i)}')^\top\mathbf{a}+b_{p_{j-1}+i}\right|$, where $\mathbf{a} \in \partial\bm{\psi}(\mathbf{u})$ and $\mathbf{b} \in \mathcal{N}(\mathbf{B}^\top)$. The order of $\gamma_{j,i}(\mathbf{u})$ governs the selection of $\lambda_j$ in the iterative scheme.\\ \addlinespace[0.3em]
$l_j^*$	& The prescribed sparsity level for the solution under $\mathbf{B}_j$ for $j \in \mathbb{N}_d$.\\ \addlinespace[0.3em]
$l_j^k$	& The actual sparsity level of the solution $\mathbf{u}^k$ at iteration $k$ under the transform $\mathbf{B}_j$, for $j \in \mathbb{N}_d$. \\ \addlinespace[0.3em]
$\mathbb{S}_{+}^s$	& The set of all $s \times s$ symmetric positive definite matrices.\\ \addlinespace[0.3em]
$\| \cdot \|_{\mathbf{H}}$ & The weighted $\ell_2$-norm induced by a positive definite matrix $\mathbf{H}$, defined as $\|\mathbf{x}\|_{\mathbf{H}} := \sqrt{\langle \mathbf{x}, \mathbf{H} \mathbf{x} \rangle}$.\\ \addlinespace[0.3em]
$\text{prox}_{f,\mathbf{H}}$ & The proximity operator of a convex function $f:\mathbb{R}^s\to \overline{\mathbb{R}}$ with respect to a positive definite matrix $\mathbf{H}$, defined by $\text{prox}_{f,\mathbf{H}}(\mathbf{x}):=\argmin\left\{\frac{1}{2}\|\mathbf{y}-\mathbf{x}\|_{\mathbf{H}}^2+f(\mathbf{y}):\mathbf{y}\in\mathbb{R}^s\right\}.$ When $\mathbf{H} = \mathbf{I}$, it is abbreviated as $\mathrm{prox}_f$. \\ \addlinespace[0.3em]
$f^*$ &	The conjugate (Fenchel dual) function of a convex function $f:\mathbb{R}^s\to \overline{\mathbb{R}}$, defined by 
$f^*(\mathbf{y}):=\sup\{\langle \mathbf{x},\mathbf{y}\rangle-f(\mathbf{x}):\mathbf{x}\in\mathbb{R}^s\}$, for all $\mathbf{y}\in\mathbb{R}^s.$\\
\bottomrule
\end{tabular}
\end{table}
%%%%%%%%%%%%%%%%%%%%%%%%%%%%%%%%%%%%%%%%%%%%%%%%%%%%%%%%

\section{Sparsity characterization of regularized solutions under transform matrices}\label{Sparsity}

In this section, we  investigate the relationship between multiple regularization parameters and the sparsity of the corresponding regularized solutions to problem \eqref{optimization_problem_under_Bj}, with respect to the transform matrices $\mathbf{B}_j$, $j \in \mathbb{N}_d$. Unlike the single-parameter setting, the multi-parameter formulation in \eqref{optimization_problem_under_Bj} allows the sparsity of the solution to be evaluated separately for each transform matrix $\mathbf{B}_j$.

%In this section, we characterize the relation between the multiple regularization parameters and the sparsity of the regularized solutions to problem \eqref{optimization_problem_under_Bj} under the transform matrices $\mathbf{B}_{j}$, $j\in\mathbb{N}_d$. Unlike the single-parameter regularization problem, the use of multiple regularization parameters in problem \eqref{optimization_problem_under_Bj} allows us to separately consider the sparsity of the solution under each transform matrix $\mathbf{B}_{j}$.

We begin by recalling the definition of sparsity level for a vector in $\mathbb{R}^s$. For $s\in\mathbb{N}$, we set $\mathbb{Z}_s:=\{0,1,\ldots,s-1\}$. A vector $\mathbf{x}\in\mathbb{R}^s$ is said to have sparsity of level $l\in\mathbb{Z}_{s+1}$ if it has exactly $l$ nonzero components. To further describe the sparsity of a vector in $\mathbb{R}^s$, the notion of the sparsity partition of $\mathbb{R}^s$ was introduced in \cite{xu2023}. Specifically, by using the canonical basis $\mathbf{e}_{s,j},j\in\mathbb{N}_s,$ for $\mathbb{R}^s$, we introduce $s+1$ numbers of subsets of $\mathbb{R}^s$ by $\Omega_{s,0}:=\{\mathbf{0}\in\mathbb{R}^s\}$ and 
$$
\Omega_{s,l}:=\left\{\sum_{j\in\mathbb{N}_l}x_{k_j}\mathbf{e}_{s,k_j}
:x_{k_j}\in\mathbb{R}\setminus{\{0\}},\ \mathrm{for} \
1\leq k_1<k_2<\cdots< k_l\leq s\right\},
\ \mathrm{for} \ l\in\mathbb{N}_s.
$$
It is clear that the sets $\Omega_{s,l},l\in \mathbb{Z}_{s+1}$, form a partition for $\mathbb{R}^s$ and for each $l\in\mathbb{Z}_{s+1}$, $\Omega_{s,l}$ coincides with the set of all vectors in $\mathbb{R}^s$ having sparsity of level $l$. 

To characterize the solution's sparsity under the transform matrices, we reformulate the regularization problem \eqref{optimization_problem_under_Bj} as an equivalent form. Let $p_0:=0$ and $
p_j:=\sum_{i\in\mathbb{N}_j}m_i$, $j\in\mathbb{N}_d$.  We decompose a vector $\mathbf{z}:=[z_k:k\in\mathbb{N}_{p_d}]\in\mathbb{R}^{p_d}$ into $d$ sub-vectors by setting \begin{equation*}\label{z_j}
\mathbf{z}_j:=[z_{p_{j-1}+i}: i\in\mathbb{N}_{m_j}]\in\mathbb{R}^{m_j},\ \mbox{for all}\ j\in\mathbb{N}_d.
\end{equation*}
By introducing a column block matrix 
\begin{equation}\label{block-matrix-B}
\mathbf{B}:=[\mathbf{B}_{j}:j\in\mathbb{N}_d]\in\mathbb{R}^{p_d\times n}, 
\end{equation}
we write $\mathbf{B}\mathbf{u}$ for the block column vector $[\mathbf{B}_j\mathbf{u}: j\in\mathbb{N}_d]$.
Accordingly, we first rewrite the regularization problem \eqref{optimization_problem_under_Bj} as 
\begin{equation}\label{optimization_problem_under_Bj_1}
\min
\left\{\bm{\psi}(\mathbf{u})+\sum_{j\in\mathbb{N}_d}\lambda_j\|(\mathbf{B}\mathbf{u})_{j}\|_{1}:\mathbf{u}\in\mathbb{R}^n\right\}.
\end{equation}

We further convert problem \eqref{optimization_problem_under_Bj_1} into an equivalent form by leveraging a change of variables. To achieve this, we consider inverting the linear system
\begin{equation}\label{LinearSystem}
    \mathbf{B}\mathbf{u}=\mathbf{z}, \ \ \mbox{for each}\ \ \mathbf{z}\in\mathcal{R}(\mathbf{B}).
\end{equation}
Here, $\mathcal{R}(\mathbf{B})$ denotes the range of $\mathbf{B}$. It is known from \cite{Bjork1996,horn2012matrix} that the general solution of the linear system \eqref{LinearSystem} can be represented by the pseudoinverse of  $\mathbf{B}$. An alternative form of the general solution was provided in \cite{Liu2023parameter}. To describe this result, we recall that if $\mathbf{B}$ has the rank $r$ satisfying $0<r\leq \mathrm{min}\{p_d,n\}$, then $\mathbf{B}$ has the SVD as
$\mathbf{B}=\mathbf{U}\mathbf{\Lambda} \mathbf{V}^{\top}$, where $\mathbf{U}$ and $\mathbf{V}$ are $p_d\times p_d$ and $n\times n$ orthogonal matrices, respectively, and $\mathbf{\Lambda}$ is a $p_d\times n$ diagonal matrix with the nonzero diagonal entries $\sigma_1\geq\cdots\geq\sigma_r>0$, which are the nonzero singular values of $\mathbf{B}$. In order to represent the general solution of linear system \eqref{LinearSystem}, we define  
an $n\times(p_d+n-r)$ matrix by employing the SVD of $\mathbf{B}$. Specifically, we denote by $\widetilde{\mathbf{U}}_{r}\in\mathbb{R}^{p_d\times r}$ the matrix composed of the first $r$ columns of $\mathbf{U}$ and define an $n\times(p_d+n-r)$ block diagonal matrix by setting
$$
\mathbf{U}':=\mathrm{diag}\left(\widetilde{\mathbf{U}}_{r}^{\top}, \mathbf{I}_{n-r}\right).
$$ 
We also introduce a diagonal matrix of order $n$ by 
$$
\mathbf{\Lambda}':=\mathrm{diag}\left(\sigma_1^{-1},\sigma_2^{-1},\ldots,\sigma_r^{-1},1,\ldots,1\right).
$$
Using these matrices, we define an $n\times(p_d+n-r)$ matrix by $\mathbf{B}'
:=\mathbf{V}\mathbf{\Lambda}'\mathbf{U}'.$ As has been shown in \cite{Liu2023parameter}, for each solution $\mathbf{u}$ of system  \eqref{LinearSystem}, there exists a unique vector $\mathbf{v}\in\mathbb{R}^{n-r}$ such that 
\begin{equation}\label{solution-Bu=z}
\mathbf{u}=\mathbf{B}'{\scalebox{0.9}{$\begin{bmatrix}
\mathbf{z}\\
\mathbf{v}\end{bmatrix}$}}.
\end{equation}
As a result, the mapping $\mathcal{B}$, defined for each $\mathbf{u}\in\mathbb{R}^n$ by
$\mathcal{B}\mathbf{u}:=\scalebox{0.8}{$\begin{bmatrix}
\mathbf{z}\\
\mathbf{v}\end{bmatrix}$},$
where $\mathbf{z}:=\mathbf{B}\mathbf{u}$ and $\mathbf{v}\in\mathbb{R}^{n-r}$ satisfies equation \eqref{solution-Bu=z}, is bijective from $\mathbb{R}^n$ onto $\mathcal{R}(\mathbf{B})\times\mathbb{R}^{n-r}$. 

By making use of the change of variables defined by equation \eqref{solution-Bu=z}, we reformulate problem \eqref{optimization_problem_under_Bj_1} as an equivalent multi-parameter regularization problem. We set
$\mathbb{M}:=\mathcal{R}(\mathbf{B})\times\mathbb{R}^{n-r}$
and let $\iota_{\mathbb{M}}:\mathbb{R}^{p_d+n-r}\rightarrow\overline{\mathbb{R}}:=\mathbb{R}\cup\{+\infty\}$ denote the indicator function of $\mathbb{M}$, that is, $\iota_{\mathbb{M}}(\mathbf{w})=0$ if $\mathbf{w}\in \mathbb{M}$, and $+\infty$ otherwise. For each $j\in\mathbb{N}_d$, we  introduce a degenerated identity matrix by
\begin{equation}\label{I-j}
\mathbf{I}^{'}_{j}:=[\mathbf{0}_{m_j\times p_{j-1}}\  \mathbf{I}_{m_j}\ \mathbf{0}_{m_j\times (p_d-p_j)} \  \mathbf{0}_{m_j\times (n-r)}]\in\mathbb{R}^{m_j\times (p_d+n-r)},
\end{equation}
where $\mathbf{0}_{s\times t}$ denotes the zero matrix of order $s\times t$. 
We show in the following lemma that the regularization problem \eqref{optimization_problem_under_Bj} is equivalent to the regularization problem 
\begin{equation}\label{optimization_problem_under_Bj_3}
\min
\left\{\bm{\psi}\circ\mathbf{B}'(\mathbf{w})+\iota_{\mathbb{M}}(\mathbf{w})+\sum_{j\in\mathbb{N}_d}\lambda_j\|\mathbf{I}^{'}_{j}\mathbf{w}\|_{1}:\mathbf{w}\in\mathbb{R}^{p_d+n-r}\right\}.
\end{equation}

\begin{lemma}\label{equi_mini}
If matrix $\mathbf{B}$ has the form \eqref{block-matrix-B} and $\mathbf{B}'$,  $\mathcal{B}$ are defined as above, then $\mathbf{u}^*$ is a solution of the regularization problem \eqref{optimization_problem_under_Bj} if and only if $\mathcal{B}\mathbf{u}^*$ is a solution of the regularization problem  \eqref{optimization_problem_under_Bj_3}.
\end{lemma}
\begin{proof}
We first prove that $\mathbf{u}^*$ is a solution of problem \eqref{optimization_problem_under_Bj} if and only if $\mathcal{B}\mathbf{u}^*$ is a solution of the constrained optimization problem  \begin{equation}\label{optimization_problem_under_Bj_2}
\min
\left\{\bm{\psi}\circ\mathbf{B}'(\mathbf{w})+\sum_{j\in\mathbb{N}_d}\lambda_j\|\mathbf{I}^{'}_{j}\mathbf{w}\|_{1}:\mathbf{w}\in\mathbb{M}\right\}.
\end{equation} 
As has been shown in \cite{Liu2023parameter}, $\mathcal{B}$ is a bijective mapping from $\mathbb{R}^n$ to $\mathbb{M}$. It suffices to verify that for all $\mathbf{u}\in\mathbb{R}^n$ there holds 
$$
\bm{\psi}(\mathbf{u})+\sum_{j\in\mathbb{N}_d}\lambda_j\|(\mathbf{B}\mathbf{u})_{j}\|_{1}=\bm{\psi}\circ\mathbf{B}'(\mathcal{B}\mathbf{u})+\sum_{j\in\mathbb{N}_d}\lambda_j\|\mathbf{I}^{'}_{j}(\mathcal{B}\mathbf{u})\|_{1}.
$$
By the definition of mapping $\mathcal{B}$, we get that 
$\mathbf{B}'\mathcal{B}\mathbf{u}=\mathbf{u}$ and $\mathbf{I}^{'}_{j}(\mathcal{B}\mathbf{u})=(\mathbf{B}\mathbf{u})_j$, which confirm the validity of the equation
above.

We next show that problem  \eqref{optimization_problem_under_Bj_2} as a constrained optimization problem is equivalent to the unconstrained optimization problem \eqref{optimization_problem_under_Bj_3}. By the definition of the indicator function $\iota_{\mathbb{M}}$,  the minimum of problem  \eqref{optimization_problem_under_Bj_3} will be assumed at an element $\mathbf{w}\in\mathbb{M}$. Thus, problem  \eqref{optimization_problem_under_Bj_3} can be
rewritten as problem \eqref{optimization_problem_under_Bj_2}. 
\end{proof}

Below, we consider how the regularization parameters $\lambda_j$, $j\in\mathbb{N}_d$ influence the sparsity of the solution of problem  \eqref{optimization_problem_under_Bj_3}. To characterize the solution of problem \eqref{optimization_problem_under_Bj_3}, we require the notion of the subdifferential of a convex function. Suppose that $f: \mathbb{R}^s\to \overline{\mathbb{R}}$ is a proper convex function. The subdifferential of $f$ at $\mathbf{x}\in{\rm dom}(f):=\{\mathbf{y}\in\mathbb{R}^s:f(\mathbf{y})<+\infty\}$ is defined by
$$
\partial f(\mathbf{x}):=\{\mathbf{y}\in\mathbb{R}^s: \ f(\mathbf{z})\geq f(\mathbf{x})+\langle \mathbf{y},\mathbf{z}-\mathbf{x}\rangle, \ \mathrm{for} \ \mathrm{all} \ \mathbf{z}\in\mathbb{R}^s\}.
$$
It is known \cite{zalinescu2002convex} that for two convex functions $f$ and $g$ on $\mathbb{R}^s$, if $g$ is continuous on $\mathbb{R}^s$ then 
$\partial (f+g)(\mathbf{x})
=\partial f(\mathbf{x}) +\partial g(\mathbf{x})$, for all $\mathbf{x}\in\mathbb{R}^s$. 
We also describe the chain rule of the subdifferential \cite{Showalter1997}. Suppose
that $f:\mathbb{R}^s\rightarrow\overline{\mathbb{R}}$ is a convex function and
$\mathbf{M}$ is an $s\times t$ matrix. If $f$ is continuous at some point of the range of
$\mathbf{M}$, then for all $\mathbf{x}\in\mathbb{R}^t$ 
\begin{equation}\label{chain-rule}
\partial (f\circ \mathbf{M})(\mathbf{x})=\mathbf{M}^{\top}\partial f(\mathbf{M}(\mathbf{x})).
\end{equation}
The Fermat rule \cite{zalinescu2002convex} states that if a proper convex function
$f:\mathbb{R}^s\rightarrow\overline{\mathbb{R}}$ has a minimum at $\mathbf{x}\in\mathbb{R}^s$ if and only if  $\mathbf{0}\in \partial f(\mathbf{x})$. 

In the next lemma, we characterize the sparsity of the solution to problem \eqref{optimization_problem_under_Bj_3}. For an $s\times t$ matrix $\mathbf{M}$, we denote by $\mathcal{N}(\mathbf{M})$ the null space of $\mathbf{M}$ and for each $i\in\mathbb{N}_{t}$, denote by $\mathbf{M}_{(i)}$ the $i$th column of $\mathbf{M}$.

\begin{lemma}\label{sparsity-equi-mini}
Suppose that $\bm{\psi}:\mathbb{R}^n\to\mathbb{R}_+$ is a convex function  and for each $j\in\mathbb{N}_d$, $\mathbf{B}_{j}$ is an $m_j\times n$ matrix. Let $\mathbf{B}$ be defined as in \eqref{block-matrix-B}.  Then problem \eqref{optimization_problem_under_Bj_3}
with $\lambda_j>0$, $j\in \mathbb{N}_d$, has a solution $\mathbf{w}^{*}:=\scalebox{0.8}{$\begin{bmatrix}
\mathbf{z^*}\\
\mathbf{v^*}\end{bmatrix}$}$, where for each $j\in \mathbb{N}_d,$ $\mathbf{z}_j^{*}=\sum_{i\in\mathbb{N}_{l_j}}z_{p_{j-1}+k_i}^*\mathbf{e}_{m_j,k_i}\in \Omega_{m_j,l_j},$ for some $l_{j}\in\mathbb{Z}_{m_j+1}$ if and only if there exist  $\mathbf{a}\in\partial\bm{\psi}(\mathbf{B}'\mathbf{w}^*)$ and $\mathbf{b}:=[b_j:j\in\mathbb{N}_{p_d}]\in\mathcal{N}(\mathbf{B}^{\top})$ such that 
\begin{equation}\label{sparsity-equi-mini-1}
(\mathbf{B}'_{(i)})^{\top}\mathbf{a}={0}, \ i\in\mathbb{N}_{p_d+n-r}\setminus\mathbb{N}_{p_d}, 
\end{equation}
and for each $j\in\mathbb{N}_d$
\begin{equation}\label{sparsity-equi-mini-2}
\lambda_j=-\left( (\mathbf{B}_{(p_{j-1}+k_i)}')^\top\mathbf{a}+b_{p_{j-1}+k_i}\right)\mathrm{sign}(z_{p_{j-1}+k_i}^*),\  i\in \mathbb{N}_{l_j},
\end{equation}
\begin{equation}\label{sparsity-equi-mini-3}
\lambda_j\geq\left|(\mathbf{B}_{(p_{j-1}+i)}')^\top\mathbf{a}+b_{p_{j-1}+i}\right|, \ i\in \mathbb{N}_{m_j}\setminus\{k_i:i\in \mathbb{N}_{l_j}\}.
\end{equation}
\end{lemma}

\begin{proof}
By the Fermat rule, we have that  
$\mathbf{w}^*:=\scalebox{0.8}{$\begin{bmatrix}
\mathbf{z}^*\\
\mathbf{v}^*\end{bmatrix}$}$ is a solution of problem \eqref{optimization_problem_under_Bj_3} if and only if  
$$
\mathbf{0}\in\partial\left(\bm{\psi}\circ\mathbf{B}'+\iota_{\mathbb{M}}+\sum_{j\in\mathbb{N}_d}\lambda_j\|\cdot\|_{1}\circ\mathbf{I}^{'}_{j}\right)(\mathbf{w}^*),
$$
which by the continuity of the $\ell_1$-norm and the chain rule \eqref{chain-rule} of the subdifferential is equivalent to
\begin{equation}\label{Fermat-rule-chain-rule}
\mathbf{0}\in(\mathbf{B}')^\top\partial\bm{\psi}(\mathbf{B}'\mathbf{w}^*)+\partial\iota_{\mathbb{M}}(\mathbf{w}^*)+\sum_{j\in\mathbb{N}_d}\lambda_j(\mathbf{I}^{'}_{j})^\top\partial \|\cdot\|_1(\mathbf{z}_j^*).
\end{equation}
Let $\mathbb{M}^{\bot}$ denote the  orthogonal complement of $\mathbb{M}$. It is known that    $\partial\iota_{\mathbb{M}}(\mathbf{w})=\mathbb{M}^{\bot}$ for all $\mathbf{w}\in \mathbb{M}$. Recalling that $\mathbb{M}:=\mathcal{R}(\mathbf{B})\times\mathbb{R}^{n-r}$, we get that 
$$
\mathbb{M}^{\bot}=(\mathcal{R}(\mathbf{B}))^{\bot}\times(\mathbb{R}^{n-r})^{\bot}=\mathcal{N}(\mathbf{B}^{\top})\times\{\mathbf{0}\}.
$$
As a result, 
$$
\partial\iota_{\mathbb{M}}(\mathbf{w})=\mathcal{N}(\mathbf{B}^{\top})\times\{\mathbf{0}\},\ \ \mbox{for all}\ \ \mathbf{w}\in \mathbb{M}.
$$
Substituting the above equation with $\mathbf{w}:=\mathbf{w}^*$ into the inclusion relation \eqref{Fermat-rule-chain-rule}, we conclude that  $\mathbf{w}^*:=\scalebox{0.8}{$\begin{bmatrix}
\mathbf{z}^*\\
\mathbf{v}^*\end{bmatrix}$}$ is a solution of problem \eqref{optimization_problem_under_Bj_3} if and only if there exist  $\mathbf{a}\in\partial\bm{\psi}(\mathbf{B}'\mathbf{w}^*)$ and $\mathbf{b}\in\mathcal{N}(\mathbf{B}^{\top})$ such that 
\begin{equation}\label{Fermat-rule-chain-rule-1}
  -(\mathbf{B}')^\top\mathbf{a}-\raisebox{-0.3ex}{\scalebox{0.8}{$ 
  \begin{bmatrix}\mathbf{b}\\ \mathbf{0}\end{bmatrix}$}}\in\sum_{j\in\mathbb{N}_d}\lambda_j(\mathbf{I}^{'}_{j})^\top\partial \|\cdot\|_1(\mathbf{z}_j^*).  
\end{equation}
Note that for each $j\in\mathbb{N}_d$
$$
\mathbf{z}_j^{*}=\sum_{i\in\mathbb{N}_{l_j}}z_{p_{j-1}+k_i}^*\mathbf{e}_{m_j,k_i}\in \Omega_{m_j,l_j},\ \mbox{with}\ z_{p_{j-1}+k_i}^*\in\mathbb{R}\setminus{\{0\}}, \ i\in \mathbb{N}_{l_j},
$$ 
with which we obtain that
\begin{equation*}\label{subdiff-1norm-z*}
\partial\|\cdot\|_1(\mathbf{z}_j^{*})
=\left\{\mathbf{x}\in\mathbb{R}^{m_j}: x_{k_i}=\mathrm{sign}(z_{p_{j-1}+k_i}^*), i\in\mathbb{N}_{l_j}\ \mbox{and}\ |x_i|\leq 1, i\in\mathbb{N}_{m_j}\setminus\{k_i:i\in \mathbb{N}_{l_j}\}\right\}.
\end{equation*}
Combining inclusion relation \eqref{Fermat-rule-chain-rule-1} with the above equation, we have that $(\mathbf{B}'_{(i)})^{\top}\mathbf{a}={0}$ for all $i\in\mathbb{N}_{p_d+n-r}\setminus\mathbb{N}_{p_d}$, which coincides with equation \eqref{sparsity-equi-mini-1} and for each $j\in\mathbb{N}_d$
$$
-(\mathbf{B}_{(p_{j-1}+k_i)}')^\top\mathbf{a}-b_{p_{j-1}+k_i}=\lambda_j\mathrm{sign}(z_{p_{j-1}+k_i}^*),\ i\in \mathbb{N}_{l_j},
$$
$$
-(\mathbf{B}_{(p_{j-1}+i)}')^\top\mathbf{a}-b_{p_{j-1}+i}\in[-\lambda_j,\lambda_j], \ i\in \mathbb{N}_{m_j}\setminus\{k_i:i\in \mathbb{N}_{l_j}\},
$$
which are equivalent to equation \eqref{sparsity-equi-mini-2} and inequality \eqref{sparsity-equi-mini-3}, respectively.
\end{proof}

Combining Lemmas \ref{equi_mini} with \ref{sparsity-equi-mini}, we now establish the relation between the regularization parameters and the sparsity of the solution to problem \eqref{optimization_problem_under_Bj}.
\begin{theorem}\label{sparsity_original}
Suppose that $\bm{\psi}:\mathbb{R}^n\to\mathbb{R}_+$ is a convex function  and for each $j\in\mathbb{N}_d$, $\mathbf{B}_{j}$ is an $m_j\times n$ matrix. Let $\mathbf{B}$ be defined as in \eqref{block-matrix-B}. Then problem \eqref{optimization_problem_under_Bj}
with $\lambda_j>0, j\in \mathbb{N}_d,$ has a solution $\mathbf{u}^{*}\in\mathbb{R}^n$, where for each $j\in \mathbb{N}_d,$ $\mathbf{B}_{j}\mathbf{u}^{*}=\sum_{i\in\mathbb{N}_{l_j}}z_{p_{j-1}+k_i}^*\mathbf{e}_{m_j,k_i}\in \Omega_{m_j,l_j},$ for some $l_{j}\in\mathbb{Z}_{m_j+1}$ if and only if there exist  $\mathbf{a}\in\partial\bm{\psi}(\mathbf{u}^*)$ and $\mathbf{b}:=[b_j:j\in\mathbb{N}_{p_d}]\in\mathcal{N}(\mathbf{B}^{\top})$ such that \eqref{sparsity-equi-mini-1}, \eqref{sparsity-equi-mini-2} and \eqref{sparsity-equi-mini-3} hold. 
\end{theorem}
\begin{proof}    
It follows from Lemma \ref{equi_mini} that $\mathbf{u}^{*}\in\mathbb{R}^n$, where for each $j\in \mathbb{N}_d,$ $\mathbf{B}_{j}\mathbf{u}^{*}=\sum_{i\in\mathbb{N}_{l_j}}z_{p_{j-1}+k_i}^*\mathbf{e}_{m_j,k_i}\in \Omega_{m_j,l_j}$ is a solution of problem \eqref{optimization_problem_under_Bj} if and only if $\mathcal{B}\mathbf{u}^*:=
\scalebox{0.8}{$\begin{bmatrix}
\mathbf{z}^*\\
\mathbf{v}^*\end{bmatrix}$}$ is a solution of problem \eqref{optimization_problem_under_Bj_3} and for each $j\in \mathbb{N}_d$,  $\mathbf{z}_j^{*}=\sum_{i\in\mathbb{N}_{l_j}}z_{p_{j-1}+k_i}^*\mathbf{e}_{m_j,k_i}\in \Omega_{m_j,l_j}$. The latter guaranteed by Lemma \ref{sparsity-equi-mini} is equivalent to that there exist $\mathbf{a}\in\partial\bm{\psi}(\mathbf{B}'\mathcal{B}\mathbf{u}^*)$ and $\mathbf{b}\in\mathcal{N}(\mathbf{B}^{\top})$ such that \eqref{sparsity-equi-mini-1}, \eqref{sparsity-equi-mini-2} and \eqref{sparsity-equi-mini-3} hold. It suffices to show that $\mathbf{a}\in\partial\bm{\psi}(\mathbf{u}^*)$. 
This is done by noting that $\mathbf{B}'\mathcal{B}\mathbf{u}^*=\mathbf{u}^*$. 
\end{proof}

Theorem \ref{sparsity_original}  provides
a characterization of the multiple regularization parameter $\lambda_j$, $j\in\mathbb{N}_d$, with which problem \eqref{optimization_problem_under_Bj} has a solution with
sparsity of a certain level $l_j$ under each transform matrix $\mathbf{B}_{j}$. By specifying the sparsity level of the solution under each transform matrix $\mathbf{B}_{j}$ to be $l_j^*$, our goal is
to find regularization parameters $\lambda_j^*$, $j\in\mathbb{N}_d$, satisfying
conditions \eqref{sparsity-equi-mini-1}, \eqref{sparsity-equi-mini-2} and \eqref{sparsity-equi-mini-3}. However, since these conditions depend on the corresponding
solution, the characterization stated in Theorem \ref{sparsity_original} can not be used directly as a multi-parameter
choice strategy. Motivated by Theorem \ref{sparsity_original}, an iterative scheme to be developed in section \ref{Itera} will enable us to choose multiple regularization parameters with which a minimizer of problem \eqref{optimization_problem_under_Bj} has a prescribed sparsity level under each transform matrix.

The next result concerns the special case that matrix $\mathbf{B}$ defined by \eqref{block-matrix-B} has full row rank, that is, $\mathrm{rank}(\mathbf{B})=p_d$.
\begin{corollary}\label{sparsity-rank}
    Suppose that $\bm{\psi}:\mathbb{R}^n\to\mathbb{R}_+$ is a convex function  and for each $j\in\mathbb{N}_d$, $\mathbf{B}_{j}$ is an $m_j\times n$ matrix. If $\mathbf{B}$ defined by \eqref{block-matrix-B} satisfies $\mathrm{rank}(\mathbf{B})=p_d$, then problem \eqref{optimization_problem_under_Bj} with $\lambda_j>0, j\in \mathbb{N}_d,$ has a solution $\mathbf{u}^{*}\in\mathbb{R}^n$, where for each $j\in \mathbb{N}_d,$ $\mathbf{B}_{j}\mathbf{u}^{*}=\sum_{i\in\mathbb{N}_{l_j}}z_{p_{j-1}+k_i}^*\mathbf{e}_{m_j,k_i}\in \Omega_{m_j,l_j},$ for some $l_{j}\in\mathbb{Z}_{m_j+1}$ if and only if there exists $\mathbf{a}\in\partial\bm{\psi}(\mathbf{u}^*)$ such that 
\begin{equation}\label{sparsity-equi-mini-1-rank}
(\mathbf{B}'_{(i)})^{\top}\mathbf{a}={0}, \ i\in\mathbb{N}_{n}\setminus\mathbb{N}_{p_d}, 
\end{equation}
and for each $j\in\mathbb{N}_d$
\begin{equation}\label{sparsity-equi-mini-2-rank}
\lambda_j=-(\mathbf{B}_{(p_{j-1}+k_i)}')^\top\mathbf{a}\mathrm{sign}(z_{p_{j-1}+k_i}^*),\  i\in \mathbb{N}_{l_j},
\end{equation}
\begin{equation}\label{sparsity-equi-mini-3-rank}
\lambda_j\geq\left|(\mathbf{B}_{(p_{j-1}+i)}')^\top\mathbf{a}\right|, \ i\in \mathbb{N}_{m_j}\setminus\{k_i:i\in \mathbb{N}_{l_j}\}.
\end{equation}
\end{corollary}
\begin{proof}
Theorem \ref{sparsity_original} ensures that $\mathbf{u}^{*}\in\mathbb{R}^n$, where for each $j\in \mathbb{N}_d,$ $\mathbf{B}_{j}\mathbf{u}^{*}=\sum_{i\in\mathbb{N}_{l_j}}z_{p_{j-1}+k_i}^*\mathbf{e}_{m_j,k_i}\in \Omega_{m_j,l_j}$ is a solution of problem \eqref{optimization_problem_under_Bj} if and only if there exist  $\mathbf{a}\in\partial\bm{\psi}(\mathbf{u}^*)$ and $\mathbf{b}\in\mathcal{N}(\mathbf{B}^{\top})$ such that \eqref{sparsity-equi-mini-1}, \eqref{sparsity-equi-mini-2} and \eqref{sparsity-equi-mini-3} hold. By the assumption that $\mathrm{rank}(\mathbf{B})=p_d$, we rewrite equation \eqref{sparsity-equi-mini-1} as equation \eqref{sparsity-equi-mini-1-rank}. It follows from $\mathrm{rank}(\mathbf{B})=p_d$ that $\mathcal{N}(\mathbf{B}^\top)=\left(\mathcal{R}(\mathbf{B})\right)^\bot=\{\mathbf{0}\}$. Then vector $\mathbf{b}$ in \eqref{sparsity-equi-mini-2} and \eqref{sparsity-equi-mini-3} is the zero vector. Thus, \eqref{sparsity-equi-mini-2-rank} and \eqref{sparsity-equi-mini-3-rank} can be obtained directly.
\end{proof}

Note that if $\bm{\psi}$ is differentiable, then the subdifferential of  $\bm{\psi}$ at $\mathbf{u}^*$
is the singleton $\nabla\bm{\psi}(\mathbf{u}^*)$. In this case, Theorem \ref{sparsity_original} and Corollary \ref{sparsity-rank} have the following simple form.

\begin{corollary}\label{sparsity-diff-rank}
Suppose that $\bm{\psi}:\mathbb{R}^n\to\mathbb{R}_+$ is a differentiable and convex function and for each $j\in\mathbb{N}_d$, $\mathbf{B}_{j}$ is an $m_j\times n$ matrix. Let $\mathbf{B}$ be defined as in \eqref{block-matrix-B}. Then problem \eqref{optimization_problem_under_Bj} with $\lambda_j>0, j\in \mathbb{N}_d,$ has a solution $\mathbf{u}^{*}\in\mathbb{R}^n$, where for each $j\in \mathbb{N}_d,$ $\mathbf{B}_{j}\mathbf{u}^{*}=\sum_{i\in\mathbb{N}_{l_j}}z_{p_{j-1}+k_i}^*\mathbf{e}_{m_j,k_i}\in \Omega_{m_j,l_j},$ for some $l_{j}\in\mathbb{Z}_{m_j+1}$ if and only if there exist $\mathbf{b}:=[b_j:j\in\mathbb{N}_{p_d}]\in\mathcal{N}(\mathbf{B}^{\top})$ such that     \begin{equation*}\label{sparsity-equi-mini-1-diff}
(\mathbf{B}'_{(i)})^{\top}\nabla\bm{\psi}(\mathbf{u}^*)={0}, \ i\in\mathbb{N}_{p_d+n-r}\setminus\mathbb{N}_{p_d}, 
\end{equation*}
and for each $j\in\mathbb{N}_d$
\begin{equation*}\label{sparsity-equi-mini-2-diff}
\lambda_j=-\left( (\mathbf{B}_{(p_{j-1}+k_i)}')^\top\nabla\bm{\psi}(\mathbf{u}^*)+b_{p_{j-1}+k_i}\right)\mathrm{sign}(z_{p_{j-1}+k_i}^*),\  i\in \mathbb{N}_{l_j},
\end{equation*}
\begin{equation*}\label{sparsity-equi-mini-3-diff}
\lambda_j\geq\left|(\mathbf{B}_{(p_{j-1}+i)}')^\top\nabla\bm{\psi}(\mathbf{u}^*)+b_{p_{j-1}+i}\right|, \ i\in \mathbb{N}_{m_j}\setminus\{k_i:i\in \mathbb{N}_{l_j}\}.
\end{equation*}
In particular, if $\mathrm{rank}(\mathbf{B})=p_d$, then the conditions reduce to
\begin{equation*}\label{sparsity-equi-mini-1-diff-rank}
(\mathbf{B}'_{(i)})^{\top}\nabla\bm{\psi}(\mathbf{u}^*)={0}, \ i\in\mathbb{N}_{n}\setminus\mathbb{N}_{p_d}, 
\end{equation*}
and for each $j\in\mathbb{N}_d$
\begin{equation*}\label{sparsity-equi-mini-2-diff-rank}
\lambda_j=-(\mathbf{B}_{(p_{j-1}+k_i)}')^\top\nabla\bm{\psi}(\mathbf{u}^*)\mathrm{sign}(z_{p_{j-1}+k_i}^*),\  k_i\in \mathbb{N}_{l_j},
\end{equation*}
\begin{equation*}\label{sparsity-equi-mini-3-diff-rank}
\lambda_j\geq\left|(\mathbf{B}_{(p_{j-1}+i)}')^\top\nabla\bm{\psi}(\mathbf{u}^*)\right|, \ i\in \mathbb{N}_{m_j}\setminus\{k_i:i\in \mathbb{N}_{l_j}\}.
\end{equation*}
\end{corollary}

In the rest of this section, we consider the special case where $p_d=n$ and for each $j\in \mathbb{N}_d,$ the transform matrix $\mathbf{B}_{j}$ takes the form
\begin{equation}\label{B-j-special}\mathbf{B}_{j}:=[\mathbf{0}_{m_j\times p_{j-1}}\  \mathbf{I}_{m_j}\ \mathbf{0}_{m_j\times (n-p_j)}]\in\mathbb{R}^{m_j\times n}.
\end{equation}
In this scenario, the multi-parameter regularization problem \eqref{optimization_problem_under_Bj} assumes the special form
\begin{equation}\label{optimization_problem}
\min
\left\{\bm{\psi}(\mathbf{u})+\sum_{j\in\mathbb{N}_d}\lambda_j\|\mathbf{u}_j\|_{1}:\mathbf{u}\in\mathbb{R}^n\right\}.
\end{equation}
We specialize the previously derived characterizations of the sparsity of regularized solutions to this special case. Moreover, particular attention is given to scenarios where the fidelity term $\bm{\psi}$ is block-separable. 

We first characterize the sparsity of the solution to problem \eqref{optimization_problem}. It follows from equation \eqref{B-j-special} that matrix $\mathbf{B}$ defined by \eqref{block-matrix-B} coincides with the identity matrix of order $n$. That is, $\mathbf{B}$ has full row rank. As a result, we may specialize Corollary \ref{sparsity-rank} to the regularization problem \eqref{optimization_problem}. In addition, the transform matrices $\mathbf{B}_{j}$, $j\in\mathbb{N}_d$, with the form \eqref{B-j-special} enable us to consider the sparsity of each sub-vector of the regularized solution separately. 

\begin{theorem}\label{sparsity-special-case}
Suppose that $\bm{\psi}:\mathbb{R}^n\to\mathbb{R}_+$ is a convex function. Then problem \eqref{optimization_problem}
with $\lambda_j>0$, $j\in \mathbb{N}_d$, has a solution $\mathbf{u}^{*}$, where for each $j\in \mathbb{N}_d$, $\mathbf{u}_j^{*}=\sum_{i\in\mathbb{N}_{l_j}}u^*_{p_{j-1}+k_i}\mathbf{e}_{m_j,k_i}\in \Omega_{m_j,l_j}$, for some $l_{j}\in\mathbb{Z}_{m_j+1}$ if and only if there exists $\mathbf{a}:=[a_j:j\in\mathbb{N}_{n}]\in\partial\bm{\psi}(\mathbf{u}^{*})$ such that for each $j\in \mathbb{N}_d$
\begin{equation}\label{sparsity-equi-mini-1-special}
\lambda_{j}=-a_{p_{j-1}+k_i}\mathrm{sign}(u^*_{p_{j-1}+k_i}), \
i\in\mathbb{N}_{l_j},
\end{equation}
\begin{equation}\label{sparsity-equi-mini-2-special}
\lambda_j\geq|a_{p_{j-1}+i}|,\ i\in \mathbb{N}_{m_j}\setminus\{k_i:i\in \mathbb{N}_{l_j}\}.
\end{equation}
In particular, if $\bm{\psi}$ is differentiable, then the conditions reduce to for each $j\in \mathbb{N}_d$
\begin{equation}\label{sparsity-equi-mini-1-special-diff}
\lambda_{j}=-\frac{\partial{\bm{\psi}}}{\partial{u_{p_{j-1}+k_i}}}(\mathbf{u}^{*})\mathrm{sign}(u^*_{p_{j-1}+k_i}), \
i\in\mathbb{N}_{l_j},
\end{equation}
\begin{equation}\label{sparsity-equi-mini-2-special-diff}
\lambda_j\geq\left|\frac{\partial{\bm{\psi}}}{\partial{u_{p_{j-1}+i}}}(\mathbf{u}^{*})\right|, \ i\in \mathbb{N}_{m_j}\setminus\{k_i:i\in \mathbb{N}_{l_j}\}.
\end{equation}
\end{theorem}
\begin{proof}
Since matrix $\mathbf{B}$ defined by equation \eqref{block-matrix-B} has full row rank, Corollary \ref{sparsity-rank} ensures that problem \eqref{optimization_problem}
with $\lambda_j>0$, $j\in \mathbb{N}_d$, has a solution $\mathbf{u}^{*}$, where for each $j\in \mathbb{N}_d$, $\mathbf{u}_j^{*}=\sum_{i\in\mathbb{N}_{l_j}}u^*_{p_{j-1}+k_i}\mathbf{e}_{m_j,k_i}\in \Omega_{m_j,l_j}$, for some $l_{j}\in\mathbb{Z}_{m_j+1}$ if and only if there exists $\mathbf{a}:=[a_j:j\in\mathbb{N}_{n}]\in\partial\bm{\psi}(\mathbf{u}^{*})$ such that conditions \eqref{sparsity-equi-mini-1-rank}, \eqref{sparsity-equi-mini-2-rank} and \eqref{sparsity-equi-mini-3-rank} hold. Note that index $i$ such that equation \eqref{sparsity-equi-mini-1-rank} holds belongs to an empty set since $p_d=n$. It is clear that matrix $\mathbf{B}'$ appearing in conditions \eqref{sparsity-equi-mini-2-rank} and \eqref{sparsity-equi-mini-3-rank} is also an identity matrix. Hence, conditions \eqref{sparsity-equi-mini-2-rank} and \eqref{sparsity-equi-mini-3-rank} reduce to \eqref{sparsity-equi-mini-1-special} and \eqref{sparsity-equi-mini-2-special}, respectively. If $\bm{\psi}$ is differentiable,
then the subdifferential of $\bm{\psi}$ at $\mathbf{u}^{*}$ is the singleton  $\nabla\bm{\psi}(\mathbf{u}^{*})$. Substituting $a_j=\frac{\partial{\bm{\psi}}}{\partial{u_{j}}}(\mathbf{u}^{*})$ into \eqref{sparsity-equi-mini-1-special} and \eqref{sparsity-equi-mini-2-special} lead directly to \eqref{sparsity-equi-mini-1-special-diff} and \eqref{sparsity-equi-mini-2-special-diff}, respectively.
\end{proof}

As a specific example, we consider the regularization problem
\begin{equation}\label{lasso_multi-parameter}
\min\left\{\frac{1}{2}\left\|\mathbf{A}\mathbf{u}-\mathbf{x}\right\|_2^2
+\sum_{j\in\mathbb{N}_d}\lambda_j\|\mathbf{u}_j\|_1:\mathbf{u}\in\mathbb{R}^{n}\right\}.
\end{equation}
In this model, the fidelity term \begin{equation}\label{fidelity_term}
\bm{\psi}(\mathbf{u}):=\frac{1}{2}\|\mathbf{A}\mathbf{u}-\mathbf{x}\|_2^2, \ \mathbf{u}\in\mathbb{R}^n,
\end{equation}
is convex and differentiable. As a result, we apply Theorem \ref{sparsity-special-case} to this model.

\begin{corollary}\label{sparsity-special-case-example}
Suppose that $\mathbf{x}\in\mathbb{R}^t$ and $\mathbf{A}\in\mathbb{R}^{t\times n}$ are given. Then the regularization problem \eqref{lasso_multi-parameter} with $\lambda_j>0,j\in\mathbb{N}_d$, has a solution $\mathbf{u}^{*}$, where for each $j\in\mathbb{N}_d$, $\mathbf{u}_j^{*}=\sum_{i\in\mathbb{N}_{l_j}}u_{p_{j-1}+k_i}^*\mathbf{e}_{m_j,k_i}\in \Omega_{m_j,l_j}$, for some $l_{j}\in\mathbb{Z}_{m_j+1}$ if and only if for each $j\in\mathbb{N}_d$ 
\begin{equation}\label{sparsity-equi-mini-1-special-diff-example}
\lambda_{j}=(\mathbf{A}_{(p_{j-1}+k_i)})^{\top}(\mathbf{x}-\mathbf{Au^*})\mathrm{sign}(u^{*}_{p_{j-1}+k_i}),\ i\in \mathbb{N}_{l_j}
\end{equation}
\begin{equation}\label{sparsity-equi-mini-2-special-diff-example}
\lambda_j\geq\big|(\mathbf{A}_{(p_{j-1}+i)})^{\top}(\mathbf{Au^*}-\mathbf{x})\big|,\  i\in \mathbb{N}_{m_j}\setminus\{k_i:i\in \mathbb{N}_{l_j}\}.
\end{equation}
\end{corollary}
\begin{proof}
Since the fidelity term $\bm{\psi}(\mathbf{u})$ defined by \eqref{fidelity_term} is convex and differentiable, Theorem \ref{sparsity-special-case} confirms 
that problem \eqref{lasso_multi-parameter} with $\lambda_j>0$, $j\in \mathbb{N}_d$, has a solution $\mathbf{u}^{*}$, where for each $j\in \mathbb{N}_d$, $\mathbf{u}_j^{*}=\sum_{i\in\mathbb{N}_{l_j}}u^*_{p_{j-1}+k_i}\mathbf{e}_{m_j,k_i}\in \Omega_{m_j,l_j}$, for some $l_{j}\in\mathbb{Z}_{m_j+1}$ if and only if \eqref{sparsity-equi-mini-1-special-diff} and \eqref{sparsity-equi-mini-2-special-diff} hold. 
Note that the gradient of $\bm{\psi}$ at $\mathbf{u}^*$ has the form $\nabla\bm{\psi}(\mathbf{u}^*)=\mathbf{A}^\top(\mathbf{A}\mathbf{u}^*-\mathbf{x})$. As a result, there holds for each $j\in\mathbb{N}_n$ and each $i\in\mathbb{N}_{m_j}$ 
$$
\frac{\partial{\bm{\psi}}}{\partial{u_{p_{j-1}+i}}}(\mathbf{u}^{*})=(\mathbf{A}_{(p_{j-1}+i)})^{\top}(\mathbf{Au^*}-\mathbf{x}).
$$
According to the above representations of the partial derivatives of $\bm{\psi}$, conditions \eqref{sparsity-equi-mini-1-special-diff} and \eqref{sparsity-equi-mini-2-special-diff} reduce to \eqref{sparsity-equi-mini-1-special-diff-example} and \eqref{sparsity-equi-mini-2-special-diff-example}, respectively.
  
\end{proof}

We next study the case that the fidelity term $\bm{\psi}$ involved in problem \eqref{optimization_problem} has special structure, that is, $\bm{\psi}$ is block separable. To describe the block separability of a function on $\mathbb{R}^s$, we introduce a partition of the index set $\mathbb{N}_s$. Let $q\in\mathbb{N}$ with $q\leq s$. We suppose that $\mathcal{S}_{s,q}:=\left\{S_{s,1},S_{s,2},\ldots, S_{s,q}\right\}$ is a partition of $\mathbb{N}_s$ in the sense that $S_{s,k}\neq\emptyset$, for all $k\in\mathbb{N}_q$, $S_{s,k}\cap S_{s,l}=\emptyset$ if $k\neq l$, and $\cup_{k\in\mathbb{N}_q}S_{s,k}=\mathbb{N}_s$. For each $k\in \mathbb{N}_q$ we denote by $s_k$ the cardinality of $S_{s,k}$ and regard $S_{s,k}$ as an {\it ordered set} in the natural order of the elements in $\mathbb{N}_s$. That is, 
$$
S_{s,k}:=\{i_{k,1}, \dots, i_{k,s_k}\}, \ \mbox{with}\  i_{k,l}\in \mathbb{N}_s, \ l\in \mathbb{N}_{s_k}\ \mbox{and}\ i_{k,1}<\dots<i_{k,s_k}.
$$
Associated with partition $\mathcal{S}_{s,q}$, we decompose $\mathbf{w}:=[w_k:k\in\mathbb{N}_s]\in\mathbb{R}^s$
into $q$ sub-vectors by setting 
$$
\mathbf{w}_{S_{s,k}}:=[w_{i_{k,1}}, \dots, w_{{i}_{k,s_k}}]\in\mathbb{R}^{s_k},\ k\in\mathbb{N}_q.
$$
A function  $\bm{\phi}:\mathbb{R}^s\rightarrow\mathbb{R}$ is called $\mathcal{S}_{s,q}$-block separable if there exist functions $\bm{\phi}_k:\mathbb{R}^{s_k}\rightarrow\mathbb{R}$, $k\in\mathbb{N}_q$ such that
$$
    \bm{\phi}(\mathbf{w})=\sum\limits_{k\in\mathbb{N}_q}\bm{\phi}_{k}(\mathbf{w}_{S_{s,k}}),\ \mathbf{w}\in\mathbb{R}^s.
$$

We now describe the block separablity of the fidelity term $\bm{\psi}$. Recall that $p_d=n$. If the partition  $\mathcal{S}_{n,d}:=\left\{S_{n,1},S_{n,2},\ldots, S_{n,d}\right\}$ for $\mathbb{N}_{n}$ is chosen with $S_{n,j}:=\{p_{j-1}+i:i\in\mathbb{N}_{m_j}\},\  j\in\mathbb{N}_d,$ then for each $j\in\mathbb{N}_d$ the sub-vector $\mathbf{u}_{S_{n,j}}$ of $\mathbf{u}\in\mathbb{R}^{n}$ coincides with $\mathbf{u}_j$. It is clear that the regularization term in problem \eqref{optimization_problem} is $\mathcal{S}_{n,d}$-block separable. We also assume that $\bm{\psi}$ is $\mathcal{S}_{n,d}$-block separable, that is, there exist functions $\bm{\psi}_{j}:\mathbb{R}^{m_j}\rightarrow\mathbb{R}_+$, $j\in\mathbb{N}_d$ such that 
\begin{equation}\label{block_separable_psi}
\bm{\psi}(\mathbf{u})=\sum\limits_{j\in\mathbb{N}_d}\bm{\psi}_{j}(\mathbf{u}_{j}),\ \mathbf{u}\in\mathbb{R}^n.
\end{equation} 
Combining the block separability of the fidelity term $\bm{\psi}$ and the norm function $\|\cdot\|_1$, the multi-parameter regularization problem \eqref{optimization_problem} can be
reduced to the following lower dimensional single-parameter regularization problems
\begin{equation}\label{optimization_problem-single-parameter}
\min
\left\{\bm{\psi}_{j}(\mathbf{u}_{j})+\lambda_j\|\mathbf{u}_{j}\|_{1}:\mathbf{u}_{j}\in\mathbb{R}^{m_{j}}\right\}, \ j\in\mathbb{N}_d.
\end{equation}
Note that the sparsity of the solution of each single parameter regularization problem \eqref{optimization_problem-single-parameter} was characterized in \cite{Liu2023parameter}. This characterization can also be derived from Theorem \ref{sparsity-special-case}.  
We further assume that functions $\bm{\psi}_{j}$, $j\in\mathbb{N}_d$ has block separability. For each $j\in\mathbb{N}_d$, let $\mathcal{S}_{m_j,q_j}:=\left\{S_{m_j,1},S_{m_j,2},\ldots, S_{m_j,q_j}\right\}$ be a partition of $\mathbb{N}_{m_j}$ and for each $k\in\mathbb{N}_{q_j}$, $s_{j,k}$ be the cardinality of $S_{m_j,k}$. For each $\mathbf{u}\in\mathbb{R}^n$, we set $\mathbf{u}_{j,k}:=(\mathbf{u}_j)_{S_{m_j,k}}$ for all $j\in\mathbb{N}_d$ and $k\in\mathbb{N}_{q_j}$. Suppose that for each $j\in\mathbb{N}_d$, $\bm{\psi}_j$ has the form
\begin{equation}\label{block_separable_psi-j}
 \bm{\psi}_j(\mathbf{u}_j)=\sum\limits_{k\in\mathbb{N}_{q_j}}\bm{\psi}_{j,k}(\mathbf{u}_{j,k}),\ \mathbf{u}_j\in\mathbb{R}^{m_j},
\end{equation} 
with $\bm{\psi}_{j,k}$ being functions from $\mathbb{R}^{s_{j,k}}$ to $\mathbb{R}_+$, $k\in\mathbb{N}_{q_j}$.  

We are ready to characterize the block sparsity of each sub-vector of the solution of problem \eqref{optimization_problem} when $\bm{\psi}$ has block separability described above. Here, we say that a vector $\mathbf{x}\in\mathbb{R}^s$ has $\mathcal{S}_{s,q}$-block
sparsity of level $l\in\mathbb{Z}_{q+1}$ if $\mathbf{x}$ has exactly $l$ number of nonzero sub-vectors with respect to partition $\mathcal{S}_{s,q}$. 
\begin{theorem}\label{sparsity-special-case-block}
Suppose that for each $j\in\mathbb{N}_d$ and each $k\in\mathbb{N}_{q_j}$, $\bm{\psi}_{j,k}:\mathbb{R}^{s_{j,k}}\to\mathbb{R}_+$ is a convex function and $\bm{\psi}_{j}$ is an $\mathcal{S}_{m_j,q_j}$-block separable function having the form \eqref{block_separable_psi-j}. Let $\bm{\psi}$ be the function with the form \eqref{block_separable_psi}. Then problem \eqref{optimization_problem}
with $\lambda_j>0$, $j\in \mathbb{N}_d$, has a solution $\mathbf{u}^{*}$, where for each $j\in \mathbb{N}_d$, $\mathbf{u}_j^{*}$ having the $\mathcal{S}_{m_j,q_j}$-block sparsity of level $l'_j\leq l_j$ for some $l_j\in\mathbb{Z}_{q_j+1}$ if and only if for each $j\in \mathbb{N}_d$ there exist distinct
$k_{j,i}\in\mathbb{N}_{q_j}$, $i\in\mathbb{N}_{l_j}$, such that
\begin{equation}\label{sparsity-equi-mini-special-block}
\lambda_j\geq\mathrm{min}\left\{\|\mathbf{y}\|_{\infty}:\mathbf{y}\in\partial\bm{\psi}_{j,k}(\mathbf{0}) \right\}, \ \mbox{for all}\ k\in\mathbb{N}_{q_j}\setminus{\{k_{j,i}:i\in\mathbb{N}_{l_j}\}}.
\end{equation}
In particular, if $\bm{\psi}_{j,k}$, $j\in\mathbb{N}_d$, $k\in\mathbb{N}_{q_j}$ are differentiable, then
condition \eqref{sparsity-equi-mini-special-block} reduces to
\begin{equation}\label{sparsity-equi-mini-special-block-diff}
    \lambda_j\geq\|\nabla\bm{\psi}_{j,k}(\mathbf{0})\|_\infty,\ \ \mbox{for all}\ \ k\in\mathbb{N}_{q_j}\setminus{\{k_{j,i}:i\in\mathbb{N}_{l_j}\}}.
\end{equation}
\end{theorem}
\begin{proof}
   Observing from the block separability of functions $\bm{\psi}$ and  $\|\cdot\|_1$, we conclude that $\mathbf{u}^*\in\mathbb{R}^n$ is a solution of problem \eqref{optimization_problem} if and only if for each $j\in\mathbb{N}_d$, $\mathbf{u}^*_j\in\mathbb{R}^{m_j}$ is a solution of problem
 \eqref{optimization_problem-single-parameter}. Theorem 3.2 in \cite{Liu2023parameter} ensures that for each $j\in\mathbb{N}_d$, $\mathbf{u}^*_j\in\mathbb{R}^{m_j}$ is a solution of problem
 \eqref{optimization_problem-single-parameter} and has the $\mathcal{S}_{m_j,q_j}$-block sparsity of level $l'_j\leq l_j$ for some $l_j\in\mathbb{Z}_{q_j+1}$ if and only if there exist distinct
$k_{j,i}\in\mathbb{N}_{q_j}$, $i\in\mathbb{N}_{l_j}$, such that
\eqref{sparsity-equi-mini-special-block} holds. For the case that $\bm{\psi}_{j,k}$, $j\in\mathbb{N}_d$, $k\in\mathbb{N}_{q_j}$ are all differentiable, it suffices to notice that the subdifferential of $\bm{\psi}_{j,k}$ at zero are the singleton $\nabla\bm{\psi}_{j,k}(\mathbf{0})$. 
This together with inequality \eqref{sparsity-equi-mini-special-block} leads to inequality 
 \eqref{sparsity-equi-mini-special-block-diff}.
\end{proof}

Unlike in Theorems \ref{sparsity_original} 
 and \ref{sparsity-special-case}, the characterization stated in Theorem \ref{sparsity-special-case-block} can be taken as a multi-parameter choice strategy. That is, when the fidelity
term is block separable, if for each $j\in\mathbb{N}_d$, the regularization parameter
$\lambda_j$ is chosen so that inequality \eqref{sparsity-equi-mini-special-block} (or \eqref{sparsity-equi-mini-special-block-diff}) holds, then the regularization problem  \eqref{optimization_problem}
has a solution with each sub-vector having a block sparsity of a prescribed level. The choice of the parameters depends on the subdifferentials
or the gradients of the functions $\bm{\psi}_{j,k}$, $j\in\mathbb{N}_d$, $k\in\mathbb{N}_{q_j}$. 

We also specialize Theorem \ref{sparsity-special-case-block} to the regularization problem
\eqref{lasso_multi-parameter}. For this purpose, we require that the fidelity
term $\bm{\psi}$ defined by \eqref{fidelity_term}
is block separable. Associated with the partition  $\mathcal{S}_{n,d}:=\left\{S_{n,1},S_{n,2},\ldots, S_{n,d}\right\}$ for $\mathbb{N}_{n}$ with $S_{n,j}:=\{p_{j-1}+i:i\in\mathbb{N}_{m_j}\}$,  $j\in\mathbb{N}_d,$ we decompose matrix $\mathbf{A}\in \mathbb{R}^{t\times n}$ into $d$ sub-matrices by setting 
$$
\mathbf{A}_{[j]}:=[\mathbf{A}_{(i)}:i\in S_{n,j}]\in\mathbb{R}^{t\times m_j},\ j\in\mathbb{N}_d.
$$
By lemma 3.4 of \cite{Liu2023parameter},
the fidelity term $\bm{\psi}$ defined by \eqref{fidelity_term}
is $\mathcal{S}_{n,d}$-block separable if and only if there holds 
\begin{equation}\label{S_nd_block_separable}
(\mathbf{A}_{[j]})^{\top}\mathbf{A}_{[k]}=\mathbf{0},\ \mbox{for all}\ j,k\in\mathbb{N}_d \ \mbox{and}\ j\neq{k}.
\end{equation}
It follows from the decomposition of $\mathbf{A}$ and that of each vector $\mathbf{u}$ in $\mathbb{R}^n$ with respect to $\mathcal{S}_{n,d}$ that $\mathbf{A}\mathbf{u}=\sum_{j\in\mathbb{N}_d}\mathbf{A}_{[j]}\mathbf{u}_j$, for all $\mathbf{u}\in\mathbb{R}^n$.
According to this equation and condition \eqref{S_nd_block_separable}, we represent  $\bm{\psi}$ defined by \eqref{fidelity_term} as in \eqref{block_separable_psi} with $\bm{\psi}_{j}$,
$j\in\mathbb{N}_d$, being defined by
\begin{equation}\label{psi_j_u_j}
\bm{\psi}_{j}(\mathbf{u}_{j})
:=\frac{1}{2}\|\mathbf{A}_{[j]}
\mathbf{u}_j\|_2^2-\mathbf{x}^{\top}\mathbf{A}_{[j]}
\mathbf{u}_j+\frac{1}{2d}\mathbf{x}^{\top}
\mathbf{x},
\  \mathbf{u}_j\in\mathbb{R}^{m_j}.
\end{equation}
To describe the block separability of functions $\bm{\psi}_{j}$, $j\in\mathbb{N}_d$, we recall that for each $j\in\mathbb{N}_d$, $\mathcal{S}_{m_j,q_j}:=\left\{S_{m_j,1},S_{m_j,2},\ldots, S_{m_j,q_j}\right\}$ is a partition of $\mathbb{N}_{m_j}$. 
Associated with the partition $\mathcal{S}_{m_j,q_j}$, matrix $\mathbf{A}_{[j]}$ can be decomposed into $q_j$ sub-matrices by setting 
$$
\mathbf{A}_{[j,k]}:=[(\mathbf{A}_{[j]})_{(i)}:i\in S_{m_j,k}]\in\mathbb{R}^{t\times s_{j,k}},\ k\in\mathbb{N}_{q_j}.
$$
It is clear that the last two terms in the right hand side of equation \eqref{psi_j_u_j} are both $\mathcal{S}_{m_j,q_j}$-block separable.
Hence, again by lemma 3.4 of \cite{Liu2023parameter}, we conclude that the functions $\bm{\psi}_{j}$ with the form \eqref{psi_j_u_j} is $\mathcal{S}_{m_j,q_j}$-block separable if and only if there holds 
\begin{equation}\label{S_mj_qj_block_separable}
(\mathbf{A}_{[j,k]})^{\top}\mathbf{A}_{[j,l]}=\mathbf{0},\ \mbox{for all}\ k,l\in\mathbb{N}_{q_j} \ \mbox{and}\ k\neq{l}.
\end{equation}
We represent $\bm{\psi}_{j}$, $j\in\mathbb{N}_d$, as in \eqref{block_separable_psi-j} when condition \eqref {S_mj_qj_block_separable} holds. For each $j\in\mathbb{N}_d$, the decomposition of $\mathbf{A}_{[j]}$ and that of each vector $\mathbf{u}_j$ in $\mathbb{R}^{m_j}$ with respect to $\mathcal{S}_{m_j,q_j}$ lead to $\mathbf{A}_{[j]}\mathbf{u}_j=\sum_{k\in\mathbb{N}_{q_j}}\mathbf{A}_{[j,k]}\mathbf{u}_{j,k}$, for all $\mathbf{u}_j\in\mathbb{R}^{m_j}$.
Substituting the above equation into definition \eqref{psi_j_u_j} with noting that condition \eqref {S_mj_qj_block_separable} holds, we represent $\bm{\psi}_{j}$ as in \eqref{block_separable_psi-j} with $\bm{\psi}_{j,k}$,
$k\in\mathbb{N}_{q_j}$, having the form
\begin{equation}\label{psi_jk_u_jk}
\bm{\psi}_{j,k}(\mathbf{u}_{j,k})
:=\frac{1}{2}\|\mathbf{A}_{[j,k]}
\mathbf{u}_{j,k}\|_2^2-\mathbf{x}^{\top}\mathbf{A}_{[j,k]}
\mathbf{u}_{j,k}+\frac{1}{2dq_j}\mathbf{x}^{\top}
\mathbf{x},
\  \mathbf{u}_{j,k}\in\mathbb{R}^{s_{j,k}}.
\end{equation}

We now apply Theorem \ref{sparsity-special-case-block} to the regularization problem \eqref{lasso_multi-parameter} when the matrix $\mathbf{A}$ satisfies conditions \eqref{S_nd_block_separable} and \eqref{S_mj_qj_block_separable}.
\begin{corollary}\label{block_separabel_parameter_choice}
Suppose that $\mathbf{x}\in\mathbb{R}^{t}$ and $\mathbf{A}\in\mathbb{R}^{t\times n}$ satisfies conditions \eqref{S_nd_block_separable} and \eqref{S_mj_qj_block_separable}. Then the regularization problem \eqref{lasso_multi-parameter} with $\lambda_j>0$, $j\in\mathbb{N}_{d}$, has a solution $\mathbf{u}^*$, where for each $j\in\mathbb{N}_{d}$, $\mathbf{u}^{*}_j$ having the $\mathcal{S}_{m_j,q_j}$-block sparsity of level $l'_j\leq l_j$ for some $l_j\in \mathbb{Z}_{q_j+1}$
if and only if for each $j\in\mathbb{N}_d$, there exist distinct $k_{j,i}\in\mathbb{N}_{q_j}$, $i\in\mathbb{N}_{l_j}$, such that $\lambda_j\geq\big\|(\mathbf{A}_{[j,k]})^{\top}\mathbf{x}\big\|_{\infty}$, for all $k\in\mathbb{N}_{q_j}\setminus{\{k_{j,i}:i\in\mathbb{N}_{l_j}\}}$.
\end{corollary}
\begin{proof}
As pointed out before, condition \eqref{S_nd_block_separable} ensures that the fidelity term $\bm{\psi}$ defined by \eqref{fidelity_term} is $\mathcal{S}_{n,d}$-block separable and has the form  \eqref{block_separable_psi} with $\bm{\psi}_{j}$,
$j\in\mathbb{N}_d$, being defined by
\eqref{psi_j_u_j}. Moreover, condition \eqref{S_mj_qj_block_separable} guarantees that for each $j\in\mathbb{N}_d$, the function $\bm{\psi}_{j}$ is $\mathcal{S}_{m_j,q_j}$-block separable and can be represented as in \eqref{block_separable_psi-j} with $\bm{\psi}_{j,k}$,
$k\in\mathbb{N}_{q_j}$, having the form \eqref{psi_jk_u_jk}. Clearly, $\bm{\psi}_{j,k},$ $j\in\mathbb{N}_d$, $k\in\mathbb{N}_{q_j}$, are all convex and differentiable functions.
Consequently, we conclude by Theorem \ref{sparsity-special-case-block} that the regularization problem \eqref{lasso_multi-parameter} with $\lambda_j>0$, $j\in\mathbb{N}_{d}$, has a solution $\mathbf{u}^*$ with for each $j\in\mathbb{N}_{d}$, $\mathbf{u}^{*}_j$ having the $\mathcal{S}_{m_j,q_j}$-block sparsity of level $l'_j\leq l_j$ for some $l_j\in \mathbb{Z}_{q_j+1}$
if and only if for each $j\in\mathbb{N}_d$ there exist distinct $k_{j,i}\in\mathbb{N}_{q_j}$, $i\in\mathbb{N}_{l_j}$, such that inequality \eqref{sparsity-equi-mini-special-block-diff} holds. Note that for each $j\in\mathbb{N}_d$, $\nabla\bm{\psi}_{j,k}(\mathbf{0})=-(\mathbf{A}_{[j,k]})^{\top}\mathbf{x}$ for all $k\in\mathbb{N}_{q_j}$. Substituting this equation into inequality \eqref{sparsity-equi-mini-special-block-diff} leads directly to the desired inequality.
\end{proof}  
%%%%%%%%%%%%%%%%%%%%%%%%%%%%%%%

\section{Sparsity-guided multi-parameter selection strategy}\label{Itera}

Theorem \ref{sparsity_original} establishes how each regularization parameter $\lambda_j$ influences the sparsity of the solution to problem \eqref{optimization_problem_under_Bj} under its associated transform matrix $\mathbf{B}_j$. Building on this result, we propose an iterative scheme for selecting multiple regularization parameters to achieve prescribed sparsity levels with respect to the different transform matrices.

%Theorem \ref{sparsity_original} characterizes the influence of each regularization parameter $\lambda_j$ on the sparsity of the solution to problem \eqref{optimization_problem_under_Bj} under the transform matrix  $\mathbf{B}_{j}$. Based on this characterization, we develop an iterative scheme in this section for selecting multiple regularization parameters that achieve prescribed sparsity levels in the solution of problem \eqref{optimization_problem_under_Bj} under different transform matrices. 

Theorem \ref{sparsity_original} shows that if $\mathbf{u}^*\in\mathbb{R}^n$ is a solution of problem \eqref{optimization_problem_under_Bj} with $\lambda_j^*>0,$ $j\in \mathbb{N}_d,$ and for each $j\in\mathbb{N}_d$, $\mathbf{u}^*$ has sparsity of level $l_j^*\in\mathbb{Z}_{m_j+1}$ under $\mathbf{B}_{j}$, then there exist  $\mathbf{a}\in\partial\bm{\psi}(\mathbf{u}^*)$ and $\mathbf{b}:=[b_j:j\in\mathbb{N}_{p_d}]\in\mathcal{N}(\mathbf{B}^{\top})$ such that for each $j\in\mathbb{N}_d$, $\lambda_j^*$ satisfies conditions \eqref{sparsity-equi-mini-2} and  \eqref{sparsity-equi-mini-3}. According to these conditions, we introduce for each $j\in\mathbb{N}_d$, a sequence $\gamma_{j,i}(\mathbf{u}^{*}),$ $i\in\mathbb{N}_{m_j}$ by  
\begin{equation}\label{rji}
\gamma_{j,i}(\mathbf{u}^{*}):=\left|(\mathbf{B}_{(p_{j-1}+i)}')^\top\mathbf{a}+b_{p_{j-1}+i}\right|, \ i\in\mathbb{N}_{m_j},
\end{equation}
and rearrange them in a nondecreasing order: 
\begin{equation}\label{nondecreasing order}
\gamma_{j, i_1}(\mathbf{u}^{*})\leq \gamma_{j, i_2}(\mathbf{u}^{*})\leq 
\cdots \leq \gamma_{j, i_{m_j}}(\mathbf{u}^{*}), \ \mbox{with}\ \{i_1,  i_2, \ldots, i_{m_j}\}=\mathbb{N}_{m_j}.
\end{equation}
The equality \eqref{sparsity-equi-mini-2} and the inequality \eqref{sparsity-equi-mini-3} that the parameter $\lambda_j^*$ needs to satisfy corresponds the non-zero components
and the zero components of $\mathbf{B}_{j}\mathbf{u}^*$,  respectively. Thus, if $\lambda_j^*>\gamma_{j, i}(\mathbf{u}^{*})$, then $(\mathbf{B}_{j}\mathbf{u}^*)_{i}$ must be zero and if $\lambda_j^*=\gamma_{j, i}(\mathbf{u}^{*})$, then $(\mathbf{B}_{j}\mathbf{u}^*)_{i}$ may be zero or nonzero. With the help of the observation above, we present the following result. 
\begin{theorem}\label{necessary—new}
Let $\bm{\psi}:\mathbb{R}^n\to\mathbb{R}_+$ be a convex function, for each $j\in\mathbb{N}_d$, $\mathbf{B}_{j}$ be an $m_j\times n$ matrix and  $\mathbf{B}$ be defined by \eqref{block-matrix-B}. Suppose that $\mathbf{u}^*$ is a solution of problem \eqref{optimization_problem_under_Bj} with $\lambda_j^*>0,$ $j\in \mathbb{N}_d$, and for each $j\in\mathbb{N}_d$, $\gamma_{j,i}(\mathbf{u}^{*}),$ $i\in\mathbb{N}_{m_j}$, defined by \eqref{rji}, are ordered as in \eqref{nondecreasing order}. Then the following statements hold true. 

(a) If for each $j\in\mathbb{N}_d$, $\mathbf{u}^*$ has sparsity of level $l_j^*\in\mathbb{Z}_{m_j+1}$ under $\mathbf{B}_{j}$, then for each $j\in\mathbb{N}_d$, $\lambda_j^*$ satisfies
\begin{equation}\label{Order}
   \gamma_{j, i_1}(\mathbf{u}^{*})\leq \cdots \leq \gamma_{j, i_{m_j-l^*_j}}(\mathbf{u}^{*})\leq\lambda_j^*=\gamma_{j, i_{m_j-l^*_j+1}}(\mathbf{u}^{*})=\cdots=\gamma_{j, i_{m_j}}(\mathbf{u}^{*}).
\end{equation} 

(b) If for each $j\in\mathbb{N}_d$, $\mathbf{u}^*$ has sparsity of level $l_j^*\in\mathbb{Z}_{m_j+1}$ under $\mathbf{B}_{j}$, then for each $j\in\mathbb{N}_d$, there exists $l_j\in\mathbb{Z}_{m_j+1}$ with $l_j\geq l_j^*$ such that $\lambda_j^*$ satisfies
\begin{equation}\label{Order-new}
   \gamma_{j, i_1}(\mathbf{u}^{*})\leq \cdots \leq \gamma_{j, i_{m_j-l_j}}(\mathbf{u}^{*})<\lambda_j^*=\gamma_{j, i_{m_j-l_j+1}}(\mathbf{u}^{*})=\cdots=\gamma_{j, i_{m_j}}(\mathbf{u}^{*}).
\end{equation}

(c) If for each $j\in\mathbb{N}_d$, there exists $l_j\in\mathbb{Z}_{m_j+1}$
such that $\lambda_j^*$ satisfies
inequality \eqref{Order-new}, then for each $j\in\mathbb{N}_d$, $\mathbf{u}^*$ has sparsity of level $l_j^*\leq l_j$ under $\mathbf{B}_{j}$.
\end{theorem}

\begin{proof}
We first prove Item (a). If $\mathbf{u}^*$ is a solution of problem \eqref{optimization_problem_under_Bj} with $\lambda_j^*>0,$ $j\in \mathbb{N}_d$, and for each $j\in\mathbb{N}_d$, $\mathbf{u}^*$ has sparsity of level $l_j^*\in\mathbb{Z}_{m_j+1}$ under $\mathbf{B}_{j}$, then for each $j\in\mathbb{N}_d$, the parameter $\lambda_j^*$, guaranteed by Theorem \ref{sparsity_original}, satisfies equality \eqref{sparsity-equi-mini-2} and inequality \eqref{sparsity-equi-mini-3}. Noting that the subset $\{k_i:i\in \mathbb{N}_{l_j^*}\}$ of $\mathbb{N}_{m_j}$ has the cardinality $l_j^*$, there are exactly $l_j^*$
elements of $\{\gamma_{j, i}(\mathbf{u}^{*}): i\in\mathbb{N}_{m_j}\}$ equal
to $\lambda_j^*$ and the remaining $m_j-l_j^*$ elements less than or equal to $\lambda_j^*$. This together with the order of  $\gamma_{j, i}(\mathbf{u}^{*})$, $i\in\mathbb{N}_{m_j}$ as in \eqref{nondecreasing order} leads to the desired inequality \eqref{Order}.

We next verify Item (b). As has been shown in Item (a), for  each $j\in\mathbb{N}_d$, $\lambda_j^*$ satisfies
inequality \eqref{Order}. 
If there is no element of the
sequence $\{\gamma_{j, i}(\mathbf{u}^{*}): i\in\mathbb{N}_{m_j}\}$ being smaller than $\lambda_j^*$, then inequality \eqref{Order} reduces to $\lambda^*_j=\gamma_{j, i_k}(\mathbf{u}^{*})$, $k\in\mathbb{N}_{m_j}$. We then get inequality \eqref{Order-new} with $l_j:=m_j$. Otherwise, we choose $k_j\in\mathbb{N}_{m_j-l^*_j}$ such that $\gamma_{j, i_{k_j}}(\mathbf{u}^{*})<\lambda_j^*=\gamma_{j, i_{k_j+1}}(\mathbf{u}^{*})$. We then rewrite inequality \eqref{Order} as inequality \eqref{Order-new} with $l_j:=m_j-k_j$. It is clear that $l_j\geq m_j-(m_j-l_j^*)=l_j^*$.

It remains to show Item (c). If $l_j=m_j,$ clearly, the sparsity level $l^*_j$ of $\mathbf{u}^*$ under $\mathbf{B}_{j}$ satisfies $l_j^*\leq l_j$. We now consider the case when $l_j<m_j$. According to Theorem \ref{sparsity_original}, the relation $\gamma_{j, i_1}(\mathbf{u}^{*})\leq \cdots \leq \gamma_{j, i_{m_j-l_j}}(\mathbf{u}^{*})<\lambda_j^*$ leads to  $(\mathbf{B}_{j}\mathbf{u}^{*})_{i_k}= 0,$ for all $k\in\mathbb{N}_{m_j-l_j}$. Hence, $\mathbf{B}_{j}\mathbf{u}^{*}$ has at least $m_j-l_j$ zero components. In other words, the number of nonzero components of $\mathbf{B}_{j}\mathbf{u}^{*}$ is at most $l_j$, that is, $\mathbf{u}^*$ has sparsity of level $l_j^*\leq l_j$ under $\mathbf{B}_{j}$. 
\end{proof}

Our goal is to find regularization parameters $\lambda^*_j$, $j\in\mathbb{N}_d$ that ensures the resulting solution $\mathbf{u}^*$ of problem \eqref{optimization_problem_under_Bj} achieves a prescribed sparsity level $l^*_j$ under each transform matrix $\mathbf{B}_{j}$. According to Item (a) of Theorem \ref{necessary—new}, for each $j\in\mathbb{N}_d$, the parameter $\lambda_j^*$ satisfies inequality \eqref{Order}. Since the sequence $\{\gamma_{j, i}(\mathbf{u}^{*}): i\in\mathbb{N}_{m_j}\}$ depends on the corresponding
solution, inequality \eqref{Order} can not be used directly as a parameter
choice strategy. Instead, it motivates us to propose an iterative scheme. The iteration begins with
initial regularization parameters $\lambda_j^0$, $j\in\mathbb{N}_d$ which are large enough so that for each $j\in\mathbb{N}_d$, the
sparsity level $l_j^0$ of the corresponding solution $\mathbf{u}^0$ of problem \eqref{optimization_problem_under_Bj} under $\mathbf{B}_{j}$ is smaller than the given target sparsity level $l^*_j$. Suppose that at step $k$, we
have $\lambda_j^k$, $j\in\mathbb{N}_d$
and the corresponding solution $\mathbf{u}^k$ with the sparsity level $l_j^k<l^*_j$, $j\in\mathbb{N}_d$ under the transform matrices $\mathbf{B}_{j}$, $j\in\mathbb{N}_d$, respectively. Item (a) of Theorem \ref{necessary—new} ensures that for each $j\in\mathbb{N}_d$, parameter $\lambda_j^k$  satisfies \begin{equation}\label{Order-algorithm}
   \gamma_{j, i^{k}_1}(\mathbf{u}^{k})\leq \cdots \leq \gamma_{j, i^{k}_{m_j-l^k_j}}(\mathbf{u}^{k})\leq\lambda_j^k=\gamma_{j, i^{k}_{m_j-l^k_j+1}}(\mathbf{u}^{k})=\cdots=\gamma_{j, i^{k}_{m_j}}(\mathbf{u}^{k}).
\end{equation} 
We choose parameter $\lambda_j^{k+1}$ at step $k+1$ from the ordered sequence in \eqref{Order-algorithm}. Motivated by inequality \eqref{Order}, we choose $\lambda_j^{k+1}$ as the $(m_j-l^*_j+1)$-th element of the ordered sequence in \eqref{Order-algorithm}, that is,
\begin{equation}\label{iteration-update}
\lambda_j^{k+1}
    :=\gamma_{j, i^{k}_{m_j-l^*_j+1}}(\mathbf{u}^{k}).
\end{equation}
As a result, for each $j\in\mathbb{N}_d$,  parameter $\lambda_j^{k+1}$ satisfies 
\begin{align}\label{Order-algorithm_1}
    \gamma_{j, i^{k}_1}(\mathbf{u}^{k})\leq \cdots \leq \gamma_{j, i^{k}_{m_j-l^*_j}}(\mathbf{u}^{k})\leq\lambda_j^{k+1}
    &=\gamma_{j, i^{k}_{m_j-l^*_j+1}}(\mathbf{u}^{k})\nonumber\\
    &\leq\cdots\leq\gamma_{j, i^{k}_{m_j-l^k_j+1}}(\mathbf{u}^{k})=\cdots=\gamma_{j, i^{k}_{m_j}}(\mathbf{u}^{k}).
\end{align}

Below, we claim that if the iterative scheme converges, then the generated parameters satisfy inequality \eqref{Order} of Theorem \ref{necessary—new}, which is a necessary condition for the resulting solution to attain the target sparsity levels. 
The convergence analysis of the proposed multi-parameter selection scheme
will be addressed in future work. 
We state the assumptions about the convergence of the iterative scheme as follows. 

(A1) For each $j\in\mathbb{N}_d$, the sequence $l_j^k$, $k\in\mathbb{N}$,
generated by the iteration scheme satisfies that $l_j^k\leq l^*_j$ for all $k\in\mathbb{N}$.

(A2) For each $j\in\mathbb{N}_d$, the sequences $\lambda_j^k$, $k\in\mathbb{N}$  generated by the iteration scheme satisfies that $\lambda_j^k\rightarrow\lambda_j$ as $k\rightarrow+\infty$ for some $\lambda_j>0$.

(A3) The solution $\mathbf{u}\in\mathbb{R}^n$ of problem \eqref{optimization_problem_under_Bj} with $\lambda_j$, $j\in\mathbb{N}_d$ satisfies that for each $j\in\mathbb{N}_d$ and each $s\in\mathbb{N}_{m_j}$, $\gamma_{j, i^{k}_s}(\mathbf{u}^{k})\rightarrow\gamma_{j, i_s}(\mathbf{u})$ as $k\rightarrow+\infty$.
\begin{proposition}
    If assumptions (A1), (A2) and (A3) hold, then for each $j\in\mathbb{N}_d$,  $\lambda_j$ satisfies
    \begin{equation}\label{Order-algorithm_2}
    \gamma_{j, i_1}(\mathbf{u})\leq \cdots \leq \gamma_{j, i_{m_j-l^*_j}}(\mathbf{u})\leq\lambda_j
    =\gamma_{j, i_{m_j-l^*_j+1}}(\mathbf{u})=\cdots=\gamma_{j, i_{m_j}}(\mathbf{u}).
\end{equation}
\end{proposition}
\begin{proof}
Note that assumption (A1)   allows us to choose parameter $\lambda_j^{k+1}$ 
 as in \eqref{iteration-update} at step $k+1$. This together with inequality \eqref{Order-algorithm} leads to inequality 
\eqref{Order-algorithm_1}. 
By taking $k\rightarrow+\infty$ on each item of inequality \eqref{Order-algorithm_1} and assumptions (A2) and (A3), we get that 
\begin{equation}\label{Order-algorithm_3}
 \gamma_{j, i_1}(\mathbf{u})\leq \cdots \leq \gamma_{j, i_{m_j-l^*_j}}(\mathbf{u})\leq\lambda_j
    =\gamma_{j, i_{m_j-l^*_j+1}}(\mathbf{u})\leq\cdots\leq\gamma_{j, i_{m_j}}(\mathbf{u}).
\end{equation}
It follows from inequality \eqref{Order-algorithm} that $\lambda_j^k=\gamma_{j, i^{k}_{m_j}}(\mathbf{u}^{k})$ for all $k\in\mathbb{N}$. Taking $k\rightarrow+\infty$ on both sides of this equation, we get that $\lambda_j=\gamma_{j, i_{m_j}}(\mathbf{u}).$  
This together with  inequality \eqref{Order-algorithm_3} leads to the desired inequality \eqref{Order-algorithm_2}.
\end{proof}

The implementation of the iterative scheme involves four key considerations. 

\begin{itemize} \item[\textbf{(1)}]\textbf{Correction of invalid updates} 
 
\vspace{1.2ex}
When $\lambda_j^{k}
=\gamma_{j, i^{k}_{m_j-l^*_j+1}}(\mathbf{u}^{k})$ for some $k\in\mathbb{N}$ and $j\in\mathbb{N}_d$, the update rule \eqref{iteration-update} yields $\lambda_j^{k+1}=\lambda_j^{k}$, resulting in an invalid update. To resolve this, we propose selecting $\lambda^{k+1}_j$ from the set $\{\gamma_{j,i}(\mathbf{u}^{k}):\gamma_{j,i}(\mathbf{u}^{k})<\lambda_j^{k}\}$, motivated by inequality \eqref{Order-new}. Specifically, we define  
$$
\Gamma_j(\mathbf{u}^{k}):=\max\{\gamma_{j,i}(\mathbf{u}^{k}):\gamma_{j,i}(\mathbf{u}^{k})<\lambda_j^{k}\}
$$ 
and update  $\lambda^{k+1}_j$ via
\begin{equation}\label{iteration-update-1}
\lambda^{k+1}_j:=\min\{\gamma_{j, i^{k}_{m_j-l^*_j+1}}(\mathbf{u}^{k}),\Gamma_j(\mathbf{u}^{k})\}.
\end{equation}
 
\item[\textbf{(2)}]\textbf{Fallback for undersized parameters} 
 
\vspace{1.2ex}
Assumption (A1) may not always hold. Suppose that  
 $l_j^k> l^*_j$ for some $k\in\mathbb{N}$ and $j\in\mathbb{N}_d$. This indicates that the parameter $\lambda_j^k$
is too small to achieve the desired sparsity level $l^*_j$.  To correct this, the next parameter $\lambda_j^{k+1}$ must be selected larger than $\lambda_j^k$. However, as shown in inequality \eqref{Order-algorithm}, all elements $\gamma_{j, i^{k}_s}(\mathbf{u}^{k})$, $s\in\mathbb{N}_{m_j}$,
are less than or equal to
$\lambda_j^k$. Consequently, the update rule \eqref{iteration-update-1} cannot produce a value exceeding $\lambda_j^k$. This means that simply proceeding to the next iteration will not sufficiently increase the performance. To resolve this, we must revert to step $k-1$ and select a new $\lambda_j^k$ that is appropriately larger. This adjustment ensures that subsequent iterations can achieve the target sparsity level $l^*_j$.

\vspace{1.2ex}
\item[\textbf{(3)}]\textbf{Tolerance for sparsity mismatches} 
 
\vspace{1.2ex}
Due to the interplay among the 
multiple regularization parameters, we do not require exact matching of sparsity levels but instead allow a tolerance error. For each $j\in\mathbb{N}_d$, let $l_j$ denote the sparsity level of a solution of problem \eqref{optimization_problem_under_Bj} under $\mathbf{B}_j$. With a given tolerance $\epsilon>0$, we say
that the solution achieves target sparsity levels $l^*_j$, $j\in\mathbb{N}_d$ if 
$$
\sum_{j\in\mathbb{N}_d}|l_j-l_j^*|\leq \epsilon.
$$

\item[\textbf{(4)}]\textbf{Efficient computation of critical vectors} 
 
\vspace{1.2ex}
At step $k+1$ of the iterative scheme, selecting the parameter $\lambda_j^{k+1}$ from the ordered candidate sequence in \eqref{Order-algorithm} requires computing  $\gamma_{j, i}(\mathbf{u}^{k})$, $i\in\mathbb{N}_{m_j}$, as defined in \eqref{rji}. This computation involves two main tasks: first, solving problem \eqref{optimization_problem_under_Bj} with $\lambda_j^k>0,$ $j\in \mathbb{N}_d,$ to obtain the solution $\mathbf{u}^k$; and second, determining vectors $\mathbf{a}\in\partial\bm{\psi}(\mathbf{u}^k)$ and $\mathbf{b}\in\mathcal{N}(\mathbf{B}^\top)$ that satisfy \eqref{sparsity-equi-mini-1}, \eqref{sparsity-equi-mini-2} and \eqref{sparsity-equi-mini-3}. Thus, efficient computing $\mathbf{u}^{k}$, $\mathbf{a}$, and $\mathbf{b}$ is crucial for implementing the iterative scheme. In certain special cases, the vectors $\mathbf{a}$ and $\mathbf{b}$ can be determined directly. For instance, when the fidelity term $\bm{\psi}$ is differentiable,  $\mathbf{a}$ reduces to $\nabla\bm{\psi}(\mathbf{u}^k)$. Similarly, if $\mathbf{B}$ has full row rank, $\mathbf{b}$ simplifies to $\mathbf{0}$. However, in the general case where $\bm{\psi}$ is non-differentiable or $\mathbf{B}$ lacks full row rank, explicit expressions for $\mathbf{a}$ or $\mathbf{b}$ in terms of $\mathbf{u}^k$ are unavailable. To address this, we present the numerical procedures for computing the three vectors, with their theoretical derivations deferred to Section \ref{fixed}. Importantly, our proposed method remains universally applicable across all scenarios. 

\quad\ \ We propose a fixed-point proximity algorithm (Algorithm \ref{FPPA}) for solving the reformulated problem \eqref{optimization_problem_under_Bj_3}  with $\lambda_j$, $j\in\mathbb{N}_d$. This algorithm returns a solution $\mathbf{w}\in\mathbb{R}^{p_d+n-r}$ along with auxiliary vectors $\hat{\mathbf{a}}\in\mathbb{R}^n$ and $\mathbf{c}\in\mathbb{R}^{p_d+n-r}$. The solution $\mathbf{u}^*$ of problem \eqref{optimization_problem_under_Bj} with $\lambda_j$, $j\in\mathbb{N}_d$, along with the vectors $\mathbf{a}\in\partial\bm{\psi}(\mathbf{u}^*)$ and $\mathbf{b}\in\mathcal{N}(\mathbf{B}^\top)$ satisfying \eqref{sparsity-equi-mini-1}, \eqref{sparsity-equi-mini-2} and \eqref{sparsity-equi-mini-3}, can then be derived as follows: $\mathbf{u}^*=\mathbf{B}'\mathbf{w}$, $\mathbf{a}=\hat{\mathbf{a}}$, and $\mathbf{b}$ corresponds to the first $p_d$ components of $\mathbf{c}$ (i.e., a subvector of $\mathbf{c}$). In addition, partitioning $\mathbf{w}$ as  $\mathbf{w}:=\scalebox{0.8}{$\begin{bmatrix}
\mathbf{z}\\
\mathbf{v}\end{bmatrix}$}$, where $\mathbf{z}\in\mathbb{R}^{p_d}$ and $\mathbf{v}\in\mathbb{R}^{n-r}$, we observe that for each $j\in\mathbb{N}_d$, the sparsity level of the subvector $\mathbf{z}_j$ matches with that of $\mathbf{u}^*$ under the transform matrix $\mathbf{B}_j$. Let $\mathbb{S}_{+}^s$ denote the set of symmetric and positive definite matrices of order $s\times s$. To ensure the convergence of Algorithm \ref{FPPA}, the matrices $\mathbf{O}\in\mathbb{S}_{+}^{p_d+n-r}$, $\mathbf{P}\in\mathbb{S}_{+}^{n}$, $\mathbf{Q}\in\mathbb{S}_{+}^{p_d+n-r}$ and $\theta\in\mathbb{R}$ require to satisfy 
\eqref{satisfy-M-condition}
and 
\eqref{satisfy-M-condition-H}. 

% FPPA
\begin{algorithm}
  \caption{Fixed-Point Proximity Algorithm for problem \eqref{optimization_problem_under_Bj_3}}
  
  \label{FPPA}
  
  \KwInput{$\bm{\psi}$, $\mathbf{B}'$, $\{\mathbf{I}^{'}_{j}:j\in\mathbb{N}_d\}$,$\{\lambda_{j}:j\in\mathbb{N}_d\}$, $\mathbb{M}$; $\mathbf{O}\in\mathbb{S}_{+}^{p_d+n-r}$, $\mathbf{P}\in\mathbb{S}_{+}^{n}$,$\mathbf{Q}\in\mathbb{S}_{+}^{p_d+n-r}$, $\theta>0$ satisfying \eqref{satisfy-M-condition} and \eqref{satisfy-M-condition-H}.}
  
  \KwInitialization{$\mathbf{w}^0$, $\hat{\mathbf{a}}^0$, $\mathbf{c}^0$.}
  \For{$k = 0,1,2,\ldots$}
  {$\mathbf{w}^{k+1} \gets  \mathrm{prox}_{\sum_{j\in\mathbb{N}_d}\lambda_j\|\cdot\|_{1}\circ\mathbf{I}^{'}_{j},\mathbf{O}}
 \big(\mathbf{w}^{k}+\mathbf{O}^{-1}(\mathbf{B}')^{\top}((\theta-2)\hat{\mathbf{a}}^{k}+(1-\theta)\hat{\mathbf{a}}^{k-1})+\mathbf{O}^{-1}((\theta-2)\mathbf{c}^{k}+(1-\theta)\mathbf{c}^{k-1})\big)$.\\
 $\hat{\mathbf{a}}^{k+1} \gets \mathrm{prox}_{\bm{\psi}^*,\mathbf{P}}
\big(\hat{\mathbf{a}}^{k}+\mathbf{P}^{-1}\mathbf{B}'(\mathbf{w}^{k+1}+\theta(\mathbf{w}^{k+1}-\mathbf{w}^{k}))\big)$.\\
$\mathbf{c}^{k+1} \gets \mathrm{prox}_{\iota_{\mathbb{M}}^*,\mathbf{Q}}\big(\mathbf{c}^{k}+\mathbf{Q}^{-1}(\mathbf{w}^{k+1}+\theta(\mathbf{w}^{k+1}-\mathbf{w}^{k})\big)$.
 }
\KwOutput{$\mathbf{w}^{k+1}$, $\hat{\mathbf{a}}^{k+1}$, $\mathbf{c}^{k+1}$}  
\end{algorithm}

\end{itemize}

By leveraging the update rule \eqref{iteration-update} and incorporating the aforementioned four considerations, we provide in Algorithm \ref{Sparsity-Guided Multi-Parameter Selection Algorithm} an iterative
scheme for selecting multiple regularization parameters that achieve prescribed sparsity
levels in the solution of problem (2.1) under different transform matrices. 

% 算法2
\begin{algorithm}
  \caption{Sparsity-Guided Multi-Parameter Selection Algorithm for problem \eqref{optimization_problem_under_Bj}}
  
  \label{Sparsity-Guided Multi-Parameter Selection Algorithm}
  
  \KwInput{$\mathbf{B}'$, $\bm{\psi}$, $\{l_j^*:j\in\mathbb{N}_{d}\}$, $\epsilon$.}
  
  \KwInitialization{Choose $\{\lambda^0_j:j\in\mathbb{N}_{d}\}$ large enough that guarantees $l^0_j\leq l_j^*$ for all $j\in\mathbb{N}_{d}$.}
  \For{$k = 0,1,2,\ldots$}
  {Solve \eqref{optimization_problem_under_Bj_3} with $\lambda_j^k$, $j\in\mathbb{N}_d$ by Algorithm \ref{FPPA}, get the vectors $\mathbf{w}^{k}=\scalebox{0.9}{$\begin{bmatrix}
\mathbf{z}^{k}\\
\mathbf{v}^{k}\end{bmatrix}$}$, $\hat{\mathbf{a}}^{k}$ and $\mathbf{c}^{k}$, and count the sparsity level $l^{k}_j$ of $\mathbf{z}^{k}_j$.\\
  \If{$\sum_{j\in\mathbb{N}_d}|l_j^k-l_j^*|\leq \epsilon $} {
    \textbf{break}  % This exits the loop early
  }
  
  \For{$j = 1,2,\ldots,d$}
  {
    \uIf{$l^k_j<l^*_j$}{Compute $\gamma_{j,i}:=\big|(\mathbf{B}_{(p_{j-1}+i)}')^\top\hat{\mathbf{a}}^{k}+c^{k}_{p_{j-1}+i})\big|$, $i\in\mathbb{N}_{m_j}$.\\
  Sort: $\gamma_{j,i_1}\leq \cdots \leq \gamma_{j,i_{m_j}}$ with $\{i_1,i_2,\cdots,i_{m_j}\}=\mathbb{N}_{m_j}$.\\
  Compute $\Gamma_j:=\mathrm{max}\big\{\gamma_{j,i}:\gamma_{j,i}< \lambda^{k}_j, i\in\mathbb{N}_{m_j}\big\}$.\\
  Update $\lambda^{k}_j:=\min\big\{\gamma_{j,i_{m_j-{l_j^*}+1}},\Gamma_j\big\}$.
  }
 \ElseIf{$l^k_j>l^*_j$}
  {Set $s_j=0$.\\
  \For{$i = 1,2,\ldots$}
  {Update $s_j:=s_j+l^{k}_j-l^{*}_j$.\\ Update $\lambda_{j}^{k}:=\min\big\{\gamma_{j,i_{m_j-{l_j^*}+1+s_j}},\Gamma_j\big\}$.\\
  Solve \eqref{optimization_problem_under_Bj_3} with $\lambda_j^k$, $j\in\mathbb{N}_d$ by Algorithm \ref{FPPA}, get the vectors $\mathbf{w}^{k}$, $\hat{\mathbf{a}}^{k}$ and $\mathbf{c}^{k}$, and count the sparsity level $l^{k}_j$ of $\mathbf{z}^{k}_j$.\\
  \If{$l^k_j\leq l_j^*$} {\textbf{break}}
  % This exits the loop early
  }
  }
  }
  {Update 
 $\lambda^{k+1}_j$ as $\lambda^{k}_j$ for all $j\in\mathbb{N}_d$.}
  }
  
  \KwOutput{$\{\lambda_j^{k}:j\in\mathbb{N}_d\}$, $\mathbf{w}^{k}$.}  
\end{algorithm}

\section{Fixed-point proximity algorithm}\label{fixed}

In this section, we develop a numerical method based on a fixed-point proximity framework to compute the solution $\mathbf{u}^*$ and the auxiliary vectors $\mathbf{a}$ and $\mathbf{b}$ described in Theorem \ref{sparsity_original}. Motivated by the effectiveness of fixed-point methods in machine learning \cite{Li2020fixed,Li2019a,Polson2015proximal}, image processing \cite{li2015multi,micchelli2011proximity}, and medical imaging \cite{Krol2012preconditioned,zheng2019sparsity}, we introduce a problem-specific fixed-point proximity algorithm tailored to our setting. We also provide a rigorous proof of its convergence.

%In this section, we develop numerical procedures based on a fixed-point proximity approach to compute the solution $\mathbf{u}^{*}$ and the auxiliary vectors $\mathbf{a}$, $\mathbf{b}$ specified in Theorem \ref{sparsity_original}. Inspired by the success of fixed-point approaches used in machine learning \cite{Li2020fixed,Li2019a,Polson2015proximal}, image processing \cite{li2015multi,micchelli2011proximity} and medical imaging \cite{Krol2012preconditioned,zheng2019sparsity}, we propose a tailored fixed-point proximity algorithm designed for our problem’s structure. We rigorously prove the convergence of the proposed algorithm.

We begin by recalling the definition of the proximity operator and several useful properties. For $\mathbf{H}\in\mathbb{S}_{+}^s$, we define the weighted inner product of $\mathbf{x}$, $\mathbf{y}\in\mathbb{R}^s$ by $\langle \mathbf{x},\mathbf{y}\rangle_{\mathbf{H}}:=\langle \mathbf{x},\mathbf{H}\mathbf{y}\rangle$ and the weighted $\ell_2$-norm of $\mathbf{x}\in\mathbb{R}^s$ by $\|\mathbf{x}\|_{\mathbf{H}}:=\langle \mathbf{x},\mathbf{x}\rangle_{\mathbf{H}}^{1/2}$. Suppose that $f:\mathbb{R}^s\to \overline{\mathbb{R}}$ is a convex function, with $\mathrm{dom}(f)\neq{\emptyset}.$ The proximity operator $\text{prox}_{f,\mathbf{H}}:\mathbb{R}^s\to\mathbb{R}^s$ of $f$ with respect to $\mathbf{H}\in\mathbb{S}_{+}^s$ is defined for $\mathbf{w}\in\mathbb{R}^s$ by 
\begin{equation}\label{proximity operator}
\text{prox}_{f,\mathbf{H}}(\mathbf{x}):=\argmin\left\{\frac{1}{2}\|\mathbf{y}-\mathbf{x}\|_{\mathbf{H}}^2+f(\mathbf{y}):\mathbf{y}\in\mathbb{R}^s\right\}.
\end{equation}
In the case that $\mathbf{H}$ coincides with the $s\times s$ identity matrix $\mathbf{I}_s$, $\text{prox}_{f,\mathbf{I}_s}$ will be
abbreviated as  $\text{prox}_{f}$. It is known \cite{micchelli2011proximity} that the proximity operator of a convex function is intimately related to its subdifferential. Specifically, if $f$ is a convex function from $\mathbb{R}^s$ to $\overline{\mathbb{R}}$, then for all $\mathbf{x}\in\mathrm{dom}(f)$, $\mathbf{y}\in\mathbb{R}^s$ and  $\mathbf{H}\in\mathbb{S}_{+}^s$
\begin{equation}\label{relation_prox_subdiff}
\mathbf{Hy}\in\partial f(\mathbf{x}) \ 
\text{if and only if} \ \mathbf{x}=\mathrm{prox}_{f,\mathbf{H}}(\mathbf{x}+\mathbf{y}).    
\end{equation}
The conjugate function of a convex function $f:\mathbb{R}^s\to \overline{\mathbb{R}}$, denoted by $f^* : \mathbb{R}^s \to \overline{\mathbb{R}}$, is defined by 
$$
f^*(\mathbf{y}):=\sup\{\langle \mathbf{x},\mathbf{y}\rangle-f(\mathbf{x}):\mathbf{x}\in\mathbb{R}^s\}, \ \ \mbox{for all}\ \ \mathbf{y}\in\mathbb{R}^s.
$$
There is a duality relationship between the subdifferentials of $f$ and its conjugate $f^*$: for all $\mathbf{x} \in \mathrm{dom}(f)$ and $\mathbf{y} \in \mathrm{dom}(f^*)$,
\begin{equation}\label{sub_conjugate}
\mathbf{x}\in \partial f^*(\mathbf{y}) \  \text{if and only if} \ \mathbf{y}\in\partial f(\mathbf{x}).
\end{equation}
This yields a fundamental identity between the proximity operators of $f$ and its conjugate $f^*$:
$$
\mathrm{prox}_{f}=\mathbf{I}_s-\mathrm{prox}_{f^*}.
$$
    
% We also need the identities 
% $$
% \mathrm{prox}_{\frac{1}{\alpha} f^*}
% =\frac{1}{\alpha} (\mathcal{I}-\mathrm{prox}_{\alpha f})(\alpha\mathcal{I}) \ \
% \text{and} 
% \ \ 
% \mathrm{prox}_{f^*,\alpha\mathcal{I}}=\mathrm{prox}_{\frac{1}{\alpha}f^*}.
% $$

We establish in the following proposition the
fixed-point equation formulation of the solution to problem \eqref{optimization_problem_under_Bj_3}.

\begin{proposition}\label{FPPA_deformation}
Suppose that $\bm{\psi}:\mathbb{R}^n\to\mathbb{R}_+$ is a convex function and for each $j\in\mathbb{N}_d$, $\mathbf{B}_{j}$ is an $m_j\times n$ matrix. Let $\mathbf{B}$ be defined as in \eqref{block-matrix-B}. Then the following statements hold true.

(a) If $\mathbf{w}\in\mathbb{R}^{p_d+n-r}$ is a solution of problem \eqref{optimization_problem_under_Bj_3} with $\lambda_j$, $j\in\mathbb{N}_d$, then there exist vectors $\hat{\mathbf{a}}\in\mathbb{R}^n$ and $\mathbf{c}\in\mathbb{R}^{p_d+n-r}$ satisfying 
\begin{align}
\mathbf{w}&=\mathrm{prox}_{\sum_{j\in\mathbb{N}_d}\lambda_j\|\cdot\|_{1}\circ\mathbf{I}^{'}_{j},\mathbf{O}}\left(\mathbf{w}-\mathbf{O}^{-1}(\mathbf{B}')^{\top}\hat{\mathbf{a}}-\mathbf{O}^{-1}\mathbf{c}\right),  \label{FPPA_nondiff1}\\
\hat{\mathbf{a}}&=\mathrm{prox}_{\bm{\psi}^*,\mathbf{P}}(\mathbf{P}^{-1}\mathbf{B}'\mathbf{w}+\hat{\mathbf{a}}), \label{FPPA_nondiff2}\\
\mathbf{c}&=\mathrm{prox}_{\iota_{\mathbb{M}}^*,\mathbf{Q}}(\mathbf{Q}^{-1}\mathbf{w}+\mathbf{c}), \label{FPPA_nondiff3}
\end{align}
for any matrices $\mathbf{O}\in\mathbb{S}_{+}^{p_d+n-r}$, $\mathbf{P}\in\mathbb{S}_{+}^{n}$ and $\mathbf{Q}\in\mathbb{S}_{+}^{p_d+n-r}$.

(b) If there exist vectors $\mathbf{w}\in\mathbb{R}^{p_d+n-r}$, $\hat{\mathbf{a}}\in\mathbb{R}^n$, $\mathbf{c}\in\mathbb{R}^{p_d+n-r}$ and matrices $\mathbf{O}\in\mathbb{S}_{+}^{p_d+n-r}$, $\mathbf{P}\in\mathbb{S}_{+}^{n}$, $\mathbf{Q}\in\mathbb{S}_{+}^{p_d+n-r}$ satisfying equations \eqref{FPPA_nondiff1}, \eqref{FPPA_nondiff2} and \eqref{FPPA_nondiff3}, then $\mathbf{w}$ is a solution of problem \eqref{optimization_problem_under_Bj_3} with $\lambda_j$, $j\in\mathbb{N}_d$.
\end{proposition}
\begin{proof}
According to Fermat rule and the chain rule \eqref{chain-rule} of the subdifferential, we have that 
$\mathbf{w}\in\mathbb{R}^{p_d+n-r}$ is a solution of problem \eqref{optimization_problem_under_Bj_3} if and only if  
\begin{equation*}\label{nondiff-Fermat-rule-chain-rule}
\mathbf{0}\in(\mathbf{B}')^\top\partial\bm{\psi}(\mathbf{B}'\mathbf{w})+\partial\iota_{\mathbb{M}}(\mathbf{w})+\partial\left(\sum_{j\in\mathbb{N}_d}\lambda_j\|\cdot\|_{1}\circ\mathbf{I}^{'}_{j}\right)(\mathbf{w}).
\end{equation*}
The latter is equivalent to that there exist  $\hat{\mathbf{a}}\in\partial\bm{\psi} (\mathbf{B}'\mathbf{w})$ and $\mathbf{c}\in\partial\iota_{\mathbb{M}}(\mathbf{w})$ such that 
\begin{equation}\label{nondiff-Fermat-rule-chain-rule1}
-(\mathbf{B}')^{\top}\hat{\mathbf{a}}-\mathbf{c}\in\partial\left(\sum_{j\in\mathbb{N}_d}\lambda_j\|\cdot\|_{1}\circ\mathbf{I}^{'}_{j}\right)(\mathbf{w}).
\end{equation}

We first prove Item (a). If $\mathbf{w}\in\mathbb{R}^{p_d+n-r}$ is a solution of problem \eqref{optimization_problem_under_Bj_3}, then there exist $\hat{\mathbf{a}}\in\mathbb{R}^n$ and $\mathbf{c}\in\mathbb{R}^{p_d+n-r}$ satisfying inclusion relation \eqref{nondiff-Fermat-rule-chain-rule1}, which further leads to  
$$
\mathbf{O}(-\mathbf{O}^{-1}(\mathbf{B}')^{\top}\hat{\mathbf{a}}-\mathbf{O}^{-1}\mathbf{c})\in\partial\left(\sum_{j\in\mathbb{N}_d}\lambda_j\|\cdot\|_{1}\circ\mathbf{I}^{'}_{j}\right)(\mathbf{w})
$$ 
for any $\mathbf{O}\in\mathbb{S}_{+}^{p_d+n-r}$. Relation \eqref{relation_prox_subdiff} ensures the equivalence between the inclusion relation above and equation \eqref{FPPA_nondiff1}. According to relation \eqref{sub_conjugate}, we rewrite the inclusion relation $\hat{\mathbf{a}}\in\partial\bm{\psi}(\mathbf{B}'\mathbf{w})$ as $\mathbf{B}'\mathbf{w}\in\partial\bm{\psi}^*(\hat{\mathbf{a}})$, which further leads to $\mathbf{P}(\mathbf{P}^{-1}\mathbf{B}'\mathbf{w})\in\partial\bm{\psi}^*(\hat{\mathbf{a}})$ for any $\mathbf{P}\in\mathbb{S}_{+}^{n}$. This guaranteed by relation \eqref{relation_prox_subdiff} is equivalent to equation \eqref{FPPA_nondiff2}. Again by relation \eqref{sub_conjugate}, the inclusion relation   $\mathbf{c}\in\partial\iota_{\mathbb{M}}(\mathbf{w})$ can be rewritten as $\mathbf{w}\in\partial\iota_{\mathbb{M}}^*(\mathbf{c})$. Hence, for any $\mathbf{Q}\in\mathbb{S}_{+}^{p_d+n-r}$, we obtain that $\mathbf{Q}(\mathbf{Q}^{-1}\mathbf{w})\in\partial\iota_{\mathbb{M}}^*(\mathbf{c})$ which guaranteed by relation \eqref{relation_prox_subdiff} is equivalent to equation \eqref{FPPA_nondiff3}.

We next verify Item (b). Suppose that vectors $\mathbf{w}\in\mathbb{R}^{p_d+n-r}$, $\hat{\mathbf{a}}\in\mathbb{R}^n$, $\mathbf{c}\in\mathbb{R}^{p_d+n-r}$ and matrices $\mathbf{O}\in\mathbb{S}_{+}^{p_d+n-r}$, $\mathbf{P}\in\mathbb{S}_{+}^{n}$, $\mathbf{Q}\in\mathbb{S}_{+}^{p_d+n-r}$ satisfying equations \eqref{FPPA_nondiff1}, \eqref{FPPA_nondiff2} and \eqref{FPPA_nondiff3}. As pointed out in the proof of Item (a), equations \eqref{FPPA_nondiff2} and \eqref{FPPA_nondiff3} are equivalent to inclusion relations $\hat{\mathbf{a}}\in\partial\bm{\psi} (\mathbf{B}'\mathbf{w})$ and $\mathbf{c}\in\partial\iota_{\mathbb{M}}(\mathbf{w})$, respectively. Moreover, equation \eqref{FPPA_nondiff1} are equivalent to inclusion relation \eqref{nondiff-Fermat-rule-chain-rule1}. Consequently, we conclude that $\mathbf{w}$ is a solution of problem \eqref{optimization_problem_under_Bj_3}.
\end{proof}

Proposition \ref{FPPA_deformation} provides the fixed-point equation formulations not only for the solution $\mathbf{w}$ of problem \eqref{optimization_problem_under_Bj_3} but also for two auxiliary vectors $\hat{\mathbf{a}}$ and $\mathbf{c}$. We now derive the correspondence between these vectors $\mathbf{w}$, $\hat{\mathbf{a}}$ and $\mathbf{c}$ and the target vectors $\mathbf{u}^*$, $\mathbf{a}$ and $\mathbf{b}$. 
\begin{theorem}\label{a-b}
Suppose that $\bm{\psi}:\mathbb{R}^n\to\mathbb{R}_+$ is a convex function and for each $j\in\mathbb{N}_d$, $\mathbf{B}_{j}$ is an $m_j\times n$ matrix. Let $\mathbf{B}$ be defined as in \eqref{block-matrix-B}. Suppose that $\mathbf{w}:=\scalebox{0.8}{$\begin{bmatrix}
\mathbf{z}\\
\mathbf{v}\end{bmatrix}$}\in\mathbb{R}^{p_d+n-r}$ with $\mathbf{v}\in\mathbb{R}^{n-r}$, $\mathbf{z}\in\mathbb{R}^{p_d}$ and for each $j\in \mathbb{N}_d,$ $\mathbf{z}_j=\sum_{i\in\mathbb{N}_{l_j}}z_{p_{j-1}+k_i}^*\mathbf{e}_{m_j,k_i}\in \Omega_{m_j,l_j}$ for some $l_{j}\in\mathbb{Z}_{m_j+1}$, and in addition $\hat{\mathbf{a}}\in\mathbb{R}^n$,  $\mathbf{c}:=[c_j:j\in\mathbb{N}_{p_d+n-r}]\in\mathbb{R}^{p_d+n-r}$. If vectors $\mathbf{w}$, $\hat{\mathbf{a}}$ and $\mathbf{c}$ satisfy equations \eqref{FPPA_nondiff1}, \eqref{FPPA_nondiff2} and \eqref{FPPA_nondiff3} for some matrices $\mathbf{O}\in\mathbb{S}_{+}^{p_d+n-r}$, $\mathbf{P}\in\mathbb{S}_{+}^{n}$, $\mathbf{Q}\in\mathbb{S}_{+}^{p_d+n-r}$, then $\mathbf{u}^*:=\mathbf{B}'\mathbf{w}$ is a solution of problem \eqref{optimization_problem_under_Bj} with $\lambda_j$, $j\in\mathbb{N}_d$ and for each $j\in\mathbb{N}_d$, $\mathbf{u}^*$ has sparsity of level $l_{j}$ under $\mathbf{B}_{j}$. Moreover, $\mathbf{a}:=\hat{\mathbf{a}}\in\partial\bm{\psi}(\mathbf{u}^*)$ and $\mathbf{b}:=[c_j:j\in\mathbb{N}_{p_d}]\in\mathcal{N}(\mathbf{B}^{\top})$ satisfy \eqref{sparsity-equi-mini-1}, \eqref{sparsity-equi-mini-2} and \eqref{sparsity-equi-mini-3}. 
\end{theorem}
\begin{proof}
Item (b) of Proposition \ref{FPPA_deformation} ensures that if $\mathbf{w}$, $\hat{\mathbf{a}}$ and $\mathbf{c}$ satisfy equations \eqref{FPPA_nondiff1}, \eqref{FPPA_nondiff2} and \eqref{FPPA_nondiff3} for some matrices $\mathbf{O}\in\mathbb{S}_{+}^{p_d+n-r}$, $\mathbf{P}\in\mathbb{S}_{+}^{n}$, $\mathbf{Q}\in\mathbb{S}_{+}^{p_d+n-r}$, then $\mathbf{w}$ is a solution of problem \eqref{optimization_problem_under_Bj_3} with $\lambda_j$, $j\in\mathbb{N}_d$. According to Lemma \ref{equi_mini}, we get that $\mathbf{u}^*:=\mathcal{B}^{-1}\mathbf{w}$ is a solution of problem \eqref{optimization_problem_under_Bj} with $\lambda_j$, $j\in\mathbb{N}_d$. By definition of mapping $\mathcal{B}$, we get that $\mathbf{u}^*=\mathbf{B}'\mathbf{w}$ and $\mathbf{B}\mathbf{u}^*=\mathbf{z}$. It follows from definition of matrix $\mathbf{B}$ that for each $j\in \mathbb{N}_d,$ $\mathbf{B}_j\mathbf{u}^*=\mathbf{z}_j$, which shows that $\mathbf{u}^*$ has sparsity of level $l_{j}$ under the transform matrix $\mathbf{B}_{j}$. 
 
 It suffices to verify that $\mathbf{a}\in\partial\bm{\psi}(\mathbf{u}^*)$, $\mathbf{b}\in\mathcal{N}(\mathbf{B}^{\top})$ and there hold \eqref{sparsity-equi-mini-1}, \eqref{sparsity-equi-mini-2} and \eqref{sparsity-equi-mini-3}. As pointed out in the proof of Item (a) of Proposition \ref{FPPA_deformation}, equations \eqref{FPPA_nondiff2} and \eqref{FPPA_nondiff3} are equivalent to inclusion relations $\hat{\mathbf{a}}\in\partial\bm{\psi} (\mathbf{B}'\mathbf{w})$ and $\mathbf{c}\in\partial\iota_{\mathbb{M}}(\mathbf{w})$, respectively. By noting that $\mathbf{u}^*=\mathbf{B}'\mathbf{w}$ and $\mathbf{a}:=\hat{\mathbf{a}}$, we get that $\mathbf{a}\in\partial\bm{\psi} (\mathbf{u}^*)$. Recalling that 
 $$
 \partial\iota_{\mathbb{M}}(\mathbf{w})=\mathcal{N}(\mathbf{B}^{\top})\times\{\mathbf{0}\}
 $$ 
 leads to $\mathbf{b}\in\mathcal{N}(\mathbf{B}^{\top})$. Note that equation \eqref{FPPA_nondiff1} are equivalent to inclusion relation \eqref{nondiff-Fermat-rule-chain-rule1}. The latter holds if and only if inclusion relation \eqref{Fermat-rule-chain-rule-1} holds. As pointed out in the proof of Lemma \ref{sparsity-equi-mini}, inclusion relation \eqref{Fermat-rule-chain-rule-1} yields that $\mathbf{a}$ and $\mathbf{b}$ satisfy \eqref{sparsity-equi-mini-1}, \eqref{sparsity-equi-mini-2} and \eqref{sparsity-equi-mini-3}. 
\end{proof}

According to Theorem \ref{a-b}, the solution $\mathbf{u}^*$ of problem \eqref{optimization_problem_under_Bj} and the vectors $\mathbf{a}$, $\mathbf{b}$ satisfying \eqref{sparsity-equi-mini-1}, \eqref{sparsity-equi-mini-2} and \eqref{sparsity-equi-mini-3} can be obtained by solving the system of fixed-point equations \eqref{FPPA_nondiff1}, \eqref{FPPA_nondiff2} and \eqref{FPPA_nondiff3}. These three fixed-point equations are coupled
together and they have to be solved simultaneously by iteration. It is convenient to write equations \eqref{FPPA_nondiff1}, \eqref{FPPA_nondiff2} and \eqref{FPPA_nondiff3} in a compact form. To this end, we utilize the three column vectors $\mathbf{w}$, $\hat{\mathbf{a}}$ and $\mathbf{c}$ to form a block column vector $\bm{\tau}\in \mathbb{R}^{2p_d+3n-2r}$ having $\mathbf{w}$, $\hat{\mathbf{a}}$ and $\mathbf{c}$ as its three blocks. That is, $\bm{\tau}^\top:=[\mathbf{w}^\top,
\hat{\mathbf{a}}^\top,
\mathbf{c}^\top]$.
% \begin{equation*}\label{vector_v}
%\mathbf{v}:=
%\begin{bmatrix}
%\mathbf{w}\\
%\mathbf{a}\\
%\mathbf{c}
%\end{bmatrix}.
%\in \mathbb{R}^{p_d+n-r}\times\mathbb{R}^{n}\times\mathbb{R}^{p_d+n-r}.
%\end{equation*}
By integrating together the three proximity operators involved in equations \eqref{FPPA_nondiff1}, \eqref{FPPA_nondiff2} and \eqref{FPPA_nondiff3}, we introduce an operator from $\mathbb{R}^{2p_d+3n-2r}$ to itself by
\begin{equation*}\label{proximity_operator_T}
\mathcal{P}:= 
\begin{bmatrix}
\mathrm{prox}_{\sum_{j\in\mathbb{N}_d}\lambda_j\|\cdot\|_{1}\circ\mathbf{I}^{'}_{j},\mathbf{O}}\\
\mathrm{prox}_{\bm{\psi}^*,\mathbf{P}}\\
\mathrm{prox}_{\iota_{\mathbb{M}}^*,\mathbf{Q}}
\end{bmatrix}.
\end{equation*}
We also define a block matrix by
\begin{equation*}\label{matrix_E}
\mathbf{E}:=\begin{bmatrix}
&\mathbf{I}_{p_d+n-r} &-\mathbf{O}^{-1}(\mathbf{B}')^{\top} &-\mathbf{O}^{-1}\\
&\mathbf{P}^{-1}\mathbf{B}' &\mathbf{I}_{n} 
&\mathbf{0}&\\
&\mathbf{Q}^{-1} &\mathbf{0}
& \mathbf{I}_{p_d+n-r}  
\end{bmatrix}.    
\end{equation*}
In the above notions, we rewrite equations \eqref{FPPA_nondiff1}, \eqref{FPPA_nondiff2} and \eqref{FPPA_nondiff3} in the following compact form
\begin{equation}\label{fixed-point-problem}
\bm{\tau} =(\mathcal{P}\circ \mathbf{E})(\bm{\tau}). 
\end{equation}

Since equations \eqref{FPPA_nondiff1}, \eqref{FPPA_nondiff2} and \eqref{FPPA_nondiff3} are represented in the compact form \eqref{fixed-point-problem}, one may define the Picard iteration based on \eqref{fixed-point-problem} to solve the fixed-point $\bm{\tau}$ of the operator $\mathcal{P}\circ \mathbf{E}$, that is
$$
\bm{\tau}^{k+1} = (\mathcal{P}\circ \mathbf{E})(\bm{\tau}^{k}),\ k=0,1,\cdots
$$
When it converges, the Picard sequence $\bm{\tau}^{k}$, $k=0,1,\cdots$, generated by the Picard iteration above,
converges to a fixed-point of the operator $\mathcal{P}\circ \mathbf{E}$. It is known \cite{Byrne2003} that the convergence of the Picard sequence requires the firmly non-expansiveness of the operator $\mathcal{P}\circ \mathbf{E}$. However, by arguments similar to those used in the proof of  Lemmas 3.1 and 3.2 of \cite{li2015multi}, we can prove that the operator  $\mathcal{P}\circ \mathbf{E}$ is not firmly non-expansive. We need to reformulate the fixed-point equation \eqref{fixed-point-problem} by appropriately split the matrix $\mathbf{E}$ guided
by the theory of the non-expansive map.

We describe the split of $\mathbf{E}$ as follows. Set 
$$
\mathbf{R}:=\text{diag}(\mathbf{O},\mathbf{P},\mathbf{Q}).
$$
By  introducing three $(2p_d+3n-2r)\times(2p_d+3n-2r)$ matrices $\mathbf{M}_{0},\mathbf{M}_{1},\mathbf{M}_{2}$ satisfying $\mathbf{M}_{0}=\mathbf{M}_{1}+\mathbf{M}_{2}$, we split the expansive matrix $\mathbf{E}$ as 
$$
\mathbf{E}
=(\mathbf{E}-\mathbf{R}^{-1}\mathbf{M}_{0})
+\mathbf{R}^{-1}\mathbf{M}_{1}+\mathbf{R}^{-1}\mathbf{M}_{2}.
$$
Accordingly, the fixed-point equation \eqref{fixed-point-problem} can be  rewritten as
$$
\bm{\tau}
=\mathcal{P}((\mathbf{E}-\mathbf{R}^{-1}\mathbf{M}_{0})\bm{\tau}
+\mathbf{R}^{-1}\mathbf{M}_{1}\bm{\tau}
+\mathbf{R}^{-1}\mathbf{M}_{2}\bm{\tau}).
$$
Based upon the above equation, we define a two-step iteration to solve the fixed-point $\bm{\tau}$ of the operator $\mathcal{P}\circ \mathbf{E}$ as
\begin{equation}\label{two-step_iterative}
\bm{\tau}^{k+1}
=\mathcal{P}((\mathbf{E}-\mathbf{R}^{-1}\mathbf{M}_{0})\bm{\tau}^{k+1}
+\mathbf{R}^{-1}\mathbf{M}_{1}\bm{\tau}^{k}
+\mathbf{R}^{-1}\mathbf{M}_{2}\bm{\tau}^{k-1}).    
\end{equation}
For equation \eqref{two-step_iterative} to serve as an effective iterative algorithm, both its solvability and convergence must be carefully examined.  

These two issues may be addressed by choosing appropriate matrices $\mathbf{M}_{0}$, $\mathbf{M}_{1}$, $\mathbf{M}_{2}$.
Specifically, by introducing a real number $\theta$, we choose $\mathbf{M}_{0}$, $\mathbf{M}_{1}$, $\mathbf{M}_{2}$ as
\begin{equation}\label{matrix_M0}
\mathbf{M}_{0}:=\begin{bmatrix}
\mathbf{O} &-(\mathbf{B}')^{\top} &-\mathbf{I}_{p_d+n-r}\\
-\theta\mathbf{B}' &\mathbf{P} 
&\mathbf{0}\\
-\theta\mathbf{I}_{p_d+n-r} &\mathbf{0}
& \mathbf{Q}
\end{bmatrix},    
\end{equation}
\begin{equation}\label{matrix_M1}
\mathbf{M}_{1}:=\begin{bmatrix}
\mathbf{O} &(\theta-2)(\mathbf{B}')^{\top} 
&(\theta-2)\mathbf{I}_{p_d+n-r}\\
-\theta\mathbf{B}' &\mathbf{P} 
&\mathbf{0}\\
-\theta\mathbf{I}_{p_d+n-r} &\mathbf{0}
& \mathbf{Q}
\end{bmatrix},    
\end{equation}
and 
\begin{equation}\label{matrix_M2}
\mathbf{M}_{2}:=\begin{bmatrix}
\mathbf{0} &(1-\theta)(\mathbf{B}')^{\top} 
&(1-\theta)\mathbf{I}_{p_d+n-r}\\
\mathbf{0}  &\mathbf{0}
&\mathbf{0}\\
\mathbf{0} &\mathbf{0} &\mathbf{0}
\end{bmatrix}.    
\end{equation}
It is clear that $\mathbf{M}_{0}=\mathbf{M}_{1}+\mathbf{M}_{2}$. Associated with these matrices, we have that
\begin{equation}\label{matrix_M0_1}
(\mathbf{E}-\mathbf{R}^{-1}\mathbf{M}_{0})
=\begin{bmatrix}
\mathbf{0}&\mathbf{0}&\mathbf{0}\\
(1+\theta)\mathbf{P}^{-1}\mathbf{B}' &\mathbf{0}&\mathbf{0}\\
(1+\theta)\mathbf{Q}^{-1}&\mathbf{0}&\mathbf{0}
\end{bmatrix},
\end{equation}
\begin{equation}\label{matrix_M1_1}
\mathbf{R}^{-1}\mathbf{M}_{1} 
=\begin{bmatrix}
\mathbf{I}_{p_d+n-r}&(\theta-2)\mathbf{O}^{-1}(\mathbf{B}')^{\top}&(\theta-2)\mathbf{O}^{-1}\\
-\theta\mathbf{P}^{-1}\mathbf{B}'&\mathbf{I}_{n}&\mathbf{0}\\
-\theta\mathbf{Q}^{-1}&\mathbf{0}&\mathbf{I}_{p_d+n-r}
\end{bmatrix}
\end{equation}
and 
\begin{equation}\label{matrix_M2_1}
\mathbf{R}^{-1}\mathbf{M}_{2}
=\begin{bmatrix}
\mathbf{0}&(1-\theta )\mathbf{O}^{-1}(\mathbf{B}')^{\top}& (1-\theta )\mathbf{O}^{-1}\\
\mathbf{0}&\mathbf{0}&\mathbf{0}\\
\mathbf{0}&\mathbf{0}&\mathbf{0}
\end{bmatrix}.
\end{equation}
We then express iteration \eqref{two-step_iterative} in terms of the vectors $\mathbf{w}^k$, $\mathbf{a}^k$ and $\mathbf{c}^k$ as
\begin{equation}\label{FPPA_w}
\left\{\begin{array}{l}
\mathbf{w}^{k+1}=\mathrm{prox}_{\sum_{j\in\mathbb{N}_d}\lambda_j\|\cdot\|_{1}\circ\mathbf{I}^{'}_{j},\mathbf{O}}
\big(\mathbf{w}^{k}+\mathbf{O}^{-1}(\mathbf{B}')^{\top}((\theta-2)\hat{\mathbf{a}}^{k}\\
\ \quad\qquad+(1-\theta)\hat{\mathbf{a}}^{k-1})+\mathbf{O}^{-1}((\theta-2)\mathbf{c}^{k}+(1-\theta)\mathbf{c}^{k-1})\big), \\
\ \hat{\mathbf{a}}^{k+1}=\mathrm{prox}_{\bm{\psi}^*,\mathbf{P}}
\big(\hat{\mathbf{a}}^{k}+\mathbf{P}^{-1}\mathbf{B}'(\mathbf{w}^{k+1}+\theta(\mathbf{w}^{k+1}-\mathbf{w}^{k}))\big),\\
\ 
\mathbf{c}^{k+1}=\mathrm{prox}_{\iota_{\mathbb{M}}^*,\mathbf{Q}}\big(\mathbf{c}^{k}+\mathbf{Q}^{-1}(\mathbf{w}^{k+1}+\theta(\mathbf{w}^{k+1}-\mathbf{w}^{k})\big). 
\end{array}\right.
\end{equation}
We first note that since matrix $\mathbf{E}-\mathbf{R}^{-1}\mathbf{M}_{0}$ is strictly block lower triangular, the two-step iteration \eqref{two-step_iterative}, as an implicit scheme in general, reduces to an explicit scheme \eqref{FPPA_w}. We remark that Algorithm \ref{FPPA}, presented in Section \ref{Itera}, implements the fixed-point proximity iteration \eqref{FPPA_w} to solve problem \eqref{optimization_problem_under_Bj_3}.

We next establish the convergence of iteration \eqref{FPPA_w}. To describe sufficient conditions for the convergence, we introduce two block matrices by 
\begin{equation}\label{F-G}
\mathbf{F}:=\begin{bmatrix}
\mathbf{P} &\mathbf{0} \\
\mathbf{0} &\mathbf{Q} 
\end{bmatrix} \ \text{and} \ 
\mathbf{G}:=\begin{bmatrix}
\mathbf{B}' \\
\mathbf{I}_{p_d+n-r} 
\end{bmatrix}.    
\end{equation}
\begin{theorem}\label{convergence_FPPA}
 If matrices $\mathbf{O}\in\mathbb{S}_{+}^{p_d+n-r}$, $\mathbf{P}\in\mathbb{S}_{+}^{n}$, $\mathbf{Q}\in\mathbb{S}_{+}^{p_d+n-r}$ and $\theta\in\mathbb{R}$ satisfy 
\begin{equation}\label{satisfy-M-condition}
|\theta|\big\| 
\mathbf{F}^{-\frac{1}{2}}\mathbf{G}
\mathbf{O}^{-\frac{1}{2}}
\big\|_2<1 
\end{equation}
and 
\begin{equation}\label{satisfy-M-condition-H}
\frac{|1-\theta|\|\mathbf{G}\|_2\max\big\{\|\mathbf{O}^{-1}\|_2,
\left\|\mathbf{F}^{-1}
\right\|_2\big\}}{1-|\theta|\big\| 
\mathbf{F}^{-\frac{1}{2}}\mathbf{G}
\mathbf{O}^{-\frac{1}{2}}
\big\|_2}<\frac{1}{2},   
\end{equation}
then the sequence 
$\bm{\tau}^k:=\scalebox{0.8}{$
\begin{bmatrix}
\mathbf{w}^k\\
\hat{\mathbf{a}}^k\\
\mathbf{c}^k
\end{bmatrix}$}$, $k\in\mathbb{N}$, generated by iteration \eqref{FPPA_w} for any given $\bm{\tau}^0, \bm{\tau}^1\in\mathbb{R}^{2p_d+3n-2r}$, converges to a fixed-point of operator $\mathcal{P}\circ\mathbf{E}$.
\end{theorem}

Theorem \ref{convergence_FPPA} demonstrates that the convergence of iteration \eqref{FPPA_w} depends on the selection of the matrices $\mathbf{O}$, $\mathbf{P}$, $\mathbf{Q}$ and the real number $\theta$. A practical choice is to take
\begin{equation}\label{OPQtheta}
\mathbf{O}:=\frac{1}{\alpha}\mathbf{I}_{p_d+n-r},\ \mathbf{P}:=\frac{1}{\rho}\mathbf{I}_{n},\ \mathbf{Q}:=\frac{1}{\beta}\mathbf{I}_{p_d+n-r},\ \theta:=1,
\end{equation}
where $\alpha$, $\rho$ and $\beta$ are positive real numbers. With this selection, the convergence conditions \eqref{satisfy-M-condition} and \eqref{satisfy-M-condition-H} reduce to 
\begin{equation}\label{OPQtheta-convergence-condition}
\left\|\begin{bmatrix}
\sqrt{\rho}\mathbf{B}'\\
\sqrt{\beta}\mathbf{I}_{p_d+n-r}
\end{bmatrix}\right\|_2<1/\sqrt{\alpha}.
\end{equation}
% $\left\|\scalebox{0.8}{$\begin{bmatrix}
% \sqrt{\rho}\mathbf{B}'\\
% \sqrt{\beta}\mathbf{I}_{p_d+n-r}
% \end{bmatrix}$}\right\|_2<1/\sqrt{\alpha}$. 
Moreover, the positive constants $\alpha$, $\rho$ and $\beta$ may be appropriately chosen to improve the convergence rate. 

To end this section, we provide a complete proof for Theorem \ref{convergence_FPPA}. We start with reviewing the notion of weakly firmly non-expansive operators introduced in \cite{li2015multi}. An operator $T:\mathbb{R}^{2s}\to \mathbb{R}^{s}$ is called weakly firmly non-expansive with respect to a set $\mathcal{M}:=\{\mathbf{M}_{0},\mathbf{M}_{1},\mathbf{M}_{2}\}$ of $s\times s $ matrices if for any $(\mathbf{s}^{i},\mathbf{z}^{i},\bm{\omega}^{i})\in\mathbb{R}^s\times\mathbb{R}^s\times\mathbb{R}^s$ satisfying $\bm{\omega}^{i}=T(\mathbf{s}^{i},\mathbf{z}^{i})$ for $i=1,2,$ there holds
\begin{equation}\label{weakly-firmly-non-expansive}
\big\langle\bm{\omega}^{2}-\bm{\omega}^{1}, \mathbf{M}_{0}(\bm{\omega}^{2}-\bm{\omega}^{1})
\big\rangle
\leq\big\langle\bm{\omega}^{2}-\bm{\omega}^{1},\mathbf{M}_{1}(\mathbf{s}^{2}-\mathbf{s}^{1})+\mathbf{M}_{2}(\mathbf{z}^{2}-\mathbf{z}^{1})\big\rangle.
\end{equation}
The graph ${\rm gra}(T)$ of operator $T$ is defined by 
$$
{\rm gra}(T):=\{(\mathbf{s},\mathbf{z},\bm{\omega}):(\mathbf{s},\mathbf{z},\bm{\omega})\in\mathbb{R}^s\times\mathbb{R}^s\times\mathbb{R}^s,\bm{\omega}=T(\mathbf{s},\mathbf{z})\}.
$$
We say the graph ${\rm gra}(T)$ of $T$
is a closed set if for any sequence $\{(\mathbf{s}^k,\mathbf{z}^k,\bm{\omega}^k)\in{\rm gra}(T):k\in\mathbb{N}\}$ converging
to $(\mathbf{s},\mathbf{z},\bm{\omega})$, there holds $(\mathbf{s},\mathbf{z},\bm{\omega})\in{\rm gra}(T)$.
Following \cite{li2015multi}, we also need the notion of Condition-$\mathbf{M}$. 
We say a set $\mathcal{M}:=\{\mathbf{M}_{0},\mathbf{M}_{1},\mathbf{M}_{2}\}$ of $s\times s $ matrices satisfies Condition-$\mathbf{M}$, if the following three hypotheses are satisfied:
\begin{itemize}[leftmargin=4em, labelwidth=2em, labelsep*=0.5em, align=left]
 \item[\textbf{(i)}]
$\mathbf{M}_{0}=\mathbf{M}_{1}+\mathbf{M}_{2}$;

\item[\textbf{(ii)}] $\mathbf{H}:=\mathbf{M}_{0}+\mathbf{M}_{2}$ is in $\mathbb{S}^s_{+}$;

 \item[\textbf{(iii)}] $\big\|\mathbf{H}^{-1/2}\mathbf{M}_{2}\mathbf{H}^{-1/2}\big\|_2<1/2$. 
\end{itemize}

% We next give a brief review of the definition and some properties of firmly non-expansive operators.
% Recall that an operator $\mathcal{J}:\mathbb{R}^s\to \mathbb{R}^s$ is called firmly non-expansive(non-expansive) with respect to a given matrix $\mathbf{H}\in\mathbb{S}_{+}^{s}$ if for all $\mathbf{x}$ and $\mathbf{y}\in\mathbb{R}^s$,
% \begin{equation*}
% \|\mathcal{J}\mathbf{x}-\mathcal{J}\mathbf{y}\|^2 _{\mathbf{H}} \leq \langle \mathcal{J}\mathbf{x}-\mathcal{J}\mathbf{y}, \mathbf{x}-\mathbf{y} \rangle_{\mathbf{H}} \ \ 
% (\|\mathcal{J}\mathbf{x}-\mathcal{J}\mathbf{y}\|_{\mathbf{H}} \leq \|\mathbf{x}-\mathbf{y}\|_{\mathbf{H}}).
% \end{equation*} 

The next result derived in \cite{li2015multi} demonstrates that if an operator $T$ with a closed graph is weakly firmly non-expansive with respect to a set $\mathcal{M}:=\{\mathbf{M}_{0},\mathbf{M}_{1},\mathbf{M}_{2}\}$ of $s\times s $ matrices satisfying Condition-$\mathbf{M}$, then a
sequence $\{\bm{\omega}^{k}:k\in\mathbb{N}\}$ generated by 
\begin{equation}\label{fixed-point_T}
\bm{\omega}^{k+1}=T(\bm{\omega}^{k},\bm{\omega}^{k-1})
\end{equation}
converges to a fixed point $\bm{\omega}$ of $T$, that is, $\bm{\omega}=T(\bm{\omega},\bm{\omega})$.

\begin{lemma}\label{converges_weakly-T}
Suppose that a set $\mathcal{M}:=\{\mathbf{M}_{0},\mathbf{M}_{1},\mathbf{M}_{2}\}$ of $s\times s $ matrices  satisfies Condition-$\mathbf{M}$,  the operator $T:\mathbb{R}^{2s}\to \mathbb{R}^{s}$ is weakly firmly non-expansive with respect to $\mathcal{M}$,  the set of fixed-points of $T$ is nonempty and $\mathrm{dom}(T)=\mathbb{R}^{2s}$. If the sequence $\{\bm{\omega}^{k}:k\in\mathbb{N}\}$ is generated by 
\eqref{fixed-point_T} for any given $\bm{\omega}^{0}$, $\bm{\omega}^{1}\in\mathbb{R}^{s}$, then $\{\bm{\omega}^{k}:k\in\mathbb{N}\}$ converges. Moreover, if the graph ${\rm gra}(T)$ of $T$
is closed, then $\{\bm{\omega}^{k}:k\in\mathbb{N}\}$ converges to a fixed-point of $T$.    
\end{lemma}

We establish below the convergence of iteration \eqref{FPPA_w} by leveraging Lemma \ref{converges_weakly-T}. For this purpose, we construct a new operator from the operator $\mathcal{P}\circ\mathbf{E}$. Let $\mathbf{M}_{0}$, $\mathbf{M}_{1}$, $\mathbf{M}_{2}$ be $(2p_d+3n-2r)\times(2p_d+3n-2r)$ matrices defined by \eqref{matrix_M0}, \eqref{matrix_M1} and \eqref{matrix_M2}, respectively. Associated with the set  $\mathcal{M}:=\{\mathbf{M}_{0},\mathbf{M}_{1},\mathbf{M}_{2}\}$, we define an operator  $\mathcal{T}_{\mathcal{M}}:\mathbb{R}^{2p_d+3n-2r}\times\mathbb{R}^{2p_d+3n-2r}\to\mathbb{R}^{2p_d+3n-2r}$ for any  $(\mathbf{s},\mathbf{z})\in\mathbb{R}^{2p_d+3n-2r}\times\mathbb{R}^{2p_d+3n-2r}$ by $\bm{\omega}:=\mathcal{T}_{\mathcal{M}}(\mathbf{s},\mathbf{z})$ with $\bm{\omega}$ satisfying 
\begin{equation}\label{iterative_TMtu}
\bm{\omega}
=\mathcal{P}((\mathbf{E}-\mathbf{R}^{-1}\mathbf{M}_{0})\bm{\omega}
+\mathbf{R}^{-1}\mathbf{M}_{1}\mathbf{s}
+\mathbf{R}^{-1}\mathbf{M}_{2}\mathbf{z}).    
\end{equation}
We claim that the operator $\mathcal{T}_{\mathcal{M}}$ is well-defined. To see this, we decompose any  $\mathbf{x}\in\mathbb{R}^{2p_d+3n-2r}$ as $\mathbf{x}:=\scalebox{0.7}{$\begin{bmatrix}
\mathbf{x}_1\\
\mathbf{x}_2\\
\mathbf{x}_3
\end{bmatrix}$}$ with $\mathbf{x}_1, \mathbf{x}_3\in\mathbb{R}^{p_d+n-r}$ and  $\mathbf{x}_2\in\mathbb{R}^{n}$. By using  representations  \eqref{matrix_M0_1},\eqref{matrix_M1_1} and \eqref{matrix_M2_1}, we rewrite equation \eqref{iterative_TMtu}  as
\begin{equation}\label{well-defined}
\left\{\begin{array}{l}
\bm{\omega}_{1}=\mathrm{prox}_{\sum_{j\in\mathbb{N}_d}\lambda_j\|\cdot\|_{1}\circ\mathbf{I}^{'}_{j},\mathbf{O}}
\big(\mathbf{s}_{1}+(\theta-2)\mathbf{O}^{-1}(\mathbf{B}')^{\top}\mathbf{s}_{2}+(\theta-2)\mathbf{O}^{-1}\mathbf{s}_{3}\\
 \quad\qquad+(1-\theta)\mathbf{O}^{-1}(\mathbf{B}')^{\top}\mathbf{z}_{2}+(1-\theta)\mathbf{O}^{-1}\mathbf{z}_{3}, \\
\bm{\omega}_{2}=\mathrm{prox}_{\bm{\psi}^*,\mathbf{P}}
\big(\mathbf{s}_{2}+(1+\theta)\mathbf{P}^{-1}\mathbf{B}'\bm{\omega}_{1}-\theta\mathbf{P}^{-1}\mathbf{B}'\mathbf{s}_{1}\big),\\
\bm{\omega}_{3}=\mathrm{prox}_{\iota_{\mathbb{M}}^*,\mathbf{Q}}\big(\mathbf{s}_{3}+(1+\theta)\mathbf{Q}^{-1}\bm{\omega}_{1}-\theta\mathbf{Q}^{-1}\mathbf{s}_{1}\big). 
\end{array}\right.
\end{equation}
Clearly, for any $(\mathbf{s},\mathbf{z})\in\mathbb{R}^{2p_d+3n-2r}\times\mathbb{R}^{2p_d+3n-2r}$, there exists a unique $\bm{\omega}\in\mathbb{R}^{2p_d+3n-2r}$ satisfying equation \eqref{well-defined}. Observing from equation \eqref{iterative_TMtu}, we see that the operator $\mathcal{T}_{\mathcal{M}}$ has the same fixed-point set as $\mathcal{P}\circ\mathbf{E}$. With the help
of the operator $\mathcal{T}_{\mathcal{M}}$, we represent equation  \eqref{two-step_iterative} in an explicit form as
\begin{equation}\label{iterative_TM}
\bm{\tau}^{k+1}
=\mathcal{T}_{\mathcal{M}}(\bm{\tau}^{k},
\bm{\tau}^{k-1})  
\end{equation}
and obtain a fixed-point of $\mathcal{T}_{\mathcal{M}}$ by this iteration. 

According to Lemma \ref{converges_weakly-T}, in order to obtain the convergence of iteration \eqref{iterative_TM}, it suffices to prove that the operator $\mathcal{T}_{\mathcal{M}}$ defined by \eqref{iterative_TMtu} is weakly firmly non-expansive with respect
to $\mathcal{M}$ and the graph ${\rm gra}(\mathcal{T}_{\mathcal{M}})$ of $\mathcal{T}_{\mathcal{M}}$
is closed, and in addition, the set $\mathcal{M}$ satisfies Condition-$\mathbf{M}$. 
We first show the properties of the operator $\mathcal{T}_{\mathcal{M}}$. For this purpose, we introduce a $(2p_d+3n-2r)\times(2p_d+3n-2r)$ skew-symmetric matrix $\mathbf{S}_{\mathbf{B}}$ as
\begin{equation*}
\mathbf{S}_{\mathbf{B}}:=\begin{bmatrix}
\mathbf{0} &-(\mathbf{B}')^{\top} &-\mathbf{I}_{p_d+n-r}\\
\mathbf{B}' &\mathbf{0} 
&\mathbf{0}\\
\mathbf{I}_{p_d+n-r} 
& \mathbf{0} & \mathbf{0} 
\end{bmatrix}    
\end{equation*}
and then represent the matrix $\mathbf{E}$ as $\mathbf{E}=\mathbf{I}_{2p_d+3n-2r}+\mathbf{R}^{-1}\mathbf{S}_{\mathbf{B}}$. 

\begin{proposition}\label{TM-continuous}
Let $\mathbf{M}_{0}$, $\mathbf{M}_{1}$, $\mathbf{M}_{2}$ be $(2p_d+3n-2r)\times(2p_d+3n-2r)$ matrices defined by \eqref{matrix_M0}, \eqref{matrix_M1} and \eqref{matrix_M2}, respectively, and $\mathcal{M}:=\{\mathbf{M}_{0},\mathbf{M}_{1},\mathbf{M}_{2}\}$. If  $\mathcal{T}_{\mathcal{M}}$ is defined by equation \eqref{iterative_TMtu}, then $\mathcal{T}_{\mathcal{M}}$ is weakly firmly non-expansive with respect to $\mathcal{M}$ and the graph ${\rm gra}(\mathcal{T}_{\mathcal{M}})$ of $\mathcal{T}_{\mathcal{M}}$
is closed.
\end{proposition}
\begin{proof}
We first prove that $\mathcal{T}_{\mathcal{M}}$ is weakly firmly non-expansive with respect to $\mathcal{M}$. It suffices to prove that for any $(\mathbf{s}^{i},\mathbf{z}^{i},\bm{\omega}^{i})\in\mathbb{R}^{2p_d+3n-2r}\times\mathbb{R}^{2p_d+3n-2r}\times\mathbb{R}^{2p_d+3n-2r}$ satisfying $\bm{\omega}^{i}=\mathcal{T}_{\mathcal{M}}(\mathbf{s}^{i},\mathbf{z}^{i})$ for $i=1,2,$ there holds equation \eqref{weakly-firmly-non-expansive}. It follows from definition \eqref{iterative_TMtu} of $\mathcal{T}_{\mathcal{M}}$ that 
\begin{equation}\label{w-s-z}
\bm{\omega}^{i}=\mathcal{P}\big((\mathbf{E}-\mathbf{R}^{-1}\mathbf{M}_{0})\bm{\omega}^{i}
+\mathbf{R}^{-1}\mathbf{M}_{1}\mathbf{s}^{i}
+\mathbf{R}^{-1}\mathbf{M}_{2}\mathbf{z}^{i}\big),\ i=1,2.
\end{equation}
By arguments similar to those used in the proof of Lemma 3.1 of \cite{li2015multi}, we have that operator $\mathcal{P}$ is firmly non-expansive with respect to $\mathbf{R}$, that is, for all $\mathbf{x}$, $\mathbf{y}\in\mathbb{R}^{2p_d+3n-2r}$,
\begin{equation*}
\|\mathcal{P}\mathbf{x}-\mathcal{P}\mathbf{y}\|^2 _{\mathbf{R}} \leq \langle \mathcal{P}\mathbf{x}-\mathcal{P}\mathbf{y}, \mathbf{x}-\mathbf{y} \rangle_{\mathbf{R}}.
\end{equation*} 
As a result, we get by equation \eqref{w-s-z} that  
\begin{equation*}
\|\bm{\omega}^{2}-\bm{\omega}^{1}\|^2_{\mathbf{R}}
\leq \big\langle\bm{\omega}^{2}-\bm{\omega}^{1},(\mathbf{R}\mathbf{E}-\mathbf{M}_{0})(\bm{\omega}^{2}-\bm{\omega}^{1})+\mathbf{M}_{1}(\mathbf{s}^{2}-\mathbf{s}^{1})+\mathbf{M}_{2}(\mathbf{z}^{2}-\mathbf{z}^{1}) \big\rangle. 
\end{equation*}
Substituting $\mathbf{R}\mathbf{E}=\mathbf{R}+\mathbf{S}_{\mathbf{B}}$ into the right hand side of the inequality above, we obtain that 
\begin{equation*}
\big\langle\bm{\omega}^{2}-\bm{\omega}^{1},
\mathbf{M}_{0}(\bm{\omega}^{2}-\bm{\omega}^{1})
\big\rangle
\leq \big\langle\bm{\omega}^{2}-\bm{\omega}^{1},\mathbf{S}_{\mathbf{B}}(\bm{\omega}^{2}-\bm{\omega}^{1})+\mathbf{M}_{1}(\mathbf{s}^{2}-\mathbf{s}^{1})+\mathbf{M}_{2}(\mathbf{z}^2-\mathbf{z}^1) \big\rangle. 
\end{equation*}
This together with the fact that  
$$
\langle\bm{\omega}^{2}-\bm{\omega}^{1},
\mathbf{S}_{\mathbf{B}}(\bm{\omega}^{2}-\bm{\omega}^{1})\rangle=0
$$ 
leads to equation \eqref{weakly-firmly-non-expansive}. 
Consequently, we conclude that $\mathcal{T}_{\mathcal{M}}$ is weakly firmly non-expansive with respect to $\mathcal{M}$.

It remains to show the closedness of the graph of operator  $\mathcal{T}_{\mathcal{M}}$. For any sequence 
$\big\{(\mathbf{s}^{k},
\mathbf{z}^{k},\bm{\omega}^{k})\in{\rm gra}(\mathcal{T}_{\mathcal{M}}):k\in\mathbb{N}\big\}
$ converging to $(\mathbf{s},\mathbf{z},\bm{\omega})\in\mathbb{R}^{2p_d+3n-2r}\times\mathbb{R}^{2p_d+3n-2r}\times\mathbb{R}^{2p_d+3n-2r}$, we obtain from definition \eqref{iterative_TMtu} of $\mathcal{T}_{\mathcal{M}}$ that 
\begin{equation*}
\bm{\omega}^{k}=\mathcal{P}((\mathbf{E}-\mathbf{R}^{-1}\mathbf{M}_{0})\bm{\omega}^{k}
+\mathbf{R}^{-1}\mathbf{M}_{1}\mathbf{s}^{k}
+\mathbf{R}^{-1}\mathbf{M}_{2}\mathbf{z}^{k}), \ \mbox{for all}\ k\in\mathbb{N}. 
\end{equation*}
Associated with $(\mathbf{s},\mathbf{z},\bm{\omega})$, we introduce a vector $\widetilde{\bm{\omega}}\in\mathbb{R}^{2p_d+3n-2r}$ as   
$$
\widetilde{\bm{\omega}}:=\mathcal{P}((\mathbf{E}-\mathbf{R}^{-1}\mathbf{M}_{0})\bm{\omega}
+\mathbf{R}^{-1}\mathbf{M}_{1}\mathbf{s}
+\mathbf{R}^{-1}\mathbf{M}_{2}\mathbf{z}). 
$$
Since $\mathcal{P}$ is firmly non-expansive with respect to $\mathbf{R}$, there holds for any $k\in\mathbb{N}$
\begin{equation*}
\|\bm{\omega}^{k}-\widetilde{\bm{\omega}}\|^2_{\mathbf{R}}
\leq \big\langle \bm{\omega}^{k}-\widetilde{\bm{\omega}},(\mathbf{R}-\mathbf{M}_{0})(\bm{\omega}^{k}-\bm{\omega})+\mathbf{M}_{1}(\mathbf{s}^{k}-\mathbf{s})+\mathbf{M}_{2}(\mathbf{z}^{k}-\mathbf{z})\big\rangle.
\end{equation*}
By letting $k\rightarrow+\infty$ with noting that $\mathbf{s}^k\rightarrow\mathbf{s}$ and $\mathbf{z}^k\rightarrow\mathbf{z}$ as $k\rightarrow+\infty$, we get $\lim_{k\to+\infty}\bm{\omega}^{k}=\widetilde{\bm{\omega}}$. As a result, $\bm{\omega}=\widetilde{\bm{\omega}}$. This together with the definition of $\widetilde{\bm{\omega}}$ leads directly to 
$\bm{\omega}=\mathcal{T}_{\mathcal{M}}(\mathbf{s},\mathbf{z})$. Therefore, the graph ${\rm gra}(\mathcal{T}_{\mathcal{M}})$ of $\mathcal{T}_{\mathcal{M}}$
is closed.
\end{proof}

The next proposition reveals that the set $\mathcal{M}$ satisfies Condition-$\mathbf{M}$ under conditions \eqref{satisfy-M-condition} and \eqref{satisfy-M-condition-H}. 

\begin{proposition}\label{satisfy-M-condition-proposition}
Let $\mathbf{M}_{0}$, $\mathbf{M}_{1}$, $\mathbf{M}_{2}$ be $(2p_d+3n-2r)\times(2p_d+3n-2r)$ matrices defined by \eqref{matrix_M0}, \eqref{matrix_M1} and \eqref{matrix_M2}, respectively. If conditions \eqref{satisfy-M-condition} and \eqref{satisfy-M-condition-H} hold, then the set $\mathcal{M}:=\{\mathbf{M}_{0}$, $\mathbf{M}_{1}$, $\mathbf{M}_{2}\}$ satisfies Condition-$\mathbf{M}$.
\end{proposition}
\begin{proof}
It is clear  that  $\mathbf{M}_{0}=\mathbf{M}_{1}+\mathbf{M}_{2}$, that is, Item (i) of Condition-$\mathbf{M}$ holds. To show the validity of Item (ii), we set $\mathbf{H}:=\mathbf{M}_{0}+\mathbf{M}_{2}$. By introducing block matrices $\mathbf{F}$ and $\mathbf{G}$ as in \eqref{F-G}, we represent $\mathbf{H}$ as 
\begin{equation*}
\mathbf{H}
% =\begin{bmatrix}
% \mathbf{O} &-\theta(\mathbf{B}')^{\top} &-\theta\mathbf{I}_{p_d+n-r} \\
% -\theta\mathbf{B}' &\mathbf{P} 
% &\mathbf{0}\\
% -\theta\mathbf{I}_{p_d+n-r} &\mathbf{0}
% & \mathbf{Q} 
%\end{bmatrix}
=\begin{bmatrix}
&\mathbf{O} &-\theta\mathbf{G}^{\top}\\
&-\theta\mathbf{G} &\mathbf{F}
\end{bmatrix}.
\end{equation*}
It follows from condition \eqref{satisfy-M-condition} that
\begin{equation}\label{inequality-norm-FGO}
\big\| 
\mathbf{F}^{-\frac{1}{2}}(-\theta\mathbf{G})
\mathbf{O}^{-\frac{1}{2}}
\big\|_2
% =|\theta|\big\| 
% \mathbf{F}^{-\frac{1}{2}}\mathbf{G}
% \mathbf{O}^{-\frac{1}{2}}
% \big\|_2
<1,
\end{equation}
which guaranteed by Lemma 6.2 in \cite{li2015multi} is equivalent to 
$\mathbf{H}\in\mathbb{S}_{+}^{2p_d+3n-2r}$. Hence, we get Item (ii) of Condition-$\mathbf{M}$.

It remains to verify Item (iii) of Condition-$\mathbf{M}$. Lemma 6.2 in \cite{li2015multi} ensures that if inequality \eqref{inequality-norm-FGO} holds, the norm of matrix $\mathbf{H}^{-1}$ can be estimated by  
\begin{equation}\label{norm_H_inverse}
\|\mathbf{H}^{-1}\|_2\leq 
\frac{\max\big\{\|\mathbf{O}^{-1}\|_2,
\left\|\mathbf{F}^{-1}
\right\|_2\big\}}{1-|\theta|\big\| 
\mathbf{F}^{-\frac{1}{2}}\mathbf{G}
\mathbf{O}^{-\frac{1}{2}}
\big\|_2}.
\end{equation}
We observe that 
$$
\mathbf{M}_{2}^{\top}\mathbf{M}_{2}
=(1-\theta)^2\begin{bmatrix}
&\mathbf{0} &\mathbf{0} \\
&\mathbf{0} &\mathbf{G}\mathbf{G}^{\top} 
\end{bmatrix},
$$
which yields that \begin{equation}\label{norm_M_2}
\|\mathbf{M}_{2}\|_2
=|1-\theta|\|\mathbf{G}\|_2.
\end{equation} 
It follows that 
\begin{equation}\label{AAequation}
    \big\|\mathbf{H}^{-1/2}\mathbf{M}_{2}\mathbf{H}^{-1/2}
\big\|_2 \leq \big\|\mathbf{H}^{-1/2}
\big\|_2^2\big\|\mathbf{M}_{2}\big\|_2.
\end{equation}
Substituting inequality \eqref{norm_H_inverse} and equation \eqref{norm_M_2} into inequality \eqref{AAequation} yields that \begin{equation*}
\big\|\mathbf{H}^{-1/2}\mathbf{M}_{2}\mathbf{H}^{-1/2}
\big\|_2 \leq \frac{|1-\theta|\|\mathbf{G}\|_2\max\big\{\|\mathbf{O}^{-1}\|_2,
\left\|\mathbf{F}^{-1}
\right\|_2\big\}}{1-|\theta|\big\| 
\mathbf{F}^{-\frac{1}{2}}\mathbf{G}
\mathbf{O}^{-\frac{1}{2}}
\big\|_2},
\end{equation*}
which together with condition  \eqref{satisfy-M-condition-H} further leads to Item (iii) of Condition-$\mathbf{M}$.
\end{proof}

Combining Lemma \ref{converges_weakly-T}, Propositions \ref{TM-continuous} and \ref{satisfy-M-condition-proposition}, we now prove Theorem \ref{convergence_FPPA} as follows.

\begin{proof}[Proof of Theorem \ref{convergence_FPPA}]
Proposition \ref{TM-continuous} ensures that $\mathcal{T}_{\mathcal{M}}$ is weakly firmly non-expansive with respect to $\mathcal{M}$ and the graph ${\rm gra}(\mathcal{T}_{\mathcal{M}})$ of $\mathcal{T}_{\mathcal{M}}$
is closed. Moreover, the set $\mathcal{M}$, guaranteed by Proposition \ref{satisfy-M-condition-proposition}, satisfies Condition-$\mathbf{M}$. That is, the hypotheses of Lemma  \ref{converges_weakly-T} are satisfied. By Lemma  \ref{converges_weakly-T}, we get that  
the sequence $\{\bm{\tau}^{k}:k\in\mathbb{N}\}$, generated by 
\eqref{iterative_TM} for any given $\bm{\tau}^{0}$, $\bm{\tau}^{1}\in\mathbb{R}^{2p_d+3n-2r}$, converges to a fixed-point of $\mathcal{T}_{\mathcal{M}}$. Note that iteration \eqref{FPPA_w} has the equivalent form \eqref{iterative_TM} and the operator $\mathcal{T}_{\mathcal{M}}$ has the same fixed-point set as $\mathcal{P}\circ\mathbf{E}$. Consequently, we conclude that the sequence $\{\bm{\tau}^{k}:k\in\mathbb{N}\}$, generated by iteration \eqref{FPPA_w} for any given $\bm{\tau}^0, \bm{\tau}^1\in\mathbb{R}^{2p_d+3n-2r}$, converges to a fixed-point of operator $\mathcal{P}\circ\mathbf{E}$.
\end{proof}
% FPPA

%%%%%%%%%%%%%%%%%%%%%%%%%%%
\section{Numerical experiments}\label{Numer}

We present six numerical experiments to thoroughly evaluate the proposed multi-parameter selection strategies. The first four experiments (Subsections 6.1--6.4) assess the algorithm's ability to achieve target sparsity levels across different models, including a real-world classification task. The last two experiments (Subsections 6.5--6.6) verify the theoretical convergence assumptions and analyze the sensitivity of key hyperparameters, respectively. All experiments were performed in MATLAB R2024b on a workstation equipped with an Intel Core i9-14900KF CPU (24 cores, 32 threads, 3.2 GHz), an NVIDIA GeForce RTX 4070 Ti GPU (12 GB GDDR6X), and 32 GB of RAM.

Throughout the experiments, the scalar $\theta$  and the matrices $\mathbf{O}$, $\mathbf{P}$, $\mathbf{Q}$ are chosen as in \eqref{OPQtheta}, with positive parameters $\alpha$, $\rho$ and $\beta$ tuned to satisfy the convergence condition \eqref{OPQtheta-convergence-condition}. 

To derive the simplified algorithm, we employ the following well-known property of proximity operators \cite{Bauschke2011}: for any convex function $f$ from $\mathbb{R}^s$ to $\overline{\mathbb{R}}$ and $\mu>0$
$$
\mathrm{prox}_{f,\mu\mathbf{I}_s}=\mathrm{prox}_{\frac{1}{\mu}f}
\ \
\text{and} 
\ \ 
\mathrm{prox}_{f^*,\mu\mathbf{I}_s}
=\frac{1}{\mu} (\mathbf{I}_s-\mathrm{prox}_{\mu f})\circ(\mu\mathbf{I}_s).
$$
Applying these relations,  iteration \eqref{FPPA_w} reduces to
\begin{equation}\label{Reduce_FPPA_w}
\left\{\begin{array}{l}
\mathbf{w}^{k+1}=\mathrm{prox}_{\alpha\sum_{j\in\mathbb{N}_d}\lambda_j\|\cdot\|_{1}\circ\mathbf{I}^{'}_{j}}
\big(\mathbf{w}^{k}-\alpha(\mathbf{B}')^{\top}\hat{\mathbf{a}}^{k}-\alpha\mathbf{c}^{k}\big), \\
\ \hat{\mathbf{a}}^{k+1}=\rho\big(\mathbf{I}_n-\mathrm{prox}_{\frac{1}{\rho}\bm{\psi}}\big)
\big(\frac{1}{\rho}\hat{\mathbf{a}}^{k}+\mathbf{B}'(2\mathbf{w}^{k+1}-\mathbf{w}^{k})\big),\\
\ 
\mathbf{c}^{k+1}=
\beta\big(\mathbf{I}_{p_d+n-r}-\mathrm{prox}_{\frac{1}{\beta}\iota_{\mathbb{M}}}\big)\big(\frac{1}{\beta}\mathbf{c}^{k}+2\mathbf{w}^{k+1}-\mathbf{w}^{k}\big),
\end{array}\right.
\end{equation}
so that Algorithm \ref{FPPA} can be simplified to Algorithm \ref{FPPA1}. 

Closed-form expressions for the proximity operators in Algorithm \ref{FPPA1} are given in Appendix~A.

% FPPA
\begin{algorithm}
  \caption{Fixed-Point Proximity Algorithm for problem \eqref{optimization_problem_under_Bj_3} with $\theta$, $\mathbf{O}$, $\mathbf{P}$, $\mathbf{Q}$ as in \eqref{OPQtheta} }
  
  \label{FPPA1}
  
  \KwInput{$\bm{\psi}$, $\mathbf{B}'$, $\{\mathbf{I}^{'}_{j}:j\in\mathbb{N}_d\}$, $\{\lambda_{j}:j\in\mathbb{N}_d\}$, $\mathbb{M}$; $\alpha$, $\rho$, $\beta$ satisfying $\left\|\scalebox{0.8}{$\begin{bmatrix}
\sqrt{\rho}\mathbf{B}'\\
\sqrt{\beta}\mathbf{I}_{p_d+n-r}
\end{bmatrix}$}\right\|_2<1/\sqrt{\alpha}$.}
  
  \KwInitialization{$\mathbf{w}^0$, $\hat{\mathbf{a}}^0$, $\mathbf{c}^0$.}
  \For{$k = 0,1,2,\ldots$}
  {$\mathbf{w}^{k+1} \gets  \mathrm{prox}_{\alpha\sum_{j\in\mathbb{N}_d}\lambda_j\|\cdot\|_{1}\circ\mathbf{I}^{'}_{j}}
 \big(\mathbf{w}^{k}-\alpha(\mathbf{B}')^{\top}\hat{\mathbf{a}}^{k}-\alpha\mathbf{c}^{k}\big)$.\\
 $\hat{\mathbf{a}}^{k+1} \gets \rho\big(\mathbf{I}_n-\mathrm{prox}_{\frac{1}{\rho}\bm{\psi}}\big)
\big(\frac{1}{\rho}\hat{\mathbf{a}}^{k}+\mathbf{B}'(2\mathbf{w}^{k+1}-\mathbf{w}^{k})\big)$.\\
$\mathbf{c}^{k+1} \gets \beta\big(\mathbf{I}_{p_d+n-r}-\mathrm{prox}_{\frac{1}{\beta}\iota_{\mathbb{M}}}\big)\big(\frac{1}{\beta}\mathbf{c}^{k}+2\mathbf{w}^{k+1}-\mathbf{w}^{k}\big)$.
 }
\KwOutput{$\mathbf{w}^{k+1}$, $\hat{\mathbf{a}}^{k+1}$, $\mathbf{c}^{k+1}$}  
\end{algorithm}

In reporting the numerical results, we adopt the following notation. For the multi-parameter regularization model, ``TSLs'' and ``SLs'' denote, respectively, the target sparsity levels $\{l^*_j : j \in \mathbb{N}_d\}$ and the sparsity levels $\{l_j : j \in \mathbb{N}_d\}$ of the solutions obtained from the parameter choice strategies. For the single-parameter regularization model, ``TSL'' and ``SL'' denote, respectively, the target sparsity level $l^*$ and the sparsity level $l$ of the corresponding solution.

In the first two experiments related to the special regularization problem \eqref{optimization_problem}, we use ``Ratio'' to denote the ratio $R$ between the number of nonzero components and the total number of components in the obtained solution. In the last two experiments, which concern the general regularization problem \eqref{optimization_problem_under_Bj}, ``Ratios'' denotes $\{R_j:j\in\mathbb{N}_d\}$, where each $R_j$ is the ratio of the number of nonzero components
to the total number of components of the solution under the transform matrix $\mathbf{B}_{j}$. When applying Algorithm \ref{Sparsity-Guided Multi-Parameter Selection Algorithm} for regularization parameter selection, ``NUM'' and ``Time'' to denote the number of iterations and the total computational time (in seconds) for selecting regularization parameters, respectively. 

\subsection{Parameter choices: a block separable case}\label{NM_block_separable}
In this experiment, we validate the parameter selection strategy described in
Corollary \ref{block_separabel_parameter_choice} using the multi-parameter regularization model \eqref{lasso_multi-parameter} for signal denoising. For comparison, we also evaluate the single-parameter regularization model for the same task.

We employ model \eqref{lasso_multi-parameter} to recover the Doppler signal function 
\begin{equation}\label{Doppler_signal}
f(t):=\sqrt{t(1-t)}\sin((2.1\pi)/(t+0.05)),
\ t\in[0,1],
\end{equation}
from its noisy data. Let $n:=262144$ and $t_j$, $j\in\mathbb{N}_n$, be the sample points on a uniform grid in $[0,1]$ with step size $h:={1}/{(n-1)}$. We recover the signal  $\mathbf{f}:=[f(t_j):j\in\mathbb{N}_n]$ from the noisy signal  $\mathbf{x}:=\mathbf{f}+\bm{\eta}$, where $\bm{\eta}$ is an additive white Gaussian noise with the signal-to-noise ratio $\text{SNR}=20$.

We describe the multi-parameter regularization model \eqref{lasso_multi-parameter} as follows. In this model, we let $d:=10$ and choose the matrix 
$\mathbf{A}\in\mathbb{R}^{n\times n}$ as the Daubechies wavelet transform with the vanishing moments $\mathrm{N}:=6$ and the coarsest resolution level $\mathrm{L}:=9$.  Let $m_1:=2^9$, $m_j:=2^{j+7}$, $j\in\mathbb{N}_{d}\setminus \{1\}$ and set $p_0:=0$, $
p_j:=\sum_{i\in\mathbb{N}_j}m_i$, $j\in\mathbb{N}_d$. We  choose the partition 
$\mathcal{S}_{n,d}:=\left\{S_{n,1},S_{n,2},\ldots,S_{n,d}\right\}$ for $\mathbb{N}_{n}$ with $S_{n,j}:=\{p_{j-1}+k:k\in\mathbb{N}_{m_j}\}$, $j\in\mathbb{N}_d$.
Associated with this partition, we decompose $\mathbf{u}:=[u_k:k\in\mathbb{N}_n]\in\mathbb{R}^n$ into $d$ sub-vectors $\mathbf{u}_j:=[u_{p_{j-1}+k}: k\in\mathbb{N}_{m_j}]$, $j\in\mathbb{N}_d$, and decompose $\mathbf{A}$ into $d$ sub-matrices $\mathbf{A}_{[j]}:=[\mathbf{A}_{(p_{j-1}+k)}:k\in \mathbb{N}_{m_j}]\in\mathbb{R}^{n\times m_j},\ j\in\mathbb{N}_d$. 
Moreover, for each $j\in\mathbb{N}_d$, we choose 
the nature partition $\mathcal{S}_{m_j,m_j}:=\left\{S_{m_j,1},S_{m_j,2},\ldots,S_{m_j,m_j}\right\}$ for $\mathbb{N}_{m_j}$. That is, $S_{m_j,k}:=\{p_{j-1}+k\}$, $ k\in\mathbb{N}_{m_j}$. Accordingly, we decompose vector $\mathbf{u}_j$ into $m_j$ sub-vectors $\mathbf{u}_{j,k}:=u_{p_{j-1}+k}$ and decompose matrix $\mathbf{A}_{[j]}$ into $m_j$ sub-matrices $
\mathbf{A}_{[j,k]}:=\mathbf{A}_{(p_{j-1}+k)}$, $ k\in\mathbb{N}_{m_j}$. 
It follows from the orthogonality of matrix $\mathbf{A}$ that conditions \eqref{S_nd_block_separable} and \eqref{S_mj_qj_block_separable} are satisfied. This allows us to choose the regularization parameters according to the strategy stated in Corollary \ref{block_separabel_parameter_choice}. We note that if $\mathbf{u}_{\bm{\lambda}}$ is a solution of problem \eqref{lasso_multi-parameter} with $\bm{\lambda}:=[\lambda_j:j\in\mathbb{N}_d]$, then for each $j\in\mathbb{N}_d$, the $\mathcal{S}_{m_j,m_j}$-block sparsity level of $(\mathbf{u}_{\bm{\lambda}})_j$ coincides with its sparsity level. 

We first validate the  parameter choice strategy stated in 
Corollary \ref{block_separabel_parameter_choice}. We set three prescribed TSLs values $[512$, $17$, $10$, $15$, $2$, $8$, $3$, $8$, $15$, $10]$,  $[512$, $90$, $60$, $10$, $12$, $12$, $18$, $15$, $12$, $9]$ and $[512$, $25$, $10$, $15$, $14$, $17$, $37$, $62$, $146$, $262]$. According to the strategy stated in Corollary \ref{block_separabel_parameter_choice}, we select the parameter $\bm{\lambda}^*:=[\lambda_j^*:j\in\mathbb{N}_d]$ with which model \eqref{lasso_multi-parameter} has solutions having the target sparsity levels. For each $j\in\mathbb{N}_{d}$, we set 
$$
\gamma_{j,k}:=\big|(\mathbf{A}_{[j,k]})^{\top}\mathbf{x}\big|,\ \ k\in\mathbb{N}_{m_j}
$$ 
and rearrange them in a nondecreasing order: 
$$
\gamma_{j,k_1}\leq \gamma_{j,k_2}\leq \cdots\leq \gamma_{j,k_{m_j}}\ \
\mbox{with}\ \ \{k_1,  k_2, \ldots, k_{m_j}\}=\mathbb{N}_{m_j}.
$$
We then choose $\lambda^*_j:=\gamma_{j,k_{m_j-l_j^*}}$ if $l^*_j<m_j$ and $\lambda^*_j:=0.1\gamma_{j,k_1}$ if $l^*_j=m_j$. We solve model \eqref{lasso_multi-parameter} with each selected value of $\bm{\lambda}^*$ for the corresponding solution $\mathbf{u}_{\bm{\lambda}^*}$ and determine the actual sparsity level SLs of $\mathbf{u}_{\bm{\lambda}^*}$. 

In this experiment, we solve model \eqref{lasso_multi-parameter} using Algorithm \ref{FPPA1} to obtain the solution $\mathbf{u}_{\bm{\lambda}^*}$. Specifically, $\mathbf{u}_{\bm{\lambda}^*}$ is directly derived from the output $(\mathbf{w},\hat{\mathbf{a}},\mathbf{c})$ of Algorithm \ref{FPPA1}, that is, $\mathbf{u}_{\bm{\lambda}^*}=\mathbf{w}$. In Algorithm \ref{FPPA1}, the function $\bm{\psi}$ is defined by \eqref{fidelity_term}, $\mathbf{B}'=\mathbf{I}_n$,  $\mathbb{M}=\mathbb{R}^{n}$, and for each $j\in\mathbb{N}_d$, $\lambda_j=\lambda^*_j$, $\mathbf{I}^{'}_{j}$ is defined by \eqref{I-j} with $p_d=n=r$. The parameters are set as $\alpha:=99$,  $\rho:=0.01$, and $\beta:=0.0001$. 
The closed-form formulas of the proximity operators $\mathrm{prox}_{\alpha\sum_{j\in\mathbb{N}_d}\lambda_j\|\cdot\|_{1}\circ\mathbf{I}^{'}_{j}}$ and $\mathrm{prox}_{\frac{1}{\rho}\bm{\psi}}$ are provided in \eqref{proximity-operator-sum_lambda_j1} and \eqref{Au-2-norm-orthogonal}, respectively. Additionally, from equation \eqref{indicator_prox}, we have that  $\mathrm{prox}_{\frac{1}{\beta}\iota_{\mathbb{M}}}=\mathbf{I}_n$.     

We report in Table \ref{Multi_parameter_signal_DWT_block} the numerical results of this experiment: the targeted sparsity levels TSLs, the selected values of parameter $\bm{\lambda}^*$, the actual sparsity levels SLs of  $\mathbf{u}_{\bm{\lambda}^*}$, the Ratio values of $\mathbf{u}_{\bm{\lambda}^*}$ and the $\mathrm{MSE}$ values of the denoised signals $\mathbf{A}\mathbf{u}_{\bm{\lambda}^*}$. Here and in the next subsection, 
$$
\mathrm{MSE}:=\frac{1}{n}\|\mathbf{A}\mathbf{u}_{\bm{\lambda}^*}-\mathbf{x}\|_2^2.
$$
Observing from Table \ref{Multi_parameter_signal_DWT_block}, the SLs values match with the TSLs values.  The numerical results in Table \ref{Multi_parameter_signal_DWT_block} demonstrate the efficacy of the strategy stated in Corollary \ref{block_separabel_parameter_choice} in selecting regularization parameters with which the corresponding solution achieves a desired sparsity level and preserves the approximation accuracy. 

To visually demonstrate the denoising effect, we present the experimental results graphically. Figure \ref{Orignal_Noisy_Block} compares the original and noisy signals, highlighting the characteristics of noise interference. Figure \ref{Orignal_Denoised_Block} shows the denoised signals reconstructed by model \eqref{lasso_multi-parameter} for the three TSLs listed in Table \ref{Multi_parameter_signal_DWT_block}, using regularization parameters selected according to Corollary \ref{block_separabel_parameter_choice}. By comparing Figures \ref{Orignal_Noisy_Block} and \ref{Orignal_Denoised_Block}, the method’s effectiveness in waveform recovery and noise suppression under different sparsity requirements can be intuitively observed.

\begin{table}[ht]
% \vspace{-0.4cm}
\caption{Multi-parameter choices $\lambda^*_j:=\gamma_{j,k_{m_j-l_j^*}}$ for model \eqref{lasso_multi-parameter} (block separable)}
\label{Multi_parameter_signal_DWT_block} 
\begin{tabular}{l||c|c|c}
\hline
TSLs &$[512, 17,10,15,2,$ &$[512,90,60,10,12,$ 
&$[512,25,10,15,14,$  \\ 
&$8,3,8,15,10]$
&$12,18,15,12,9]$ 
&$17,37,62,146,262]$  \\
$\bm{\lambda}^*$ ($\times 10^{-2}$)
&$[0.22,20.60,8.99,7.75,10.22,$  
&$[0.22,4.32,5.97,8.33,9.08,$ 
&$[0.22,8.92,8.99,7.75,8.78,$ \\
&$9.28,10.48,10.47,10.68,11.69]$ &$9.10,9.67,10.06,10.79,11.71]$
&$8.92,8.98,9.00,9.00,9.00]$  \\
$\mathrm{SLs}$  
&$[512,17,10,15,2,$  &$[512,90,60,10,12,$ 
&$[512,25,10,15,14,$  \\ 
&$8,3,8,15,10]$  
&$12,18,15,12,9]$ 
&$17,37,62,146,262]$  \\ 
   
Ratio (\%) &$0.23$ &$0.29$ &$0.42$ \\
$\mathrm{MSE}$ ($\times10^{-6}$) &$5.88$ &$2.59$ &$3.46$  \\  
\hline
\end{tabular}
\end{table}

\begin{figure}[h]
\centering
\begin{tabular}{{c@{\hspace{15pt}}c}}
\includegraphics[width=5cm,height=4cm]{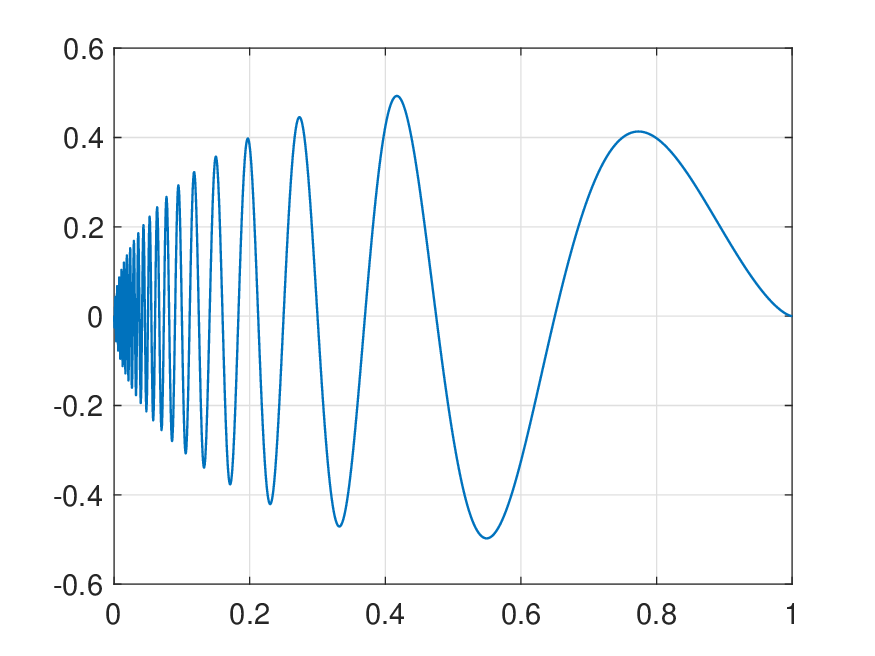}
&\includegraphics[width=5cm,height=4cm]{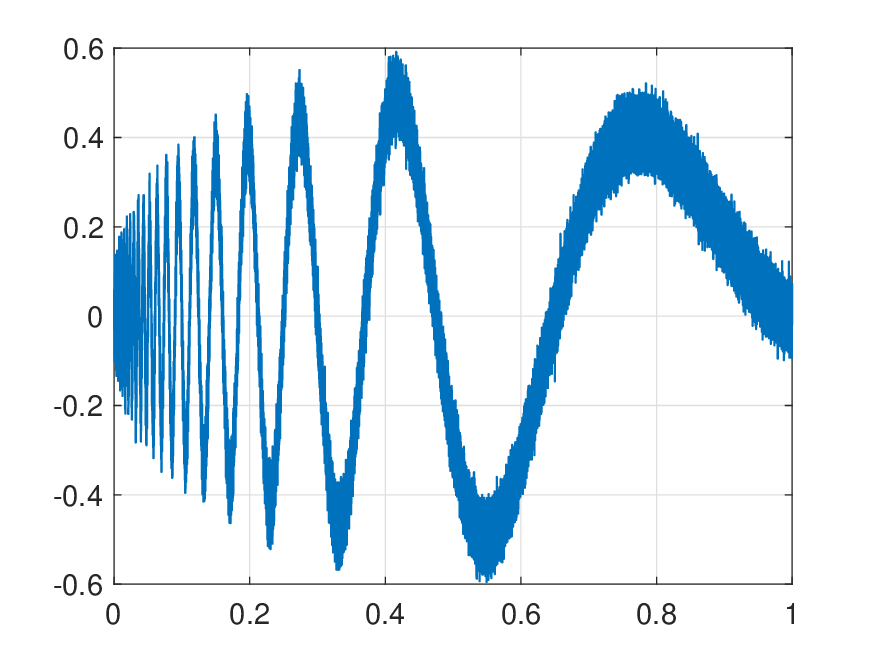}
\\
{\scriptsize (a)} & {\scriptsize (b)} 
\end{tabular}
\\
\caption{(a) Original signal; (b) Noisy signal.} 
\label{Orignal_Noisy_Block}
\end{figure}

\begin{figure}[h]
\centering
\begin{tabular}{c@{\hspace{5pt}}c@{\hspace{5pt}}c}
\includegraphics[width=5cm,height=4cm]{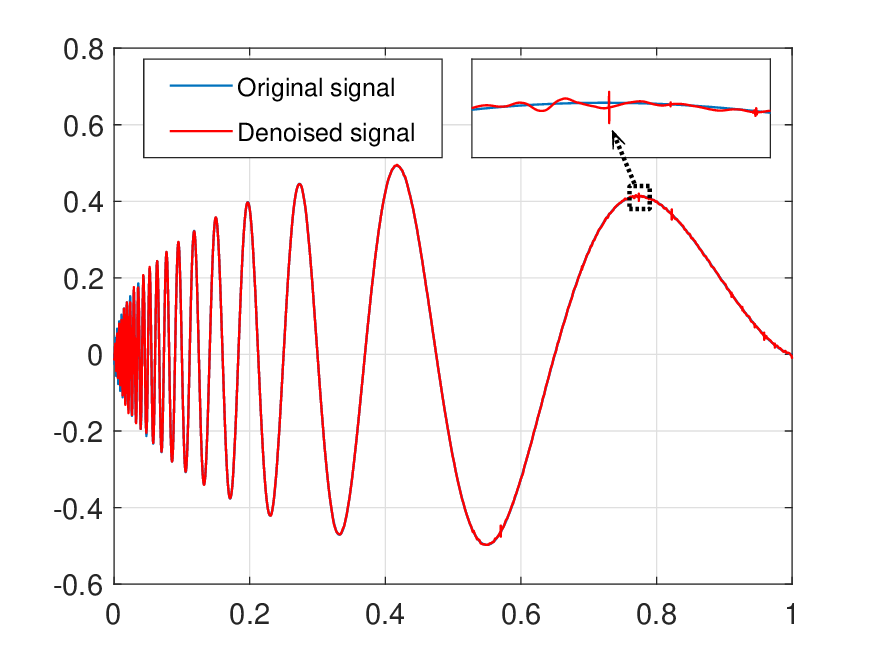} &
\includegraphics[width=5cm,height=4cm]{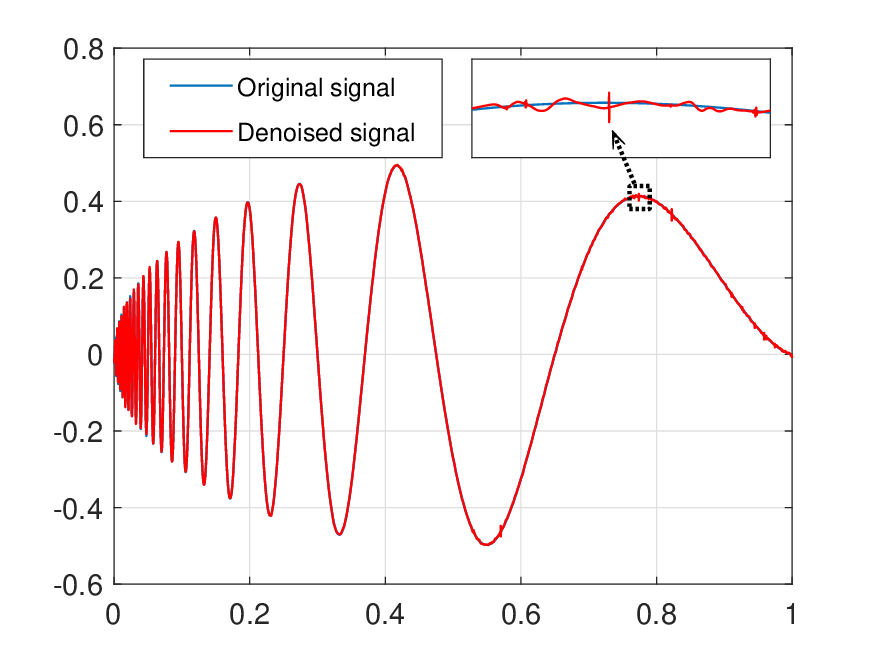} &
\includegraphics[width=5cm,height=4cm]{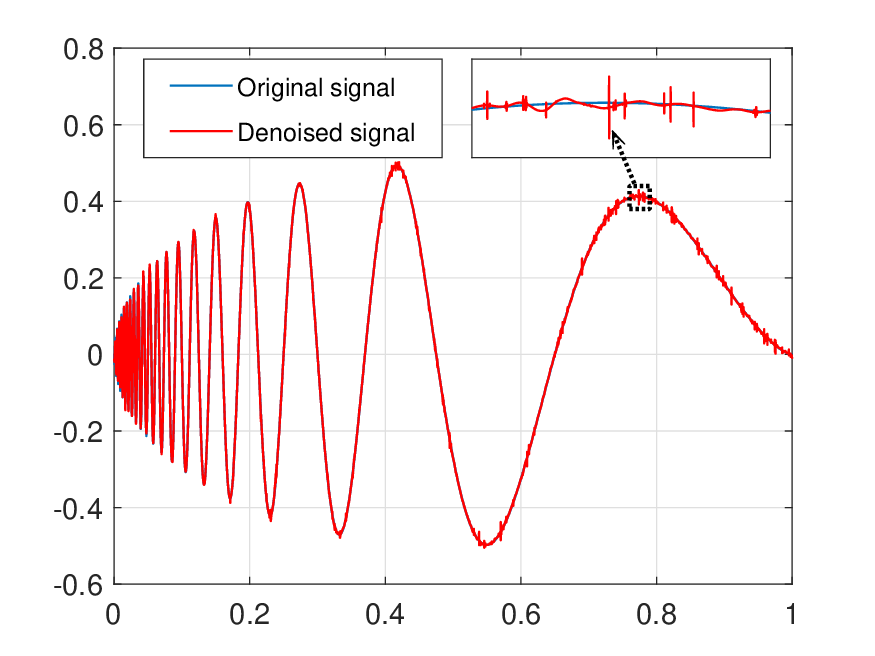} \\[6pt]
{\scriptsize (a)} & {\scriptsize (b)} & {\scriptsize (c)}
\end{tabular}
\caption{Original and denoised signals under sparsity ratios: (a) $0.23\%$; (b) $0.29\%$; (c)  $0.42\%$.}
\label{Orignal_Denoised_Block}
\end{figure}

For the comparison purpose, we also consider the single-parameter regularization model \eqref{lasso_multi-parameter} with $d:=1$. Let 
$$
\gamma_k:=|(\mathbf{A}_{(k)})^{\top}\mathbf{x}|, 
\ \ k\in\mathbb{N}_n,
$$
rearranged in a nondecreasing order: 
$$
\gamma_{k_1}\leq \gamma_{k_2}\leq \cdots \leq \gamma_{k_n}\ \  \mbox{with}\ \ \{k_1,k_2,\cdots,k_n\}=\mathbb{N}_n.
$$
For three TSL values $600$, $750$ and $1100$, we choose $\lambda^*:=\gamma_{k_{n-l^*}}$ according to Corollary \ref{block_separabel_parameter_choice} with $d:=1$. By solving model  \eqref{lasso_multi-parameter} with $d:=1$ via Algorithm \ref{FPPA1} with identical parameters $\alpha$, $\rho$ and $\beta$ as those used in the multi-parameter regularization model, we obtain the solution $\mathbf{u}_{\lambda^*}$ and determine
the actual sparsity level SL of $\mathbf{u}_{\lambda^*}$. The targeted sparsity levels TSL,
the selected values of parameter $\lambda^*$, the actual sparsity levels $\mathrm{SL}$ of  $\mathbf{u}_{\lambda^*}$, the Ratio values of $\mathbf{u}_{\lambda^*}$ and the $\mathrm{MSE}$ values of the denoised signals $\mathbf{A}\mathbf{u}_{\lambda^*}$ are reported in Table \ref{Single_parameter_signal_DWT_block}.
% Here, the $\mathrm{ERR}$ error is compared to the minimal norm solution 
% \begin{equation}\label{mini_norm}
% \mathbf{u}^{\dag}:=\argmin\{\|\mathbf{u}\|_1:\mathbf{A}\mathbf{u}
% =\mathbf{x},\mathbf{u}\in\mathbb{R}^n\}
% \end{equation}
% of the prediction problem. 
Observing from Tables \ref{Multi_parameter_signal_DWT_block} and   \ref{Single_parameter_signal_DWT_block}, compared with the single-parameter regularization model, the multi-parameter regularization model may provide a solution having better approximation
accuracy with the same sparsity level.

\begin{table}[ht]
% \vspace{-0.4cm}
\caption{\label{Single_parameter_signal_DWT_block} Single-parameter choices $\lambda^*:=\gamma_{k_{n-l^*}}$ for  model \eqref{lasso_multi-parameter} (block separable)}
\setlength{\tabcolsep}{9.2mm
\begin{tabular}{l||c|c|c}
\hline
TSL &$600$ &$750$ &$1100$   \\   
${\lambda}^*$ ($\times 10^{-2}$) &$10.69$ &$9.80$
&$8.98$\\   
$\mathrm{SL}$ &$600$
&$750$
&$1100$\\ 
Ratio (\%) &$0.23$
&$0.29$
&$0.42$\\
$\mathrm{MSE}$ ($\times10^{-5}$)
&$2.61$ &$2.25$
&$1.94$\\  
\hline
\end{tabular}}
\end{table}

\subsection{Parameter choices: a non-separable case}\label{NM_non_separable}
This experiment is devoted to validating the efficacy of Algorithm \ref{Sparsity-Guided Multi-Parameter Selection Algorithm}. We again consider recovering
the Doppler signal function defined by \eqref{Doppler_signal} from its noisy data by the multi-parameter regularization model \eqref{lasso_multi-parameter}. The original signal $\mathbf{f}$ and the noisy signal $\mathbf{x}$ are chosen in the same way as in Subsection  \ref{NM_block_separable}.

In this experiment, the matrix $\mathbf{A}\in\mathbb{R}^{n\times n}$ is determined by the biorthogonal wavelet `bior2.2' available in Matlab with the coarsest resolution level $\mathrm{L}:=9$. In both the analysis and synthesis filters, `bior2.2' possesses 2 vanishing moments. Such  a matrix does not satisfies conditions \eqref{S_nd_block_separable} and \eqref{S_mj_qj_block_separable}. As a result, we choose the regularization parameters by employing Algorithm \ref{Sparsity-Guided Multi-Parameter Selection Algorithm}. The number $d$ of the regularization parameters, the sub-vectors $\mathbf{u}_j$, $j\in\mathbb{N}_d$ of 
a vector $\mathbf{u}\in\mathbb{R}^n$ and the sub-matrix $\mathbf{A}_{[j]}$, $j\in\mathbb{N}_d$ of $\mathbf{A}$ are all defined as in Subsection  \ref{NM_block_separable}.
% We solve model \eqref{error_regur_model_q=2} using FPPA with $\mathrm{prox}_{\alpha\Phi}$ and $\mathrm{prox}_{\frac{1}{\rho}\Psi}$ at $\mathrm{z}\in\mathbb{R}^n$, as described in subsection \ref{NM_block_separable}, to obtain a numerical solution $\mathbf{u}_{\bm{\lambda}}^{\delta}$.

We set three prescribed TSLs values $[512$, $17$, $10$, $15$, $2$, $8$, $3$, $8$, $15$, $10]$,  $[512$, $90$, $60$, $10$, $12$, $12$, $18$, $15$, $12$, $9]$ and $[512$, $25$, $10$, $15$, $14$, $17$, $37$, $62$, $146$, $262]$. We choose the parameters $\lambda_j^*$, $j\in\mathbb{N}_d$ by Algorithm \ref{Sparsity-Guided Multi-Parameter Selection Algorithm} with $\bm{\psi}$ defined by \eqref{fidelity_term}, $\mathbf{B}'=\mathbf{I}_n$ and $\epsilon=10$. In Algorithm \ref{Sparsity-Guided Multi-Parameter Selection Algorithm}, we choose the initial parameter $\lambda_j^0=\big\|\mathbf{A}_{[j]}^{\top}{\mathbf{x}}\big\|_{\infty}$, $j\in\mathbb{N}_d$ and solve model \eqref{lasso_multi-parameter} by Algorithm \ref{FPPA1} with $\alpha:=49.50$, $\rho:=0.02$ and $\beta:=0.0001$. The operator $\mathrm{prox}_{\alpha\sum_{j\in\mathbb{N}_d}\lambda_j\|\cdot\|_{1}\circ\mathbf{I}^{'}_{j}}$is given by \eqref{proximity-operator-sum_lambda_j1}, while  $\mathrm{prox}_{\frac{1}{\rho}\bm{\psi}}$ follows from \eqref{Au-2-norm}. Additionally, for the indicator function $\iota_{\mathbb{M}}$, we have that $\mathrm{prox}_{\frac{1}{\beta}\iota_{\mathbb{M}}}=\mathbf{I}_n$. The targeted sparsity levels TSLs, the selected values of parameter $\bm{\lambda}^*$ chosen by Algorithm \ref{Sparsity-Guided Multi-Parameter Selection Algorithm}, the actual sparsity levels $\mathrm{SLs}$ of $\mathbf{u}_{\bm{\lambda}^*}$, the Ratio values of $\mathbf{u}_{\bm{\lambda}^*}$, the numbers $\mathrm{NUM}$ of iterations for $\bm{\lambda}^*$, the computational time $\mathrm{Time(s)}$ and the $\mathrm{MSE}$ values of the denoised signals $\mathbf{A}\mathbf{u}_{\bm{\lambda}^*}$ are reported in Table \ref{Multi_parameter_signal_DWT}. For the three TSLs values, the algorithm reaches the stopping criteria within $5$, $4$ and $10$ iterations, respectively. The SLs values obtained by Algorithm \ref{Sparsity-Guided Multi-Parameter Selection Algorithm} match with the TSLs values within tolerance error $\epsilon=10$. The numerical results in
Table \ref{Multi_parameter_signal_DWT}  validate the
efficacy of Algorithm \ref{Sparsity-Guided Multi-Parameter Selection Algorithm} for obtaining regularization parameters leading to a solution with
desired sparsity level and approximation error. Figure \ref{Orignal_Noisy_Block} compares the original signal with its noisy version. Using the three TSLs listed in Table \ref{Multi_parameter_signal_DWT}, Figure \ref{Orignal_Noisy_Nonseparable} shows the denoised signals obtained from model \eqref{lasso_multi-parameter} with regularization parameters selected via Algorithm \ref{Sparsity-Guided Multi-Parameter Selection Algorithm}.

% \begin{table}[ht]
% \caption{\label{Multi_parameter_signal_DWT} Multi-parameter choices by Algorithm \ref{Sparsity-Guided Multi-Parameter Selection Algorithm} for model \eqref{lasso_multi-parameter} (nonseparable)}
% \begin{tabular}{l||c|c|c}
% \hline
% $\mathrm{TSLs}$ 
% &$[512,118,120,300,500$  &$[512,248,400,640,1200$ &$[512,428,400,660,1200$  \\  
% &$650,800,1600,2000,8400]$ &$1600,2000,3600,7800,12000]$ &$1400,1800,3600,15000,35000]$ \\     
% ${\bm{\lambda}}^*$  
% &$[0.13,3.37,4.08,3.50,3.76,4.31,$ &$[0.18,1.85,2.03,2.27,2.41,3.09,$ &$[0.01,0.50,1.97,2.23,2.41,3.29,$  \\ 
% &$4.97,5.08,6.07,6.51]\times 10^{-2}$ &$3.80,4.06,4.30,5.92]\times 10^{-2}$ &$3.92,4.03,3.27,3.86]\times 10^{-2}$ \\
% $\mathrm{SLs}$  
% &$[512,118,120,296,498,$  
% &$[512,248,400,640,1200,$ 
% &$[512,427,400,660,1198,$  \\  
% &$650,800,1598,2000,8400]$ &$1599,1999,3599,7800,11999]$ &$1398,1799,3598,14999,34999]$\\ 
% Ratio  &$5.72\% $ &$11.44\% $ &$22.88\% $ \\
% $\mathrm{NUM}$  &$4$ &$7$ &$6$  \\
% $\mathrm{Time(s)}$  &$52$ &$91$ &$76$  \\
% $\mathrm{MSE}$  &$1.61\times10^{-5}$ &$3.57\times10^{-5}$ &$8.77\times10^{-5}$  \\  
% \hline
% \end{tabular}
% \end{table}

\begin{table}[ht]
\caption{\label{Multi_parameter_signal_DWT} Multi-parameter choices by Algorithm \ref{Sparsity-Guided Multi-Parameter Selection Algorithm} for model \eqref{lasso_multi-parameter} (nonseparable)}
\begin{tabular}{l||c|c|c}
\hline
$\mathrm{TSLs}$ 
&$[512,17,10,15,2,$  &$[512,90,60,10,12,$ &$[512,25,10,15,14,$  \\  
&$8,3,8,15,10]$ &$12,18,15,12,9]$ &$17,37,62,146,262]$ \\     
${\bm{\lambda}}^*$ ($\times10^{-2}$) 
&$[0.12,27.18,17.52,7.87,8.29,$ &$[0.08,4.28,5.14,7.38,7.41,$ &$[0.12,16.06,14.73,7.36,7.40,$  \\ 
&$8.28,9.75,9.32,10.41,13.75]$ &$8.05,8.45,9.11,10.65,13.75]$ &$7.60,7.78,8.10,8.71,10.91]$ \\
$\mathrm{SLs}$  
&$[512,15,8,9,2,$  
&$[511,89,59,9,12,$ 
&$[512,25,8,15,14,$  \\  
&$8,3,8,15,10]$ &$11,18,15,12,9]$ &$16,37,62,146,262]$\\ 
Ratio (\%) &$0.23 $ &$0.28 $ &$0.42 $ \\
$\mathrm{NUM}$  &$5$ &$4$ &$10$  \\
Time (s) &$61$ &$47$ &$125$  \\
$\mathrm{MSE}$ ($\times10^{-6}$)  &$13.13$ &$3.69$ &$8.68$  \\  
\hline
\end{tabular}
\end{table}

\begin{figure}[h]
\centering
\begin{tabular}{c@{\hspace{5pt}}c@{\hspace{5pt}}c}
\includegraphics[width=5cm,height=4cm]{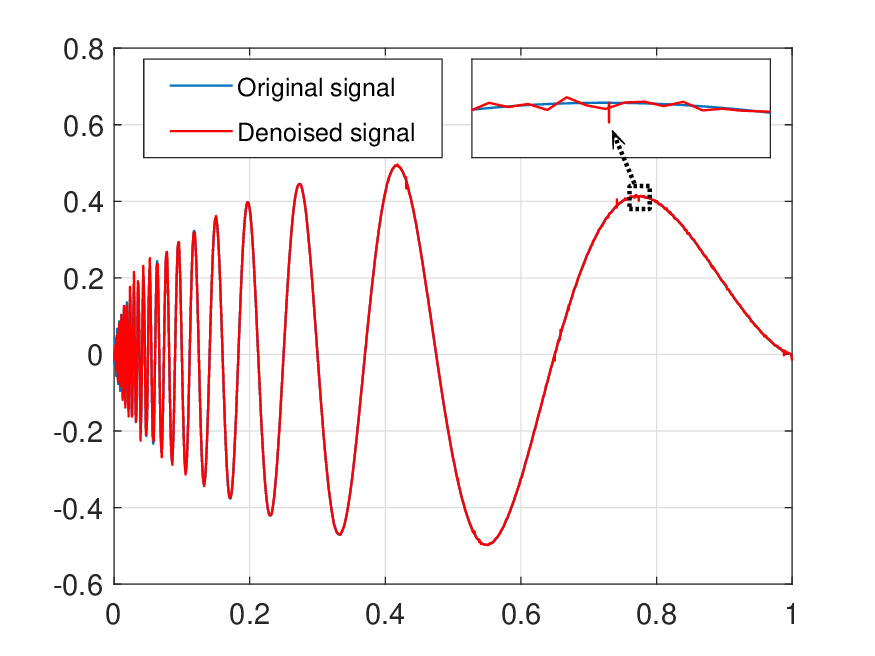} &
\includegraphics[width=5cm,height=4cm]{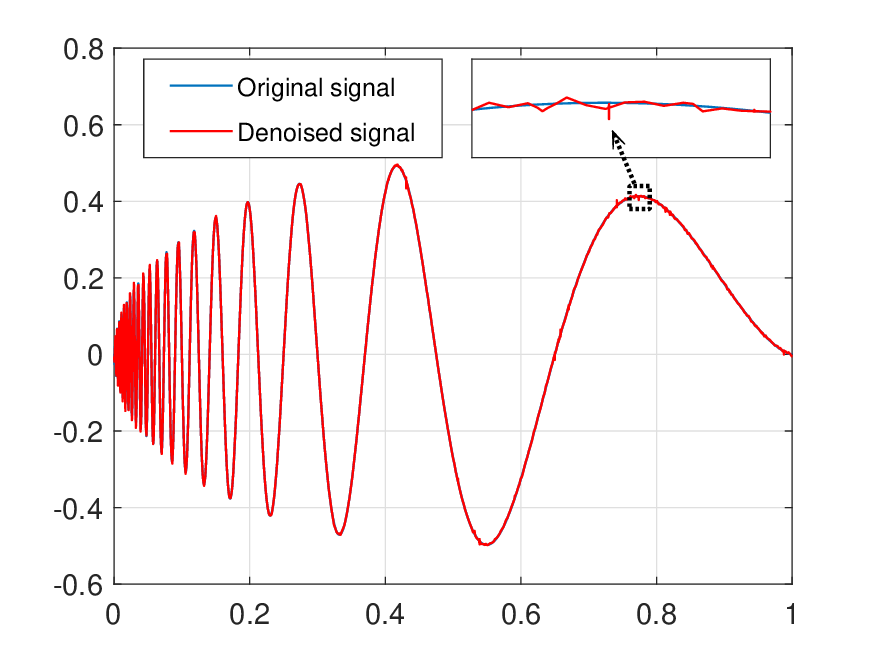} &
\includegraphics[width=5cm,height=4cm]{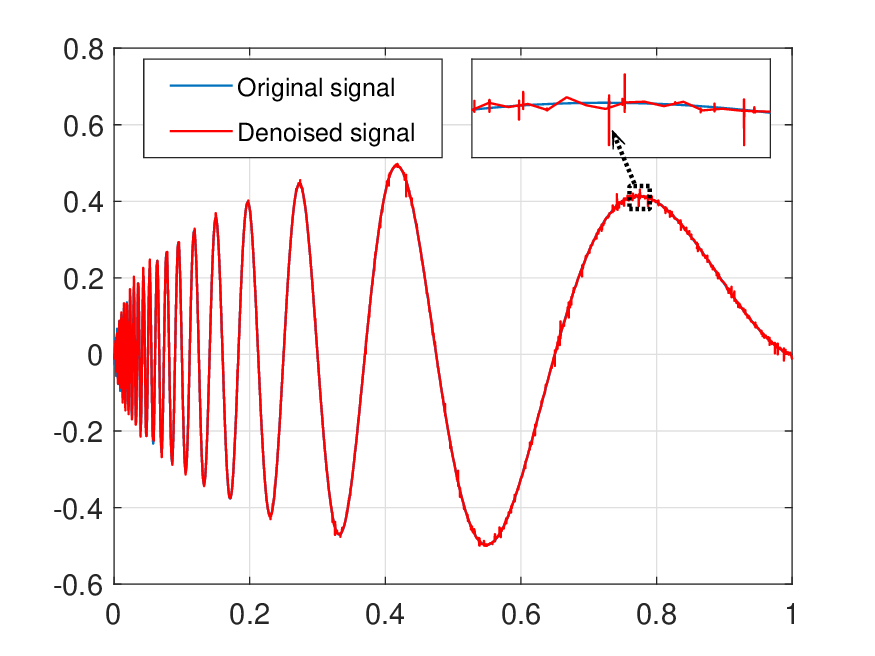} \\[6pt]
{\scriptsize (a)} & {\scriptsize (b)} & {\scriptsize (c)}
\end{tabular}
\caption{Original and denoised signals under sparsity ratios: (a) $0.23\%$; (b) $0.28\%$; (c)  $0.42\%$.}
\label{Orignal_Noisy_Nonseparable}
\end{figure}

\subsection{Compound sparse denoising}\label{Compound sparse denoising}
We consider in this experiment parameters choices for the compound sparse denoising regularization model \eqref{CSD} implemented by Algorithm \ref{Sparsity-Guided Multi-Parameter Selection Algorithm}. In this case, the fidelity term $\bm{\psi}$ defined by \eqref{fidelity-CSD} is differentiable and the transform matrix $\mathbf{B}:=\scalebox{0.8}{$\begin{bmatrix}
\mathbf{I}_n\\
\mathbf{D}\end{bmatrix}$}$ satisfies $rank(\mathbf{B})=n$. Clearly, $\mathbf{B}$ does not have full row rank.

In this experiment, the noisy data $\mathbf{y}:=[y(t):t\in\mathbb{N}_{n}]$ with $n=300$ is generated by adding a low-frequency sinusoid signal $f(t):=\sin(0.021\pi t)$, two additive step discontinuities $u(t)=2$ for  $t<0.3n$, $1$ for $t>0.6n$ and $0$ otherwise, and additive white Gaussian noise $\eta(t)$ with 0.3 standard deviation and zero mean. Set $\mathbf{f}:=[f(t):t\in\mathbb{N}_{n}]$ and $\mathbf{u}:=[u(t):t\in\mathbb{N}_{n}]$. We estimate $\mathbf{f}$ and $\mathbf{u}$ from the noisy data $\mathbf{y}$ by using simultaneous low-pass filtering and compound
sparse denoising. The estimate $\mathbf{u}^*$ of $\mathbf{u}$ is obtained by solving the compound
sparse denoising model \eqref{CSD}, where the $(n-4)\times n$ high-pass filter matrix $\mathbf{H}$ 
is chosen as in Example A of \cite{Selesnick2014Simultaneous}.
% the fourth-order high-pass filter $\mathbf{H}\in\mathbb{R}^{(n-2m)\times n}$ with $m=2$ and cut-off frequency $\omega_c=0.044\pi$. is chosen as in \cite{Selesnick2014Simultaneous}. 
The estimate $\mathbf{f}^*$ of  $f$ is then obtained by  $\mathbf{f}^*=(\tilde{\mathbf{I}}-\mathbf{H})(\mathbf{y}-\mathbf{u}^*)$, where $\tilde{\mathbf{I}}$ is the identity matrix of order $n$ with the first $2$ and last $2$ rows removed. To measure the filtering effect, we define the $\mathrm{MSE}$ by
$$
\mathrm{MSE}:=\frac{1}{n-4}\big\|\tilde{\mathbf{I}}(\mathbf{f}+\mathbf{u})-\mathbf{f}^*-\tilde{\mathbf{I}}\mathbf{u}^*\big\|_2^2.
$$
 
We employ Algorithm \ref{Sparsity-Guided Multi-Parameter Selection Algorithm} to select multiple regularization parameters that ensures
the resulting solution of model \eqref{CSD} achieves given TSLs. In Algorithm \ref{Sparsity-Guided Multi-Parameter Selection Algorithm}, we solve model \eqref{optimization_problem_under_Bj_3} by Algorithm \ref{FPPA1} where $d=2$, $\bm{\psi}$ has the form \eqref{fidelity-CSD}, 
$\mathbf{B}_{1}$ is the identity matrix of order $n$ and $\mathbf{B}_{2}$ is the $(n-1)\times n$ first order difference matrix. 
% Recall that $p_d=2n-1$. Moreover, we choose $\theta:=1$, $\mathbf{O}:=\frac{1}{\alpha}\mathbf{I}_{p_d}$, $\mathbf{P}:=\frac{1}{\rho}\mathbf{I}_{n}$ and $\mathbf{Q}:=\frac{1}{\beta}\mathbf{I}_{p_d}$. 
To ensure the convergence of Algorithm \ref{FPPA1}, we set $\alpha:=0.1$, $\rho:=4$ and $\beta:=4$. The closed-form formulas of the proximity operators $\mathrm{prox}_{\alpha\sum_{j\in\mathbb{N}_d}\lambda_j\|\cdot\|_{1}\circ\mathbf{I}^{'}_{j}}$, $\mathrm{prox}_{\frac{1}{\beta}\iota_{\mathbb{M}}}$ and $\mathrm{prox}_{\frac{1}{\rho}\bm{\psi}}$ with $\bm{\psi}$ being defined by  \eqref{fidelity-CSD} are given in \eqref{proximity-operator-sum_lambda_j1}, \eqref{indicator_prox} and \eqref{2-norm-2}, respectively. 

We set four prescribed $\mathrm{TSLs}$ values $[20,30]$, $[80,60]$, $[115,50]$ and $[110,90]$ and choose the regularization parameters $\lambda_j^*$, $j\in\mathbb{N}_d$ by employing Algorithm \ref{Sparsity-Guided Multi-Parameter Selection Algorithm} with $\epsilon=2$. We report in Table \ref{filtering_problem} the targeted sparsity levels $\mathrm{TSLs}$, the initial values of $\bm{\lambda}^0$, the selected values of parameter $\bm{\lambda}^*$ chosen by Algorithm \ref{Sparsity-Guided Multi-Parameter Selection Algorithm}, the actual sparsity levels $\mathrm{SLs}$ of $\mathbf{u}^*$, the Ratios values of $\mathbf{u}^*$, the numbers $\mathrm{NUM}$ of iterations for $\bm{\lambda}^*$, the computational time $\mathrm{Time(s)}$ and the $\mathrm{MSE}$ values of the filtered signal. For the five TSLs values, the algorithm meets the stopping criteria after $6$, $13$, $15$, and $5$ iterations, respectively. The SLs values obtained by Algorithm \ref{Sparsity-Guided Multi-Parameter Selection Algorithm} match
with the TSLs values within tolerance error $\epsilon=2$. The numerical results in
Table \ref{filtering_problem} demonstrate the efficacy of Algorithm \ref{Sparsity-Guided Multi-Parameter Selection Algorithm} in selecting multiple regularization parameters
to achieve desired sparsity levels of the solution. Figure \ref{Orignal_Noisy_CSD} compares the original signal with its noisy counterpart. Using the four TSLs listed in Table \ref{filtering_problem}, Figure \ref{Orignal_Denoised_CSD} shows the denoised signals obtained from model \eqref{CSD} with regularization parameters selected via Algorithm \ref{Sparsity-Guided Multi-Parameter Selection Algorithm}.

\begin{table}[ht]
\caption{\label{filtering_problem} Multi-parameter choices by Algorithm \ref{Sparsity-Guided Multi-Parameter Selection Algorithm} for model \eqref{CSD}}
\setlength{\tabcolsep}{5mm
\begin{tabular}{l||c|c|c|c}
\hline
TSLs  &$[20,30]$  &$[80,60]$ &$[115,50]$ &$[110,90]$
\\  
$\bm{\lambda}^0$ &$[0.5,0.08]$ &$[0.5,1.0]$ &$[0.07,0.4]$
&$[0.2,0.4]$\\ 
$\bm{\lambda}^*$
&$[0.4854,0.0675]$  &$[0.1982,0.1742]$ 
&$[0.0668,0.3793]$  &$[0.1492,0.1637]$
\\ 
$\mathrm{SLs}$  &$[20,29]$ 
&$[79,59]$ &$[116,51]$ &$[109,89]$
\\ 
Ratios (\%)
&$[6.67,9.70]$  &$[26.33,19.73]$
&$[38.67,17.06]$
&$[36.33,29.77]$
\\
$\mathrm{NUM}$  &$6$ 
&$13$ &$15$ &$5$\\ 
Time (s) &$11$ &$22$ &$29$ &$9$\\
$\mathrm{MSE}$ ($\times 10^{-2}$)
&$2.51$  
&$1.58$ 
&$1.10$
&$1.59$
\\  \hline
\end{tabular}}
\end{table}

% \begin{table}[ht]
% \caption{\label{filtering_problem} Multi-parameter choices by Algorithm \ref{Sparsity-Guided Multi-Parameter Selection Algorithm} for model \eqref{CSD}}
% \begin{tabular}{l||c|c|c|c|c|c|c}
% \hline
% TSLs  &$[10,5]$ &$[20,20]$
% &$[20,30]$  &$[50,40]$ &$[80,60]$
% &$[115,50]$
% &$[110,90]$
% \\  
% $\bm{\lambda}^0$ &$[0.6,1.0]$ &$[0.5,0.5]$
% &$[0.5,0.08]$  &$[0.5,0.4]$ &$[0.5,1.0]$ &$[0.07,0.4]$
% &$[0.2,0.4]$
% \\ 
% $\bm{\lambda}^*$
% &$[0.47,0.62]$  
% &$[0.36,0.21]$ &$[0.49,0.07]$  
% &$[0.24,0.19]$  
% &$[0.20,0.17]$ 
% &$[0.07,0.38]$
% &$[0.15,0.16]$
% \\ 
% $\mathrm{SLs}$  &$[9,6]$  &$[20,18]$ &$[20,29]$ 
% &$[51,40]$ &$[79,59]$
% &$[116,51]$
% &$[109,89]$
% \\ 
% Ratios&$[3.00\%,$ &$[6.67\%, $ &$[6.67\%,$ &$[17.00\%,$ &$[26.33\%,$
% &$[38.67\%,$
% &$[36.33\%,$
% \\
% &$2.01\%] $ &$6.02\%] $ &$9.70\%]$ &$13.38\%]$ &$19.73\%]$
% &$17.06\%]$
% &$29.77\%]$
% \\
% $\mathrm{NUM}$  &$8$  &$8$ &$6$ 
% &$8$ &$13$ &$15$ &$5$\\ 
% $\mathrm{Time(s)}$  &$14$ &$14$ &$11$ &$14$ &$22$ &$29$ &$9$\\
% $\mathrm{MSE}$  &$3.23\times10^{-2}$ &$2.18\times10^{-2}$ &$2.51\times10^{-2}$  
% &$1.67\times10^{-2}$ &$1.58\times10^{-2}$ 
% &$1.10\times10^{-2}$
% &$1.59\times10^{-2}$
% \\  
% \hline
% \end{tabular}
% \end{table}

\begin{figure}[h]
\centering
\begin{tabular}{c@{\hspace{15pt}}c}
\includegraphics[width=5cm,height=4cm]{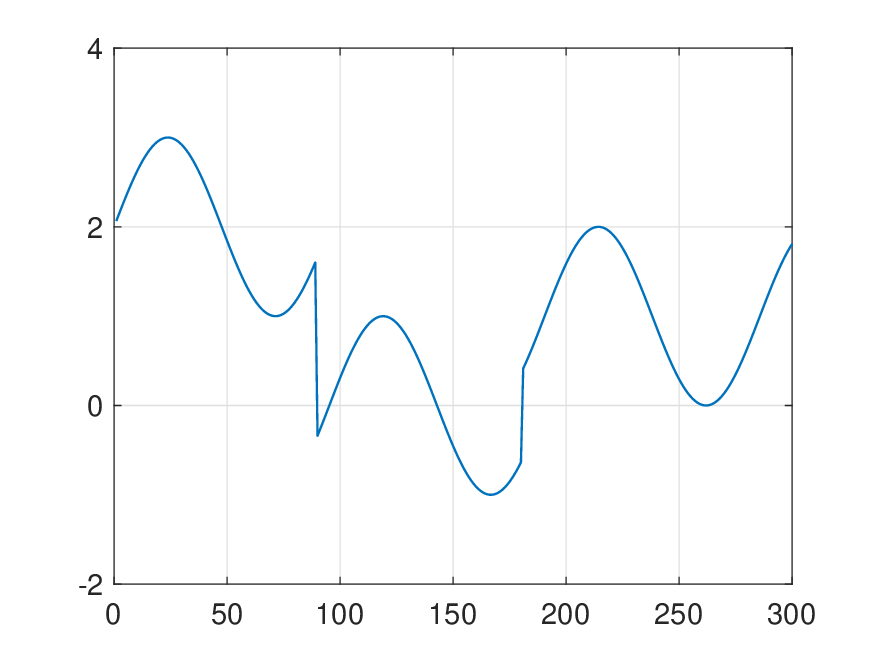}
&\includegraphics[width=5cm,height=4cm]{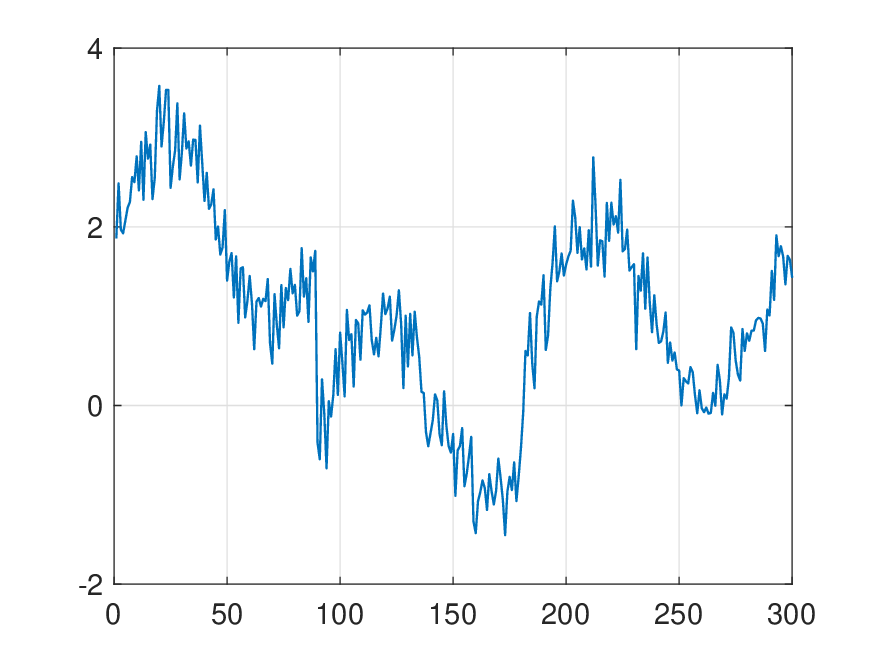}
\\
{\scriptsize (a)} & {\scriptsize (b)} 
\end{tabular}
\\
\caption{(a) Original signal; (b) Noisy signal.} 
\label{Orignal_Noisy_CSD}
\end{figure}

\begin{figure}[h]
\centering
\begin{tabular}{c@{\hspace{15pt}}c}
\includegraphics[width=5cm,height=4cm]{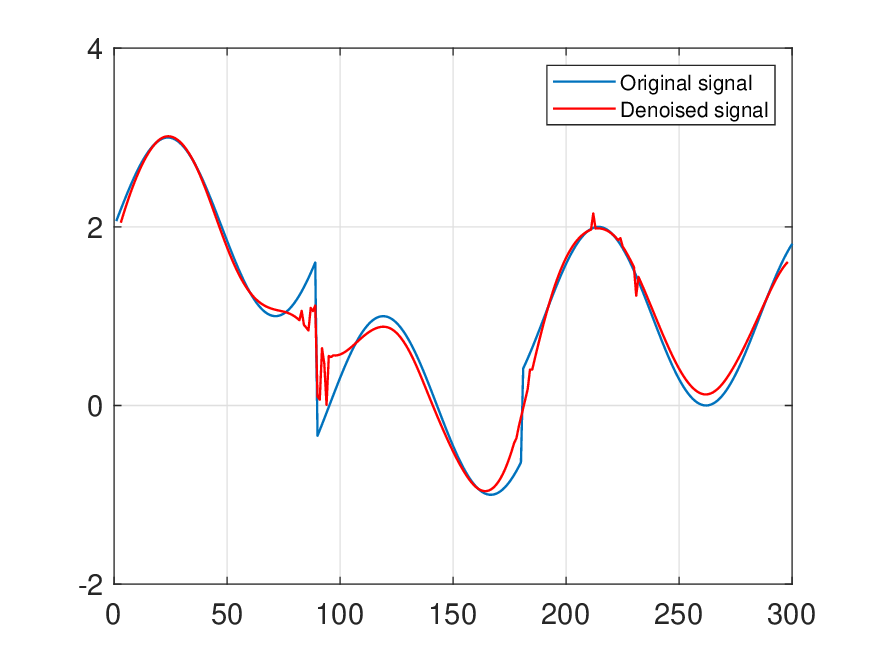} &
\includegraphics[width=5cm,height=4cm]{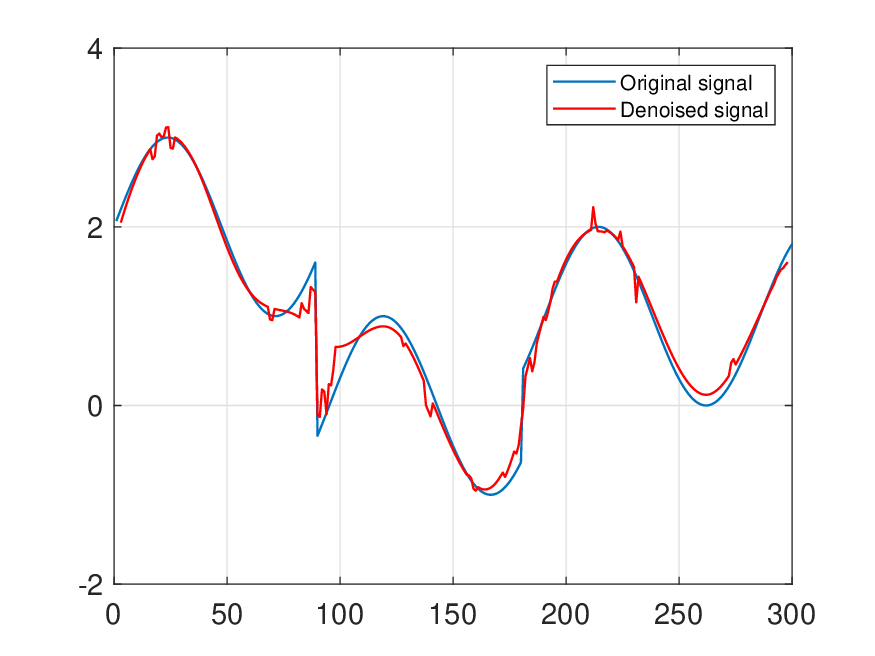} \\[6pt]
{\scriptsize (a)} & {\scriptsize (b)} \\
\includegraphics[width=5cm,height=4cm]{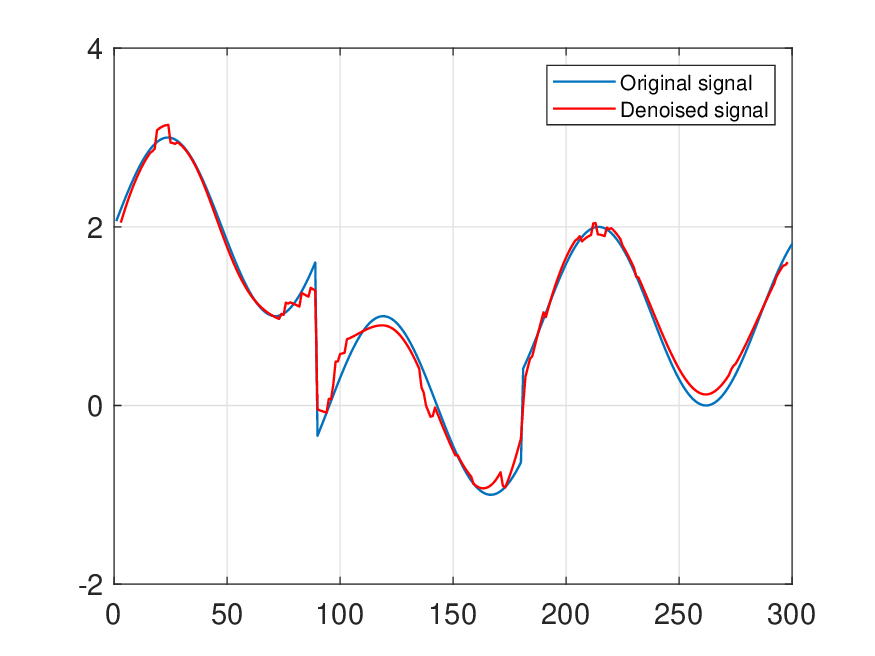} &
\includegraphics[width=5cm,height=4cm]{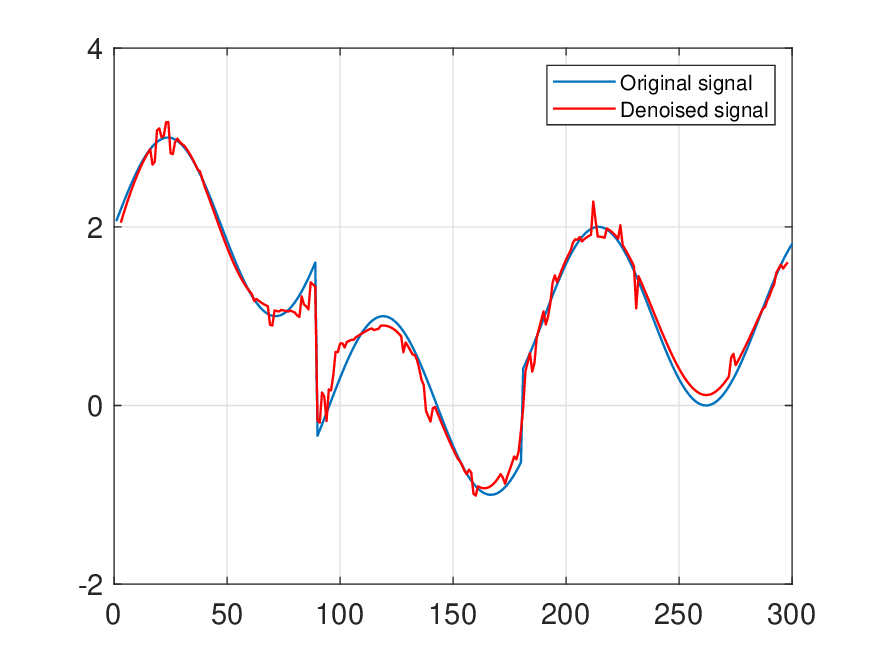} \\[6pt]
{\scriptsize (c)} & {\scriptsize (d)} \\
\end{tabular}
\caption{Original and denoised signals under sparsity ratios: (a)  $[6.67\%,9.70\%]$; 
(b) $[26.33\%,19.73\%]$; 
(c) $[38.67\%,17.06\%]$; 
(d) $[36.33\%,29.77\%]$.}
\label{Orignal_Denoised_CSD}
\end{figure}

\subsection{Fused SVM}\label{Fused SVM}
The goal of this experiment is to validate parameters choices for the
fused SVM model \eqref{fused-SVM} implemented by Algorithm \ref{Sparsity-Guided Multi-Parameter Selection Algorithm}. In model \eqref{fused-SVM}, the fidelity term  $\bm{\psi}(\mathbf{u}):=\bm{\phi}(\mathbf{YXu})$ is not differential and the transform matrix $\mathbf{B}:=\scalebox{0.8}{$\begin{bmatrix}
\mathbf{I}_n\\
\mathbf{D}\end{bmatrix}$}$ does not have full row rank.

The dataset utilized for this experiment is the set of handwriting digits sourced from the modified national institute of standards and technology (MNIST) database \cite{lecun1998gradient}. The original MNIST database consists of 60000 training samples and 10000 testing samples of the digits `0' through `9'. We consider the binary classification problem with
two digits `7' and `9', by taking 12214 training samples and 2037 testing samples of these two digits from the database. Let $p:=12214$ be the number of training samples and $\{y_j:j\in\mathbb{N}_p\}\subset\{-1,1\}$ be the labels of training data in which -1 and 1 represent the digits `7' and `9', respectively. In addition, let $n:=784$ be the the number of pixels in each sample.

We implement Algorithm  \ref{Sparsity-Guided Multi-Parameter Selection Algorithm} to select multiple regularization parameters with which the resulting solution of model \eqref{fused-SVM} achieves given TSLs. In Algorithm  \ref{Sparsity-Guided Multi-Parameter Selection Algorithm}, we solve model \eqref{optimization_problem_under_Bj_3} by Algorithm \ref{FPPA1} where $d=2$, $\bm{\psi}(\mathbf{u}):=\bm{\phi}(\mathbf{YXu})$, $\mathbf{B}_{1}$ is the identity matrix of order $n$ and $\mathbf{B}_{2}$ is the $(n-1)\times n$ first order difference matrix. It is worth noting that the proximity operator of function $\frac{1}{\rho}\bm{\psi}$ can not be expressed explicitly. Instead, we solve model \eqref{optimization_problem_under_Bj_3} by Algorithm \ref{FPPA1} with $\bm{\psi}$ and $\mathbf{B}'$ being replaced by $\bm{\phi}$ and $\mathbf{YX}\mathbf{B}'$, respectively, and obtain at step $k$ three vectors $\widetilde{\mathbf{w}}^k$, $\widetilde{\mathbf{a}}^k$ and $\widetilde{\mathbf{c}}^k$. Then the vectors $\mathbf{w}^k$, $\hat{\mathbf{a}}^k$ and $\mathbf{c}^k$ emerging in Algorithm \ref{Sparsity-Guided Multi-Parameter Selection Algorithm} can be obtained by $\mathbf{w}^k:=\widetilde{\mathbf{w}}^{k}$, $\hat{\mathbf{a}}^k:=(\mathbf{YX})^{\top}\widetilde{\mathbf{a}}^k$ and $\mathbf{c}^k:=\widetilde{\mathbf{c}}^k$. In this experiment, we choose $\alpha:=0.001$, $\rho:=0.002$ and $\beta:=10$ to guarantee the convergence of Algorithm \ref{FPPA1}. 
The closed-form formulas of the proximity operators $\mathrm{prox}_{\alpha\sum_{j\in\mathbb{N}_d}\lambda_j\|\cdot\|_{1}\circ\mathbf{I}^{'}_{j}}$, $\mathrm{prox}_{\frac{1}{\beta}\iota_{\mathbb{M}}}$ and $\mathrm{prox}_{\frac{1}{\rho}\bm{\phi}}$ are given in \eqref{proximity-operator-sum_lambda_j1}, \eqref{indicator_prox} and \eqref{max-function}, respectively.  

We set five prescribed TSLs values $[80,80]$, $[90,110]$, $[150,120]$, $[280,200]$ and $[360,400]$ and use Algorithm \ref{Sparsity-Guided Multi-Parameter Selection Algorithm} with $\epsilon=2$ to select the regularization parameters $\lambda_j^*$,  $j\in\mathbb{N}_d$. The targeted sparsity levels $\mathrm{TSLs}$, the initial values of $\bm{\lambda}^0$, the selected values of parameter $\bm{\lambda}^*$ chosen by Algorithm \ref{Sparsity-Guided Multi-Parameter Selection Algorithm}, the actual sparsity levels $\mathrm{SLs}$ of $\mathbf{u}^*$, the Ratios values of $\mathbf{u}^*$, the numbers $\mathrm{NUM}$ of iterations for $\bm{\lambda}^*$, the computational time $\mathrm{Time(s)}$, the accuracy on the training datasets ($\mathrm{TrA}$) and the accuracy on the testing datasets ($\mathrm{TeA}$) are reported in Table \ref{fused_SVM_experiment}.  Algorithm  \ref{Sparsity-Guided Multi-Parameter Selection Algorithm} meets the stopping criteria $\epsilon=2$ within 12, 5, 14, 13 and 5 iterations
for the five TSLs values, respectively. The SLs values obtained by Algorithm \ref{Sparsity-Guided Multi-Parameter Selection Algorithm} match
with the TSLs values within tolerance error $\epsilon=2$. These results validate the
effectiveness of Algorithm \ref{Sparsity-Guided Multi-Parameter Selection Algorithm} for obtaining multiple regularization parameters leading to a solution with 
desired sparsity levels.

\begin{table}[ht]
\caption{\label{fused_SVM_experiment} Multi-parameter choices by Algorithm \ref{Sparsity-Guided Multi-Parameter Selection Algorithm} for model \eqref{fused-SVM}}
\begin{tabular}{l||c|c|c|c|cc}
\hline
TSLs  &$[80,80]$ &$[90,110]$
&$[150,120]$  &$[280,200]$ &$[360,400]$\\  
$\bm{\lambda}^0$ &$[170,13]$ &$[110,3]$
&$[42,9]$  &$[5,4]$ &$[1,0.2]$\\ 
$\bm{\lambda}^*$
&$[169.9511,12.9933]$  
&$[109.6491,2.9996]$ &$[41.9924,8.9851]$  
&$[4.8402,3.8580]$  
&$[0.9076,0.1747]$ \\ 
$\mathrm{SLs}$  &$[80,79]$  &$[89,109]$ &$[150,118]$ 
&$[280,199]$ &$[359,399]$\\ 
Ratios (\%)   &$[10.20, 10.09] $ &$[11.35,13.92] $ &$[19.13, 15.07]$ &$[35.71, 25.42]$ &$[45.79, 50.96] $\\
$\mathrm{NUM}$  &$12$  &$5$ &$14$ 
&$13$ &$5$\\ 
Time ($\times 10^{3}$ s)  &$6.0544$  &$2.5240$ &$7.0732$ 
&$6.5576$ &$2.5257$\\ 
{TrA} (\%)  &$93.59$ &$94.51$ &$95.15 $  
&$96.50$ &$ 97.32$ \\
{TeA} (\%)  &$94.65$ &$95.24$ &$95.93$  
&$96.51$ &$96.86$ \\
\hline
\end{tabular}
\end{table}

The numerical experiments demonstrate the effectiveness of our proposed algorithms in applications such as signal denoising and classification. Notably, Subsections \ref{NM_block_separable} and \ref{NM_non_separable} involve large-scale signal processing problems with $n=262144$ data points, highlighting our method's capability to handle massive datasets efficiently. Despite the large data sizes, our algorithm achieves the target sparsity levels with reasonable computational time, as shown in Table \ref{Multi_parameter_signal_DWT}. In contrast to the large-scale datasets in Subsection \ref{NM_non_separable}, the dataset used in Subsection \ref{Compound sparse denoising} is relatively small. However, the non-sparse nature of the high-pass filter matrix $\mathbf{H}$ induces high computational complexity in evaluating the proximity operator $\mathrm{prox}_{\frac{1}{\rho}\bm{\psi}}$, leading to significantly increased runtime during parameter selection. The experiment in Subsection \ref{Fused SVM} exhibits significantly longer runtimes compared to those in Subsections \ref{NM_non_separable} and \ref{Compound sparse denoising}, despite using comparable iteration counts. This is because each iteration step requires solving an equivalent fused SVM model with the hinge loss function using the fixed-point proximity algorithm. The non-smoothness and non-strong convexity of the hinge loss function result in slow convergence of the fixed-point proximity algorithm.

\subsection{Numerical verification of Assumptions (A1)--(A3)}

This subsection provides numerical and heuristic evidence supporting the validity of Assumptions (A1)--(A3) within the experimental framework established in Subsection \ref{Compound sparse denoising}. We maintain the configuration of the compound sparse denoising model \eqref{CSD}, utilizing the same observed noisy signal $\mathbf{y}$ and high-pass filter matrix $\mathbf{H}$.

To select the regularization parameters, we apply Algorithm \ref{Sparsity-Guided Multi-Parameter Selection Algorithm} with $\epsilon=0$ and an initial $\bm{\lambda}^0 = [0.6, 0.730]$. The target sparsity levels (TSLs) are set at $[20, 30]$. All algorithmic components remain unchanged from the previous experiment, including the subproblem solver (Algorithm \ref{FPPA1} configured with $d=2$, $\alpha=0.1$, $\rho=4$, and $\beta=4$). Figure \ref{Iter20_30} illustrates the evolution of the parameters and the resulting sparsity levels over seven iterations, while Figure \ref{Iterative_deviation} tracks the deviations of the quantities $\gamma_{j,i_s^k}(\mathbf{u}^k)$ from their terminal values.

\subsubsection*{Analysis of Assumptions}

\begin{itemize}
    \item \textbf{Verification of Assumption (A1):} 
    As shown in Figure \ref{Iter20_30}(a), the sparsity level (SL) for $\lambda_1$ increases monotonically from 11 to 20, never exceeding the target. Similarly, Figure \ref{Iter20_30}(b) shows the SL for $\lambda_2$ rising from 14 to 30. Throughout the process, the condition $l_j^k \leq l_j^*$ holds at every iteration $k$, providing direct empirical confirmation of Assumption (A1).

    \item \textbf{Support for Assumption (A2):} 
    The sequences $\{\lambda_1^k\}$ and $\{\lambda_2^k\}$ in Figure \ref{Iter20_30} exhibit clear monotonicity (strictly decreasing) and remain bounded below by zero. Furthermore, the update magnitude $|\lambda_j^{k+1} - \lambda_j^k|$ diminishes as iterations progress. This behavior---monotonicity, boundedness, and diminishing updates---is characteristic of a convergent sequence. Although the algorithm terminates upon reaching the TSLs, the unambiguous pre-convergence trend strongly suggests that the sequences $\{\lambda_j^k\}$ would converge to a positive limit $\lambda_j > 0$, supporting (A2).

    \item \textbf{Evidence for Assumption (A3):} 
    Figure \ref{Iterative_deviation} displays the absolute differences $|\gamma_{j,i_s^k}(\mathbf{u}^k) - \gamma_{j,i_s^7}(\mathbf{u}^7)|$ for $j=1,2$. These differences decrease rapidly across all indices $i_s$ as $k$ increases. This trend indicates that  for each $j$ and $s$, the sequence $\{\gamma_{j,i_s^k}(\mathbf{u}^k)\}$ converges to $\gamma_{j,i_s^7}(\mathbf{u}^7)$. Treating the terminal solution $\mathbf{u}^7$ as a candidate limit point,  this consistent convergence provides strong heuristic support for the plausibility of Assumption (A3).
\end{itemize}

\begin{figure}[h]
\centering
\begin{tabular}{cc}
\includegraphics[width=6.5cm]{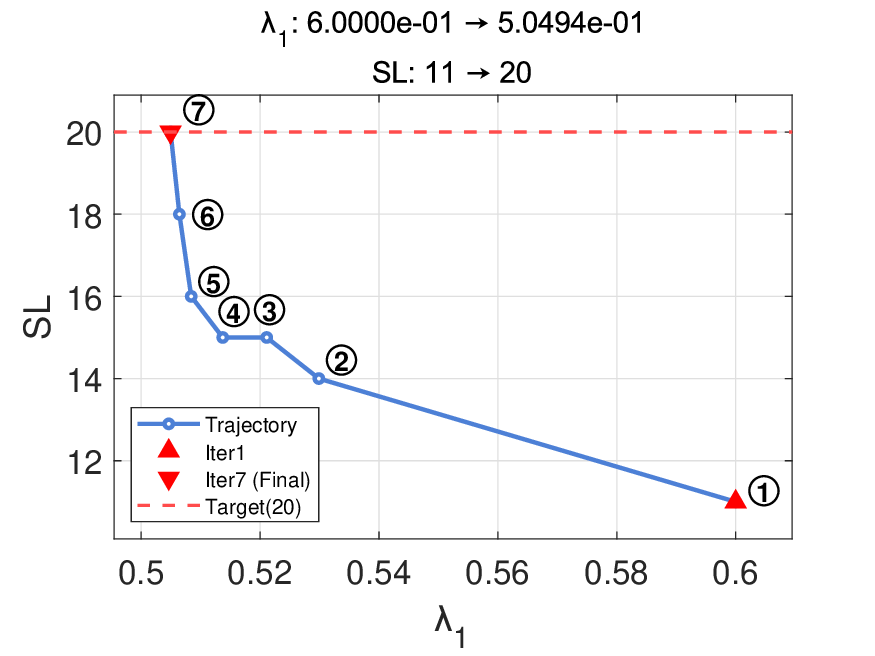} & 
\includegraphics[width=6.5cm]{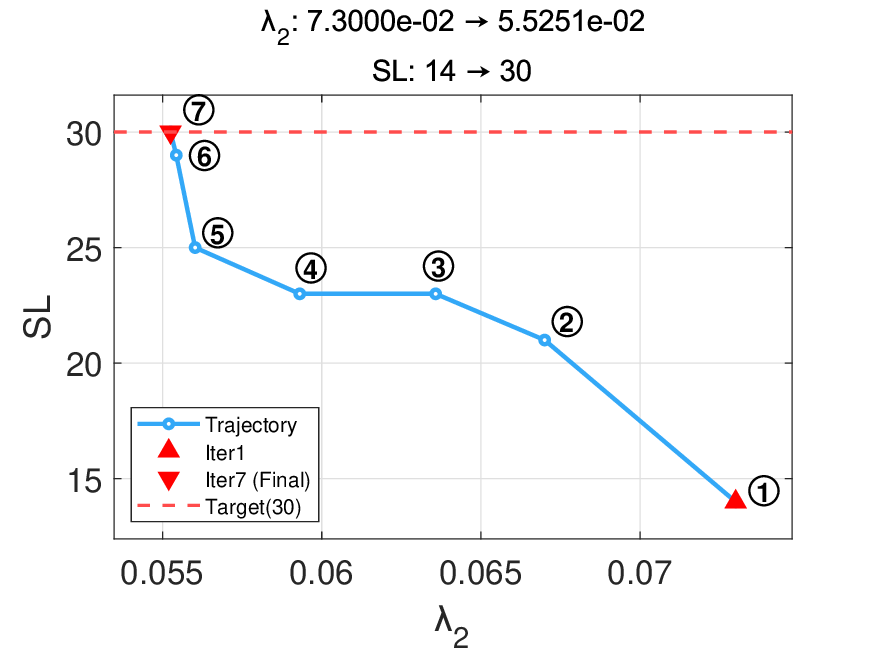} \\
{\scriptsize (a)} & {\scriptsize (b)}
\end{tabular}
\caption{(a) Evolution of $\lambda_1$ and (b) $\lambda_2$ to achieve the TSLs of $[20, 30]$.}
\label{Iter20_30}
\end{figure}

\begin{figure}[h]
\centering
\begin{tabular}{cc}
\includegraphics[width=6.5cm]{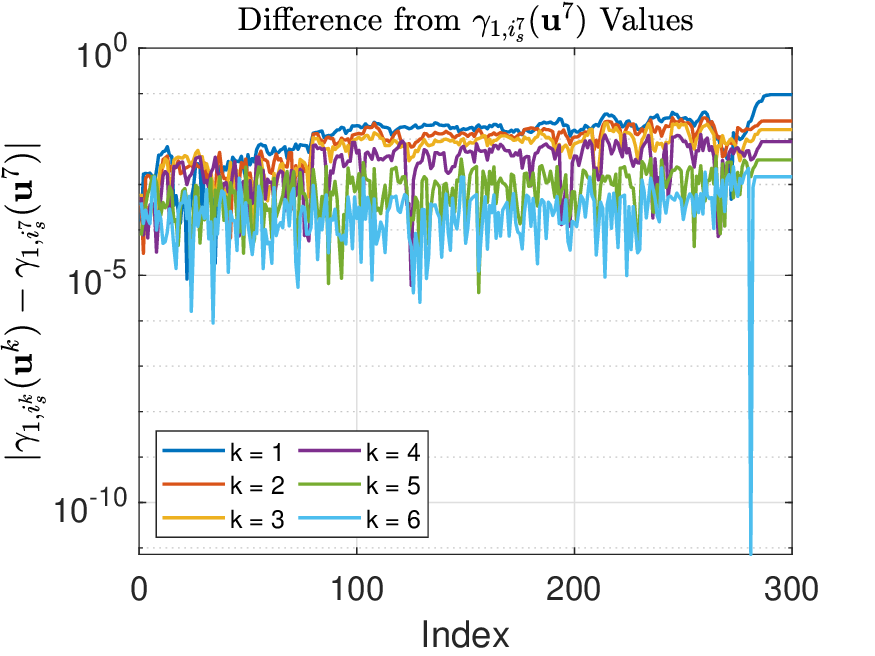} & 
\includegraphics[width=6.5cm]{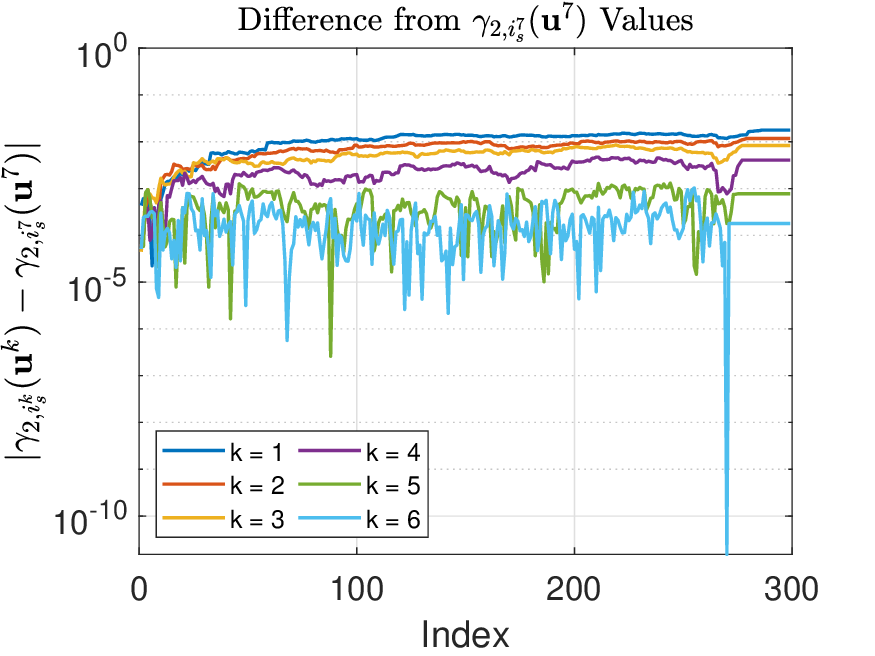} \\
{\scriptsize (a)} & {\scriptsize (b)}
\end{tabular}
\caption{Iterative deviation of the quantities (a) $\gamma_{1,i_s^7}(\mathbf{u}^7)$ and (b) $\gamma_{2,i_s^7}(\mathbf{u}^7)$ from the terminal solution.}
\label{Iterative_deviation}
\end{figure}

\subsubsection*{Support Set Evolution and Convergence Insights}

To further investigate the convergence properties of Algorithm \ref{Sparsity-Guided Multi-Parameter Selection Algorithm}, we analyze the evolution of the support sets. We define the support of a vector $\mathbf{x} \in \mathbb{R}^s$ as $\mathrm{supp}(\mathbf{x}) := \{i \in \mathbb{N}_s : x_i \neq 0\}$. Table \ref{Inclusion_relationship} details the support sets for $\mathbf{u}^k$ and $\mathbf{Du}^k$ across iterations, where bold indices indicate elements newly included in the support set compared to the previous iteration.

\begin{itemize}
    \item \textbf{Monotonic Inclusion:} The support sets exhibit a nested structure, satisfying $\mathrm{supp}(\mathbf{u}^k) \subseteq \mathrm{supp}(\mathbf{u}^{k+1})$ and $\mathrm{supp}(\mathbf{D}\mathbf{u}^k) \subseteq \mathrm{supp}(\mathbf{D}\mathbf{u}^{k+1})$.
    \item \textbf{Synchronized Growth:} The expansion of the support sets corresponds with the increase in SLs from $[11, 14]$ to $[20, 30]$, demonstrating that the algorithm effectively controls the sparsity structure.
    \item \textbf{Coupled Sparsity:} Updates are often correlated; for instance, index 25 is added to both $\mathrm{supp}(\mathbf{u}^k)$ and $\mathrm{supp}(\mathbf{D}\mathbf{u}^k)$ at iteration 5.
\end{itemize}

While this nested inclusion structure is not a universal theoretical guarantee, its presence in our numerical tests provides valuable insight into the algorithm's stability and informs the theoretical analysis of well-behaved scenarios.

\begin{table}[ht]
\centering
\caption{Evolution and inclusion relationship of support sets across iterations}
\label{Inclusion_relationship}
\begin{tabular}{c||c|p{4.5cm}|p{6cm}}
\hline
\textbf{Iter.} $k$ & \textbf{SLs} & \centering $\mathrm{supp}(\mathbf{u}^k)$ & \centering $\mathrm{supp}(\mathbf{D}\mathbf{u}^k)$ \tabularnewline \hline\hline
1 & [11, 14] & 2, 87, 88, 89, 90, 91, 93, 94, 212, 231, 300 & 1, 2, 86, 88, 89, 91, 92, 93, 94, 211, 212, 230, 231, 299 \\ \hline
2 & [14, 21] & 2, \textbf{83}, 87, 88, 89, 90, 91, 93, 94, \textbf{184}, 212, \textbf{224}, 231, 300 & 1, 2, \textbf{82}, \textbf{83}, 86, 88, 89, \textbf{90}, 91, 92, 93, 94, \textbf{183}, \textbf{184}, 211, 212, \textbf{223}, \textbf{224}, 230, 231, 299 \\ \hline
3 & [15, 23] & 2, 83, 87, 88, 89, 90, 91, 93, 94, \textbf{178}, 184, 212, 224, 231, 300 & 1, 2, 82, 83, 86, 88, 89, 90, 91, 92, 93, 94, \textbf{177}, \textbf{178}, 183, 184, 211, 212, 223, 224, 230, 231, 299 \\ \hline
4 & [15, 23] & 2, 83, 87, 88, 89, 90, 91, 93, 94, 178, 184, 212, 224, 231, 300 & 1, 2, 82, 83, 86, 88, 89, 90, 91, 92, 93, 94, 177, 178, 183, 184, 211, 212, 223, 224, 230, 231, 299 \\ \hline
5 & [16, 25] & 2, \textbf{25}, 83, 87, 88, 89, 90, 91, 93, 94, 178, 184, 212, 224, 231, 300 & 1, 2, \textbf{24}, \textbf{25}, 82, 83, 86, 88, 89, 90, 91, 92, 93, 94, 177, 178, 183, 184, 211, 212, 223, 224, 230, 231, 299 \\ \hline
6 & [18, 29] & 2, 25, \textbf{70}, 83, 87, 88, 89, 90, 91, 93, 94, \textbf{173}, 178, 184, 212, 224, 231, 300 & 1, 2, 24, 25, \textbf{69}, \textbf{70}, 82, 83, 86, 88, 89, 90, 91, 92, 93, 94, \textbf{172}, \textbf{173}, 177, 178, 183, 184, 211, 212, 223, 224, 230, 231, 299 \\ \hline
7 & [20, 30] & 2, 25, 70, 83, 87, 88, 89, 90, 91, 93, 94, \textbf{95}, \textbf{96}, 173, 178, 184, 212, 224, 231, 300 & 1, 2, 24, 25, 69, 70, 82, 83, 86, 88, 89, 90, 91, 92, 93, 94, \textbf{96}, 172, 173, 177, 178, 183, 184, 211, 212, 223, 224, 230, 231, 299 \\ \hline
\end{tabular}
\end{table}

\subsection{Sensitivity analysis of hyperparameters}

This subsection evaluates the sensitivity of Algorithm \ref{Sparsity-Guided Multi-Parameter Selection Algorithm} to two critical hyperparameters: the initial regularization vector $\bm{\lambda}^0$ and the stopping threshold $\epsilon$. Utilizing the experimental setup defined in Subsection \ref{Compound sparse denoising}, we maintain the compound sparse denoising model \eqref{CSD}, the observed signal $\mathbf{y}$, and the high-pass filter matrix $\mathbf{H}$ to ensure consistency. The objective is to assess algorithmic robustness and provide practical heuristics for parameter tuning.

\subsubsection*{Experimental Configuration}
We systematically vary $\bm{\lambda}^0$ and $\epsilon$ while fixing the target sparsity levels at $\text{TSLs} = [80, 60]$. Performance is measured via the attained sparsity levels (SLs), the number of iterations ($\mathrm{NUM}$), total computational time, and the mean squared error (MSE). 

\subsubsection*{Impact of Initial Values ($\bm{\lambda}^0$)}
Table \ref{Different_lambda0} reports the performance metrics for different choices of $\bm{\lambda}^0$ with a fixed threshold $\epsilon = 2$. In the table, $\mathrm{ISLs}$ denotes the initial sparsity levels associated with $\bm{\lambda}^0$, serving as a baseline.

\begin{table}[ht]
\centering
\caption{Sensitivity to $\bm{\lambda}^0$ ($\text{TSLs} = [80,60]$, $\epsilon=2$)}
\label{Different_lambda0}
\small
\begin{tabular}{l||c|c|c|c|c}
\hline
$\bm{\lambda}^0$ & $[0.5, 1]$ & $[0.3, 0.4]$ & $[0.28, 0.2]$ & $[0.25, 0.18]$ & $[0.20, 0.17]$ \\ \hline\hline
$\mathrm{ISLs}$ & $[3, 3]$ & $[17, 10]$ & $[33, 27]$ & $[50, 39]$ & $[77, 60]$ \\
$\bm{\lambda}^*$ & $[0.1982, 0.1742]$ & $[0.1970, 0.1753]$ & $[0.1978, 0.1746]$ & $[0.1994, 0.1704]$ & $[0.1992, 0.1700]$ \\
$\mathrm{SLs}$ & $[79, 59]$ & $[80, 60]$ & $[79, 59]$ & $[78, 60]$ & $[78, 60]$ \\
Ratios (\%) & $[26.33, 19.73]$ & $[26.67, 20.07]$ & $[26.33, 19.73]$ & $[26.00, 20.07]$ & $[26.00, 20.07]$ \\
$\mathrm{NUM}$ & 13 & 10 & 9 & 5 & 2 \\
Time (s) & 22 & 17 & 15 & 9 & 5 \\
$\mathrm{MSE}$ ($\times 10^{-2}$) & 1.58 & 1.57 & 1.58 & 1.59 & 1.59 \\ \hline
\end{tabular}
\end{table}

\begin{itemize}
    \item \textbf{Invariance of Solution Quality:} Despite $\bm{\lambda}^0$ spanning a wide range, the converged parameters $\bm{\lambda}^*$ consistently cluster within a narrow interval ($\lambda_1^* \in [0.1970, 0.1994]$ and $\lambda_2^* \in [0.1700, 0.1753]$). Similarly, the MSE and final SLs remain stable, indicating that the algorithm reaches the same optimal region regardless of the starting point.
    \item \textbf{Efficiency Gains:} While the final quality is robust, convergence speed is highly sensitive to the proximity of $\bm{\lambda}^0$ to $\bm{\lambda}^*$. As the initial ISLs approach the TSLs, the iteration count drops from 13 to 2, and runtime decreases by over 75\%, providing clear guidance for efficient parameter initialization.
\end{itemize}

\subsubsection*{Impact of Stopping Threshold ($\epsilon$)}
Table \ref{Different_epsilon} details the algorithm's behavior for different $\epsilon$ values using a fixed $\bm{\lambda}^0 = [0.5, 1.0]$.

\begin{table}[ht]
\centering
\caption{Sensitivity to $\epsilon$ ($\text{TSLs} = [80,60]$, $\bm{\lambda}^0=[0.5, 1.0]$)}
\label{Different_epsilon}
\small
\begin{tabular}{l||c|c|c|c|c}
\hline
$\epsilon$ & 0 & 1 & 2 & 3 (4) & 5 \\ \hline\hline
$\bm{\lambda}^*$ & \begin{tabular}[c]{@{}c@{}}$[0.19821588,$\\ $ \ 0.17418961]$\end{tabular} & \begin{tabular}[c]{@{}c@{}}$[0.19821588,$\\ $\ 0.17418969]$\end{tabular} & \begin{tabular}[c]{@{}c@{}}$[0.19821637,$\\ $\ 0.17419025]$\end{tabular} & \begin{tabular}[c]{@{}c@{}}$[0.19837105,$\\ $\ 0.17431384]$\end{tabular} & \begin{tabular}[c]{@{}c@{}}$[0.19881365,$\\ $\ 0.17460994]$\end{tabular} \\ \hline
$\mathrm{SLs}$ & $[80, 60]$ & $[80, 59]$ & $[79, 59]$ & $[78, 59]$ & $[77, 58]$ \\
Ratios (\%) & $[26.67, 20.07]$ & $[26.67, 19.73]$ & $[26.33, 19.73]$ & $[26.00, 19.73]$ & $[25.67, 19.40]$ \\
$\mathrm{NUM}$ & 23 & 15 & 13 & 9 & 8 \\
Time (s) & 40 & 25 & 22 & 15 & 13 \\
$\mathrm{MSE}$ ($\times 10^{-2}$) & 1.58 & 1.58 & 1.58 & 1.58 & 1.58 \\ \hline
\end{tabular}
\end{table}

\begin{itemize}
    \item \textbf{Robustness of Denoising:} The solution quality (MSE) is unaffected by $\epsilon$ within the range $[0, 5]$, suggesting the core denoising capability is stable relative to the convergence threshold.
    \item \textbf{The Accuracy--Efficiency Trade-off:} Increasing $\epsilon$ significantly accelerates the algorithm (e.g., from 40s to 13s). However, higher thresholds lead to a reduction in sparsity control accuracy. At $\epsilon=0$, the algorithm matches the TSLs exactly ($[80, 60]$), whereas at $\epsilon=5$, the result deviates to $[77, 58]$ due to premature termination.
\end{itemize}

\textbf{Practical Recommendation:} For most applications with $d=2$, a threshold of \textbf{$\epsilon = 2$} is recommended. It provides a balanced compromise between high sparsity control accuracy and computational efficiency. Thresholds above 5 should generally be avoided to prevent excessive deviation from the target sparsity structure.
%%%%%%%%%%%%%%%%%%%

%%%%%%%%%%%%%%%%%%%
\section{Conclusion}\label{Concl}
This paper presents a comprehensive study on sparsity-guided multi-parameter selection in $\ell_1$-regularized models. The proposed framework addresses the critical challenge of selecting multiple regularization parameters to achieve prescribed sparsity levels under different transform matrices. The key contributions are as follows: 

\begin{enumerate}
    \item \textbf{Sparsity--parameter characterization:} We establish a precise relationship between regularization parameters and solution sparsity using tools from convex analysis. This characterization allows independent control of sparsity under each transform matrix, enhancing flexibility in promoting structured sparsity. 

    \item \textbf{Iterative multi-parameter scheme:} We develop an iterative algorithm to simultaneously determine multiple parameters and solutions with target sparsity levels. The scheme includes corrections for invalid updates and fallback mechanisms for undersized parameters, ensuring reliable and consistent performance in practice. 

    \item \textbf{Efficient computation via fixed-point proximity:} To address the computational challenges of jointly computing the regularized solution and two auxiliary vectors at each iteration, we propose a fixed-point proximity algorithm and prove its convergence. This algorithm is a crucial component of the iterative parameter selection scheme. 

    \item \textbf{Extensive numerical validation:} Experiments on signal denoising, compound sparse denoising, and fused SVM classification demonstrate the effectiveness of the proposed multi-parameter selection strategies in producing solutions with the desired sparsity and high approximation accuracy. Dedicated experiments further provide numerical evidence supporting the theoretical assumptions of the iterative scheme and show the algorithm’s robustness to key hyperparameters. 
\end{enumerate}

For future work, we plan to (1) analyze the convergence of the iterative scheme for multiple parameter selection and (2) investigate the dual role of regularization parameters in both alleviating ill-posedness and promoting sparsity, as suggested by our numerical results.

\section*{Appendices}   
\begin{appendices}
\section{The closed-form formulas of proximity operators}
\renewcommand{\thetheorem}{A.\arabic{theorem}}
In this appendix, we provide the closed-form formulas for three proximity operators  utilized in iteration \eqref{Reduce_FPPA_w}. 

We begin with the proximity operator $\mathrm{prox}_{\alpha\sum_{j\in\mathbb{N}_d}\lambda_j\|\cdot\|_{1}\circ\mathbf{I}^{'}_{j}}$ with $\alpha>0$, $\lambda_j>0$, $j\in\mathbb{N}_d$   and $\mathbf{I}^{'}_{j}$, $j\in\mathbb{N}_d$ being defined by \eqref{I-j}.

\begin{proposition}
If $\alpha>0$, $\lambda_j>0$, $j\in\mathbb{N}_d$   and $\mathbf{I}^{'}_{j}$, $j\in\mathbb{N}_d$ are defined by \eqref{I-j}, then the proximity operator $\mathrm{prox}_{\alpha\sum_{j\in\mathbb{N}_d}\lambda_j\|\cdot\|_{1}\circ\mathbf{I}^{'}_{j}}$ at $
\scalebox{0.8}{$\begin{bmatrix}
\mathbf{z}\\
\mathbf{v}\end{bmatrix}$}\in\mathbb{R}^{p_d+n-r}$ with $\mathbf{z}:=[z_{k}:k\in\mathbb{N}_{p_d}]\in\mathbb{R}^{p_d}$ and $\mathbf{v}\in\mathbb{R}^{n-r}$ has the form 
\begin{equation}\label{proximity-operator-sum_lambda_j}
\mathrm{prox}_{\alpha\sum_{j\in\mathbb{N}_d}\lambda_j\|\cdot\|_{1}\circ\mathbf{I}^{'}_{j}}\left(\raisebox{-0.3ex}{\scalebox{0.8}{$\begin{bmatrix}
\mathbf{z}\\
\mathbf{v}\end{bmatrix}$}}\right):=\scalebox{0.8}{$\begin{bmatrix}\bm{\mu}\\
\mathbf{v}\end{bmatrix}$}
\ \mbox{with}\ \bm{\mu}:=[\mu_{k}: k\in\mathbb{N}_{p_d}]\in\mathbb{R}^{p_d},
\end{equation}
where for each $j\in\mathbb{N}_d$ and each  $i\in\mathbb{N}_{m_j}$
\begin{equation}\label{proximity-operator-sum_lambda_j_mu}
\mu_{p_{j-1}+i}=
\left\{\begin{array}{ll}
z_{p_{j-1}+i}-\alpha\lambda_j, \ &\mathrm{if}\ z_{p_{j-1}+i}>\alpha\lambda_j,\\
z_{p_{j-1}+i}+\alpha\lambda_j, \ & \mathrm{if}\ z_{p_{j-1}+i}<-\alpha\lambda_j,\\
 0, \ & \mathrm{if}\ z_{p_{j-1}+i}\in[-\alpha\lambda_j,\alpha\lambda_j].
\end{array}\right.
\end{equation} 
\end{proposition}
\begin{proof}
For each $\mathbf{z}:=[z_{k}:k\in\mathbb{N}_{p_d}]\in\mathbb{R}^{p_d}$ and each $\mathbf{v}\in\mathbb{R}^{n-r}$, we set $\mathrm{prox}_{\alpha\sum_{j\in\mathbb{N}_d}\lambda_j\|\cdot\|_{1}\circ\mathbf{I}^{'}_{j}}\left({\scalebox{0.8}{$\begin{bmatrix}
\mathbf{z}\\
\mathbf{v}\end{bmatrix}$}}\right):={\scalebox{0.8}{$\begin{bmatrix}\bm{\mu}\\
\bm{\nu}\end{bmatrix}$}}$ with $ \bm{\mu}:=[\mu_{k}: k\in\mathbb{N}_{p_d}]\in\mathbb{R}^{p_d}$ and $\bm{\nu}\in\mathbb{R}^{n-r}$. It follows from definition \eqref{proximity operator} of the proximity operator that 
\begin{equation*}\label{alpah_prox_argmin}
{\scalebox{0.8}{$\begin{bmatrix}\bm{\mu}\\
\bm{\nu}\end{bmatrix}$}}=\argmin\left\{\frac{1}{2}\left\|\scalebox{0.8}{$\begin{bmatrix}
\widetilde{\mathbf{z}}\\
\widetilde{\mathbf{v}}\end{bmatrix}$}-\scalebox{0.8}{$\begin{bmatrix}
\mathbf{z}\\
\mathbf{v}\end{bmatrix}$}\right\|_2^2+\alpha \sum_{j\in\mathbb{N}_d}\lambda_j\left\|\mathbf{I}^{'}_{j}\scalebox{0.8}{$\begin{bmatrix}
\widetilde{\mathbf{z}}\\
\widetilde{\mathbf{v}}\end{bmatrix}$}\right\|_1:\scalebox{0.8}{$\begin{bmatrix}
\widetilde{\mathbf{z}}\\
\widetilde{\mathbf{v}}\end{bmatrix}$}\in\mathbb{R}^{p_d+n-r}\right\}   
\end{equation*}
which further leads to 
$$
{\scalebox{0.8}{$\begin{bmatrix}\bm{\mu}\\
\bm{\nu}\end{bmatrix}$}}=\argmin\left\{\sum_{j\in\mathbb{N}_d}\left(\frac{1}{2}\|\widetilde{\mathbf{z}}_j-
\mathbf{z}_j\|_2^2+\alpha\lambda_j\|\widetilde{\mathbf{z}}_j\|_1\right)+\frac{1}{2}\|
\widetilde{\mathbf{v}}-
\mathbf{v}\|_2^2:\scalebox{0.8}{$\begin{bmatrix}
\widetilde{\mathbf{z}}\\
\widetilde{\mathbf{v}}\end{bmatrix}$}\in\mathbb{R}^{p_d+n-r}\right\}.
$$
As a result, we have that 
\begin{equation}\label{mu_j}
\bm{\mu}_j=\argmin \left\{\frac{1}{2}\|\widetilde{\mathbf{z}}_j-
\mathbf{z}_j\|_2^2+\alpha\lambda_j\|\widetilde{\mathbf{z}}_j\|_1:\widetilde{\mathbf{z}}_j\in\mathbb{R}^{m_j}\right\}, \ j\in\mathbb{N}_d,   
\end{equation}
and 
\begin{equation}\label{nu}
\bm{\nu}=\argmin\left\{\frac{1}{2}\|
\widetilde{\mathbf{v}}-
\mathbf{v}\|_2^2:\widetilde{\mathbf{v}}\in\mathbb{R}^{n-r}\right\}. 
\end{equation}
According to Examples 2.3 and 2.4 in \cite{micchelli2011proximity} and noting that $\bm{\mu}_j:=[\mu_{p_{j-1}+i}:i\in\mathbb{N}_{m_j}]$ and $\mathbf{z}_j:=[z_{p_{j-1}+i}:i\in\mathbb{N}_{m_j}]$, $j\in\mathbb{N}_d$, we obtain from equation \eqref{mu_j} that
$$
\mu_{p_{j-1}+i}=\max\{|z_{p_{j-1}+i}|-\alpha\lambda_j,0
\}{\rm sign}(z_{p_{j-1}+i}), \ j\in\mathbb{N}_d,\  i\in\mathbb{N}_{m_j}.
$$
That is, equation \eqref{proximity-operator-sum_lambda_j} holds. Moreover,
equation \eqref{nu} leads directly to $\bm{\nu}=\mathbf{v}$.
\end{proof}

In particular, when $n=r$, equation \eqref{proximity-operator-sum_lambda_j} reduces to the following simplified form for all $\mathbf{z}\in\mathbb{R}^{p_d}$:
\begin{equation}\label{proximity-operator-sum_lambda_j1}
\mathrm{prox}_{\alpha\sum_{j\in\mathbb{N}_d}\lambda_j\|(\cdot)_j\|_{1}}(
\mathbf{z}):=\bm{\mu}
\ \mbox{with}\ \bm{\mu}:=[\mu_{k}: k\in\mathbb{N}_{p_d}]\in\mathbb{R}^{p_d},
\end{equation}
where for each $j\in\mathbb{N}_d$ and each $i\in\mathbb{N}_{m_j}$, each component $\mu_{p_{j-1}+i}$ is defined by \eqref{proximity-operator-sum_lambda_j_mu}.

The next closed-form formula concerns the proximity operator $\mathrm{prox}_{\frac{1}{\beta}\iota_{\mathbb{M}}}$ with $\mathbb{M}:=\mathcal{R}(\mathbf{B})\times\mathbb{R}^{n-r}$ and $\beta>0$. For a matrix $\mathbf{M}$, we denote by $\mathbf{M}^{\dag}$ the pseudoinverse of $\mathbf{M}$.
\begin{proposition}\label{prox-indicator}
If $\mathbb{M}:=\mathcal{R}(\mathbf{B})\times\mathbb{R}^{n-r}$ and $\beta>0$, then the proximity operator $\mathrm{prox}_{\frac{1}{\beta}\iota_{\mathbb{M}}}$ at $
\scalebox{0.8}{$\begin{bmatrix}
\mathbf{z}\\
\mathbf{v}\end{bmatrix}$}\in\mathbb{R}^{p_d+n-r}$ with $\mathbf{z}:=[z_{k}:k\in\mathbb{N}_{p_d}]\in\mathbb{R}^{p_d}$ and $\mathbf{v}\in\mathbb{R}^{n-r}$ has the form 
\begin{equation*}%\label{beta-indicator_prox}
\mathrm{prox}_{\frac{1}{\beta}\iota_{\mathbb{M}}}\left(\scalebox{0.8}{$\begin{bmatrix}
\mathbf{z}\\
\mathbf{v}\end{bmatrix}$}\right)=\scalebox{0.8}{$\begin{bmatrix}\mathbf{B}(\mathbf{B}^{\top}\mathbf{B})^{\dag}\mathbf{B}^{\top}\mathbf{z}\\\mathbf{v}\end{bmatrix}$}.
\end{equation*}
\end{proposition}
\begin{proof}
According to definition \eqref{proximity operator} of the proximity operator, we obtain for each $\mathbf{z}\in\mathbb{R}^{p_d}$ and each $\mathbf{v}\in\mathbb{R}^{n-r}$ that
$\mathrm{prox}_{\frac{1}{\beta}\iota_{\mathbb{M}}}\left({\scalebox{0.8}{$\begin{bmatrix}
\mathbf{z}\\
\mathbf{v}\end{bmatrix}$}}\right)
:={\scalebox{0.8}{$\begin{bmatrix}\bm{\mu}\\
\bm{\nu}\end{bmatrix}$}}$ with $ \bm{\mu}:=[\mu_{k}: k\in\mathbb{N}_{p_d}]\in\mathbb{R}^{p_d}$ and $\bm{\nu}\in\mathbb{R}^{n-r}$, where
$$
{\scalebox{0.8}{$\begin{bmatrix}\bm{\mu}\\
\bm{\nu}\end{bmatrix}$}}=\argmin\left\{\frac{1}{2}\left\|\scalebox{0.8}{$\begin{bmatrix}
\widetilde{\mathbf{z}}\\
\widetilde{\mathbf{v}}\end{bmatrix}$}-\scalebox{0.8}{$\begin{bmatrix}
\mathbf{z}\\
\mathbf{v}\end{bmatrix}$}\right\|_2^2+\frac{1}{\beta}\iota_{\mathbb{M}}\left({\scalebox{0.8}{$\begin{bmatrix}
\widetilde{\mathbf{z}}\\
\widetilde{\mathbf{v}}\end{bmatrix}$}}\right):\scalebox{0.8}{$\begin{bmatrix}
\widetilde{\mathbf{z}}\\
\widetilde{\mathbf{v}}\end{bmatrix}$}\in\mathbb{R}^{p_d+n-r}\right\}.
$$
By noting that $\mathbb{M}:=\mathcal{R}(\mathbf{B})\times\mathbb{R}^{n-r}$, we rewrite the above equation as 
\begin{equation}\label{mu1}
\bm{\mu}=\argmin\left\{\frac{1}{2}\|\widetilde{\mathbf{z}}-
\mathbf{z}\|_2^2:\widetilde{\mathbf{z}}\in\mathcal{R}(\mathbf{B})\right\}
\end{equation}
and 
\begin{equation}\label{nu1}
\bm{\nu}=\argmin\left\{\frac{1}{2}\|
\widetilde{\mathbf{v}}-
\mathbf{v}\|_2^2:
\widetilde{\mathbf{v}}\in\mathbb{R}^{n-r}\right\}.
\end{equation}
Equation \eqref{mu1} shows that $\bm{\mu}$ is the best approximation  to $\mathbf{z}$ from the subspace $\mathcal{R}(\mathbf{B})$. Hence, we get that $\bm{\mu}=\mathbf{B}\mathbf{x}$ with vector $\mathbf{x}\in\mathbb{R}^n$ satisfying
$$
(\mathbf{B}\mathbf{y})^{\top}(\mathbf{z}-\mathbf{B}\mathbf{x})=0,\ \mbox{for all}\ \mathbf{y}\in\mathbb{R}^n.
$$
By rewriting the above equation as 
$$
\mathbf{y}^{\top}\mathbf{B}^{\top}(\mathbf{z}-\mathbf{B}\mathbf{x})=0,\ \mbox{for all}\ \mathbf{y}\in\mathbb{R}^n,
$$
we have that vector $\mathbf{x}$ is a solution of the linear system 
$$
\mathbf{B}^{\top}\mathbf{B}\mathbf{x}=\mathbf{B}^{\top}\mathbf{z}.
$$
By using the pseudoinverse of $\mathbf{B}^{\top}\mathbf{B}$, we represent  $\mathbf{x}$ as
$$
\mathbf{x}=(\mathbf{B}^{\top}\mathbf{B})^{\dag}\mathbf{B}^{\top}\mathbf{z}+\mathbf{x}_0,
$$
where $\mathbf{x}_0\in\mathbb{R}^n$  satisfying $\mathbf{B}^{\top}\mathbf{B}\mathbf{x}_0=0$. Note that $\mathbf{B}^{\top}\mathbf{B}\mathbf{x}_0=0$ if and only if $\mathbf{B}\mathbf{x}_0=0$. Thus, $\bm{\mu}=\mathbf{B}\mathbf{x}=\mathbf{B}(\mathbf{B}^{\top}\mathbf{B})^{\dag}\mathbf{B}^{\top}\mathbf{z}$. Moreover, we can obtain directly from equation \eqref{nu1} that $\bm{\nu}=\mathbf{v}$.
\end{proof}

We now consider the special case $n=r$, where $\mathbb{M}$ coincides with the range space $\mathcal{R}(\mathbf{B})$.  Since $\mathbf{B}$ has full column rank, the matrix $\mathbf{B}^{\top}\mathbf{B}$ is nonsingular. As a direct consequence of Proposition \ref{prox-indicator}, the proximity operator $\mathrm{prox}_{\frac{1}{\beta}\iota_{\mathbb{M}}}$ admits the explicit form
 \begin{equation}\label{indicator_prox}
    \mathrm{prox}_{\frac{1}{\beta}\iota_{\mathbb{M}}}(\mathbf{z})=\mathbf{B}(\mathbf{B}^{\top}\mathbf{B})^{-1}\mathbf{B}^{\top}\mathbf{z}, \ \mbox{for all}\ \mathbf{z}:=[z_{k}:k\in\mathbb{N}_{p_d}]\in\mathbb{R}^{p_d}.
\end{equation}  

Finally, we give closed-form formulas for the proximity operators
of some loss functions, which will be used in numerical experiments. The first loss function is the $\ell_2$-norm composed with a matrix.
\begin{proposition}\label{2-norm}
If $\mathbf{A}$ is an $m\times n$ matrix, $\mathbf{x}\in\mathbb{R}^m$ and $\rho>0$, then the proximity operator $\mathrm{prox}_{\frac{1}{2\rho}\|\mathbf{A}\cdot-\mathbf{x}\|_2^2}$ at $\mathbf{z}\in\mathbb{R}^n$ has the form  
\begin{equation}\label{Au-2-norm}
\mathrm{prox}_{\frac{1}{2\rho}\|\mathbf{A}\cdot-\mathbf{x}\|_2^2}(\mathbf{z})=\left(\rho\mathbf{I}_n+\mathbf{A}^{\top}\mathbf{A}\right)^{-1}\left(\rho\mathbf{z} + \mathbf{A}^{\top}\mathbf{x}\right).    
\end{equation}
\end{proposition}
\begin{proof}
By setting $\mathrm{prox}_{\frac{1}{2\rho}\|\mathbf{A}\cdot-\mathbf{x}\|_2^2}(\mathbf{z}):=\bm{\mu}$, we obtain from definition \eqref{proximity operator} of the proximity operator that 
$$
\bm{\mu}=\argmin\left\{\frac{1}{2}\|\mathbf{u}-\mathbf{z}\|_2^2+\frac{1}{2\rho}\|\mathbf{A}\mathbf{u}-\mathbf{x}\|_2^2:\mathbf{u}\in\mathbb{R}^n\right\},
$$
which together with the Fermat rule leads to 
\begin{equation*}
\bm{\mu}-\mathbf{z}+\frac{1}{\rho}\mathbf{A}^{\top}
(\mathbf{A}\bm{\mu}-\mathbf{x})=\mathbf{0}.
\end{equation*}
By rewriting the above equation as 
$$
\left(\rho\mathbf{I}_n+\mathbf{A}^{\top}\mathbf{A}\right)\bm{\mu}
=\rho\mathbf{z} + \mathbf{A}^{\top}\mathbf{x}
$$
and noting that $\rho\mathbf{I}_n+\mathbf{A}^{\top}\mathbf{A}$ is nonsingular, we obtain that 
$$
\bm{\mu}=\left(\rho\mathbf{I}_n+\mathbf{A}^{\top}\mathbf{A}\right)^{-1}\left(\rho\mathbf{z} + \mathbf{A}^{\top}\mathbf{x}\right),
$$
which completes the proof of this proposition.
\end{proof}

Two special cases of Proposition \ref{2-norm} are employed in our numerical experiments. In Subsection  \ref{NM_block_separable}, the loss function $\bm{\psi}$ is defined by  \eqref{fidelity_term} with $\mathbf{A}$ being an orthogonal matrix. In this case, equation \eqref{Au-2-norm} reduces to
\begin{equation}\label{Au-2-norm-orthogonal}
  \mathrm{prox}_{\frac{1}{\rho}\bm{\psi}}(\mathbf{z}):=\frac{1}{\rho+1}\left(\rho\mathbf{z} + \mathbf{A}^{\top}\mathbf{x}\right), \ \mbox{for all}\ \mathbf{z}\in\mathbb{R}^n.    
\end{equation}
In Subsection \ref{Compound sparse denoising}, the loss function $\bm{\psi}$ is defined by \eqref{fidelity-CSD} and the proximity operator $\mathrm{prox}_{\frac{1}{\rho}\bm{\psi}}$ at $\mathbf{z}\in\mathbb{R}^n$ can be represented by 
\begin{equation}\label{2-norm-2}
\mathrm{prox}_{\frac{1}{\rho}\bm{\psi}}(\mathbf{z}):=(\rho\mathbf{I}_n+\mathbf{H}^{\top}\mathbf{H})^{-1}(\rho\mathbf{z}+\mathbf{H}^{\top}\mathbf{Hy}).
\end{equation}

The loss function in Subsection \ref{Fused SVM} is chosen as $\bm{\phi}(\mathbf{z}):=\sum_{j\in\mathbb{N}_p}\mathrm{max}\{0,1-z_j\}$ for $\mathbf{z}:=[z_j:j\in\mathbb{N}_p]\in\mathbb{R}^p$. The closed-form formula of the proximity operator $\mathrm{prox}_{\frac{1}{\rho}\bm{\phi}}$ with $\rho>0$ has been given in \cite{Li2019a}. Specifically, the proximity operator
$\mathrm{prox}_{\frac{1}{\rho}\bm{\phi}}$ at $\mathbf{z}\in\mathbb{R}^p$ has the form $\mathrm{prox}_{\frac{1}{\rho}\bm{\phi}}(\mathbf{z}):=[\mu_j:j\in\mathbb{N}_p]$, where for all $j\in\mathbb{N}_p$,
\begin{equation}\label{max-function}
\mu_j:=
\left\{\begin{array}{ll}
z_{j}+\frac{1}{\rho}, \ &\mathrm{if}\ z_{j}<1-\frac{1}{\rho},\\
z_{j}, \ & \mathrm{if}\ z_{j}>1,\\
 1, \ & \mathrm{if}\ z_j\in[1-\frac{1}{\rho}, 1].
\end{array}\right.   
\end{equation}
\end{appendices}

% Authors must disclose all relationships or interests that 
% could have direct or potential influence or impart bias on 
% the work: 
%
\section*{Conflict of interest}
The authors declare that they have no conflict of interest.

% BibTeX users please use one of
%\bibliographystyle{spbasic}      % basic style, author-year citations
%\bibliographystyle{spmpsci}      % mathematics and physical sciences
%\bibliographystyle{spphys}       % APS-like style for physics
%\bibliography{}   % name your BibTeX data base

% Non-BibTeX users please use

\end{document}